\documentclass[a4paper,10pt, leqno]{amsart}
\usepackage{amsmath,amsfonts,amsthm,amssymb,amscd}

\date{September 3, 2024}

\usepackage{graphicx}
\usepackage{amssymb}
\usepackage[active]{srcltx}
\numberwithin{equation}{section}
\newtheorem{Theorem}{Theorem}[section]
\newtheorem{Thm}[Theorem]{Theorem}
\newtheorem{Fact}[Theorem]{Fact}
\newtheorem{Corollary}[Theorem]{Corollary}

\newtheorem{Lemma}[Theorem]{Lemma}
\newtheorem{Proposition}[Theorem]{Proposition}
\newtheorem{Prop}[Theorem]{Proposition}
\theoremstyle{remark}
\newtheorem{Remark}[Theorem]{Remark}
\newtheorem{Rmk}[Theorem]{Remark}
\theoremstyle{definition}
\newtheorem{Definition}[Theorem]{Definition}

\newtheorem{Example}[Theorem]{Example}
\newtheorem{Exa}[Theorem]{Example}
\newtheorem*{acknowledgements}{Acknowledgements}
\newcommand{\inner}[2]{\left\langle{#1},{#2}\right\rangle_L}
\newcommand{\Inner}[2]{\left\langle{#1},{#2}\right\rangle}

\newcommand{\R}{{\mathbb R}}

\newcommand{\C}{{\mathbb C}}
\newcommand{\E}{\mathbb{E}}
\newcommand{\Lo}{\mathbb{L}}
\newcommand{\mc}[1]{{\mathcal #1}}
\newcommand{\mb}[1]{{\mathbf #1}}
\newcommand{\pmt}[1]{{\begin{pmatrix} #1  \end{pmatrix}}}

\renewcommand{\phi}{\varphi}
\renewcommand{\epsilon}{\varepsilon}
\newcommand{\op}[1]{{\operatorname{ #1}}}

\renewcommand{\phi}{\varphi}

\usepackage[active]{srcltx}

\title{
Cuspidal edges and generalized cuspidal edges
in the Lorentz-Minkowski $3$-space
}

\author{T.~Fukui}
\address[Toshizumi Fukui]{
Department of Mathematics,
Saitama University,
255, Shimo-Okubo,
Saitama, 338-8570,
Japan
}
\email{tfukui@rimath.saitama-u.ac.jp}

\author{R.~Kinoshita}%{Ryosuke Knosita}
\address[Ryosuke Kinosita]{
   Department of Mathematical and Computing Sciences,
   Tokyo Institute of Technology,
   2-12-1-W8-34, O-okayama, Meguro-ku,
   Tokyo 152-8552, Japan.
}
\email{kinosita.r.math@gmail.com}
\author{D.~Pei}
\address[Donghe Pei]{
School of Mathematics and Statics,
Northeast Normal University,
Changchun 130024 Jilin
China
}
\email{peidh340@nenu.edu.cn}

\author{M.~Umehara}%{Masaaki Umehara}
\address[Masaaki Umehara]{%
   Department of Mathematical and Computing Sciences,
   Tokyo Institute of Technology,
   2-12-1-W8-34, O-okayama, Meguro-ku,
   Tokyo 152-8552, Japan.
}
\email{umehara@is.titech.ac.jp}

\author{H.~Yu}
\address[Haiou Yu]{
School of Mathematics,
Jilin University of Finance and Economics,
Changchun 130117 Jilin,
China
}
\email{yuhaiou@jlufe.edu.cn}

\thanks{The first and the fourth
authors were supported 
by Grant-in-Aid for Scientific Research 
(C) 19K03486 and
(B) 21H00981, respectively, 
from Japan Society for the Promotion of Science.}

\begin{document}

\maketitle

\begin{abstract}
It is well-known that every cuspidal edge in the 
Euclidean space $\E^3$ cannot have a bounded mean curvature function. 
On the other hand, in the Lorentz-Minkowski space $\Lo^3$, 
zero mean curvature surfaces admit cuspidal edges. 
One natural question is to ask when a cuspidal edge has 
bounded mean curvature in $\Lo^3$.
We show that such a phenomenon  
occurs only when the image of the singular set is 
a light-like curve in $\Lo^3$.
Moreover, we also investigate the behavior of principal 
curvatures in this case as well as other possible cases.  
In this paper, almost all calculations are 
given for generalized cuspidal edges as well as for cuspidal edges.
We define the \lq\lq order" 
at each generalized cuspidal edge singular point
is introduced. 
As nice classes of zero-mean curvature surfaces in $\Lo^3$,
``maxfaces" and ``minfaces" are known, and
generalized cuspidal edge singular points
on maxfaces and minfaces are of order $4$.
One of the important results is that
the generalized cuspidal 
edges of order $4$ exhibit a quite similar behaviors as those on
maxfaces and minfaces.
\end{abstract}

\setcounter{tocdepth}{1}\tableofcontents

\section*{Introduction}
We denote the Euclidean $3$-space by 
$(\E^3;x,y,z)$ and the Lorentz-Minkowski $3$-space by 
$(\Lo^3;x,y,z)$, 
where the signature of $\Lo^3$ is taken as $(++-)$.
Let $f:U\to \Lo^3$ be a regular surface (i.e. an immersion).
A point on $U$ is said to be {\it space-like} (resp. {\it time-like}) 
if the pull-back of the
canonical Lorentzian metric of $\Lo^3$ 
(i.e. the  first fundamental form of $f$) 
induces a Riemannian metric (resp. a Lorentzian metric)
at the point.
In this paper, $\R^3$ denotes $\Lo^3$ or $\E^3$ ignoring the
canonical metrics.
Fix $\delta>0$ and an open interval $I$ on $\R$.

\begin{Definition}
Let $U$ be a domain of $\R^2$
and
$
f:U\to \R^3
$
a $C^\infty$-map. 
A  point $p\in U$ is called a {\it cuspidal edge singular point} of
$f$ if there exist local 
diffeomorphisms $\phi$ and $\Psi$
on $\R^2$ and $\R^3$ such that 
$\phi(p)=(0,0)$, $\Psi\circ f(p)=(0,0,0)$
and $\Psi\circ f\circ \phi^{-1}(s,t)=(s,t^2,t^3)$.
\end{Definition}

We fix  a regular curve
$
\Gamma:I\ni s \mapsto f(s,0)\in \R^3
$.

\begin{Definition}\label{def:217}
A $C^\infty$-map
$
f:I\times (-\delta,\delta)\to \R^3
$
is called a {\it cuspidal edge} 
along $\Gamma$
if 
$(s,0)$ is a cuspidal edge singular point of $f$
for each $s\in I$.
\end{Definition}

We next define \lq\lq generalized cuspidal edges":

\begin{Definition}\label{def:GE}
Fix $\delta>0$ and an open interval $I$ on $\R$.
A $C^\infty$-map
$$
f:I\times (-\delta,\delta)\ni (s,t)\mapsto f(s,t)\in \R^3
$$
is called a {\it generalized cuspidal edge} 
along the curve
$
\Gamma:I\ni s \mapsto f(s,0)\in \R^3
$ 
if $f(s,t)$ satisfies the following properties:
\begin{enumerate}
\item[(a)] $f_t(s,0)$ vanishes identically for each $s\in I$, and
\item[(b)] the vectors $f_{tt}(s,0)$ and
$\Gamma'(s)(=f_s(s,0))$ are linearly independent for each $s\in I$.
\end{enumerate}
Moreover, each point $(s,0)$ is called
a {\it generalized cuspidal edge singular point}.
\end{Definition}

Cuspidal edges are typical examples of generalized cuspidal edges 
(cf. Proposition \ref{prop:1182}).
For a given generalized cuspidal edge $f$, we set
\begin{equation}\label{eq:Domf2}
\mc U_f:=I\times (-\delta,\delta),\qquad 
\Sigma_f:=\{(s,0)\in U\,;\, s\in I\},
\end{equation}
which are the domain of definition  and
the singular set of $f$, respectively.
Then, 
the set $\mc U_f\setminus \Sigma_f$
of regular points of $f$ has two connected components. 
The purpose of this paper is
to investigate the behavior of the Gaussian curvature,
the mean curvature, the principal curvatures and umbilical points of $f$.
We are particularly interested in the behavior of 
mean curvature and umbilical points: 
\begin{itemize}
\item 
Although the mean curvature functions $H^E$ of cuspidal edges
in $\E^3$ diverge along their singular sets,
there are cuspidal edges in $\Lo^3$ whose
mean curvature functions are bounded. 
Moreover, it is also known that
cuspidal edge singular points 
appear as one of the most general 
singular points on zero-mean curvature (i.e. ZMC) 
surfaces in $\Lo^3$ (\cite{A, FSUY, UY}).
Compared to the case in $\E^3$,
it would be interesting to investigate when the 
mean curvature function $H^L$ around an arbitrarily given
cuspidal edge singular point in $\Lo^3$ becomes bounded.
\item
A regular point $p\in \mc U_f$ of $f$
is called an {\it umbilical point} (resp.
 {\it a quasi-umbilical point})
if it is a space-like or time-like point of $f$
at which the two principal curvatures coincide
and the Weingarten matrix (cf. \eqref{eq:Wf})
 is diagonal (resp. is non-diagonal).
On a given surface, 
the umbilical points with $\R^3=\Lo^3$ 
appear at a different location 
than the umbilical points with $\R^3=\E^3$, in general.
Moreover, quasi-umbilical points never appear on
regular surfaces in $\E^3$ and space-like regular surfaces in $\Lo^3$,
but often appear on time-like regular surfaces in $\Lo^3$.

In $\E^3$, umbilical points never accumulate at 
any cuspidal edge singular point (cf. \cite{F,MSUY,T}).
Also in $\Lo^3$, the same conclusion holds
as long as a cuspidal edge singular point $p$ is
time-like or space-like (see Definition~\ref{def:736} and
Proposition \ref{thm:4087} in the appendix).
On the other hand, a light-like cuspidal edge singular point 
(cf. Definition~\ref{def:causality712}) could be an accumulation 
point of umbilical points, and it is worth investigating this
possibility.
\end{itemize}
Therefore, the goal of this paper is to give answers to the above 
questions  as much as possible.
Honda, Izumiya, Saji and Teramoto \cite{HIST} gave normal forms of germs of  
cuspidal edges in $\Lo^3$ and investigated cuspidal edges at light-like points
from a different point of view than ours.

Fix  a generalized cuspidal edge $f:\mc U_f\to \R^3$ along $\Gamma$.
At each singular point $p:=(s,0)\in \Sigma_f$,
we consider the sign
\begin{equation}\label{eq:Ep258}
\sigma^C(p):=\op{sgn}(d^C(s))\in \{-1,0,1\}
\end{equation}
of the value
\begin{equation}\label{eq:Dp262}
d^C(s):=\det(\Gamma'(s),f_{tt}(s,0),\Gamma''(s)).
\end{equation}
By definition, the sign
$\sigma^C(p)$ at the singular point $p$ is common in $\E^3$ and $\Lo^3$.
If $\sigma^C(p)$ is positive (or negative, zero),
we say that the germ of the generalized cuspidal edge $f$ at $p$ is
{\it right-handed} (or {\it left-handed}, {\it neutral}).
The following three assertions hold (see Propositions
\ref{prop:GE} and \ref{prop:GL}):
\begin{enumerate}
\item $\sigma^C(p)$ vanishes if and only if
the osculating plane of $\Gamma$ at $p$
coincides with the limiting tangent plane 
(which is the plane in $\R^3$ passing through $f(p)$ obtained as the limit
of the tangent planes  at its regular points)
of $f$ at $p$.
\item  
The sign of $\sigma^C(p)$ coincides with that of
the Euclidean limiting normal curvature $\kappa^E_\nu(s)$ if we think $\R^3$
as $\E^3$.
\item  
If the normal vector of $f$ at $p$
points in a space-like or time-like direction,
then the 
sign of $\sigma^C(p)$ also coincides with that of
the Lorentzian limiting normal curvature $\kappa^L_\nu(s)$
if we think $\R^3$
as $\Lo^3$.
\end{enumerate}
Regarding the above facts, we give the following:

\begin{Definition}
A generalized cuspidal edge singular point $p$ of 
$f:\mc U_f\to \R^3$
is said to be {\it  generic} if
$\sigma^C(p)$ does not vanish.
\end{Definition}

We remark that the invariant $\sigma^C$ is generalized for an 
invariant of cuspidal edges in a Riemannian 3-manifold $(M^3,g)$
(cf.~\cite{SUY3}, where $\sigma^C$ is  denoted by $\sigma^C_g$).

A regular curve $\Gamma:I\to \Lo^3$ defined on an interval $I$
is called of {\it type $S$} (resp. {\it type $T$})
if the velocity vector field $\Gamma'(s)$ 
points in a space-like (resp. time-like) direction
for each $s\in I$.
If $\Gamma$ is neither type $S$ nor type $T$,
then there exists $s\in I$
such that $\Gamma'(s)$
points in a light-like direction.
We are interested in the special case that $\Gamma'(s)$ 
is a light-like direction for any $s\in I$.
Such a $\Gamma$ is called a {\it regular curve of type $L$}.

We next define the {\it order} of a singular point 
$p\,(=(s,0)\in \mc \Sigma_f)$ of a given generalized cuspidal edge 
$f:\mc U_f\to \Lo^3$,
denoted by $i_p$  (cf. Definition \ref{def:R}):
We denote by $\Delta^L(s,t)$
the determinant of the symmetric matrix associated
with the first fundamental form of $f$ (cf. \eqref{eq:Delta}).
Roughly speaking, the order $i_p$ is the minimal number $k$ 
such that $\partial^k \Delta^L(s,0)/\partial t^k$
does not vanish (cf. Definition \ref{def:R}).
The order is a concept that only makes sense for $f$ lying in $\Lo^3$
(when $f$ lying in $\E^3$, $i_p$ can be also defined
in the same manner but is always equal to $2$).
If the order $i_p$ ($p\in \Sigma_f$) is an even constant, 
then all regular points on a 
sufficiently small neighborhood $V_p$ of $p$
have the same causal type, in particular, they are space-like or
time-like on $V_p$.
On the other hand, if $i_p$ is an odd constant, then any regular points on a sufficiently small
neighborhood of $p$  is space-like on one side of the singular
curve $\Sigma_f$ and time-like on the opposite side.
As a special case, $i_p=\infty$ can occur (see \eqref{ex:IpInfty}).
The following result describes the generic behavior of
generalized cuspidal edges in $\Lo^3$ 
(as a consequence, most cuspidal edges in $\Lo^3$
have unbounded mean curvature functions):

\medskip
\noindent
{\bf Theorem A.}
{\it Let $f:\mc U_f\to \Lo^3$ be a generalized cuspidal edge along a
regular curve $\Gamma$ in $\Lo^3$.
If $\Gamma$ is of type $S$ or of type $T$,
then, for each singular point $p:=(s,0)\in \Sigma_f$,
there exists a sufficiently small neighborhood $V_p(\subset \mc U_f)$  of $p$
such that the following assertions hold $($the first three assertions are
about the causality of $f$ which is defined in Definition~\ref{def:causality712}$)$:
\begin{enumerate}
\item \label{item:A1}
 $f$ is space-like on $V_p \setminus \Sigma_f$
if $\Gamma$ is of type $S$ and the $\Lo^3$-cuspidal direction vector
$\mb D^L_f(s)$ is space-like $($cf. Definition \ref{def:CD617}$)$.
\item \label{item:A2}
$f$ is time-like on $V_p \setminus \Sigma_f$
if 
\begin{itemize}
\item  $\Gamma$ is of type $T$, or 
\item $\Gamma$ is of type $S$ and $\mb D^L_f(s)$ is a time-like vector.
\end{itemize}
\item \label{item:A3}
If $\Gamma$ is of type $S$ and $\mb D^L_f(s)$ is a light-like vector,
then the causal type of $f$ is different between
both sides of $V_p\setminus \Sigma_f$.
\item  \label{item:A4}
If $\mb D^L_f(s)$ is not a light-like vector,
then the limiting normal curvature
$\kappa^L_\nu(s)$ is defined $($cf. \eqref{eq:knSL}$)$,
and the sign of $\kappa^L_\nu(s)$ coincides with $\sigma^C(p)$.
\item \label{item:A5}
The order $i_p$
is equal to $2$ 
if  $\mb D^L_f(s)$ is not  a light-like vector at $p$.
On the other hand,
if  $\mb D^L_f(s)$ points in 
a light-like direction, 
then $i_p \ge 3$ holds.
Moreover, the equality holds
if and only if $p$ is a 
cuspidal edge singular point.
\item \label{item:A6}
If $p$ is a cuspidal edge singular point, then
the mean curvature $H^L$  of $f$ is unbounded on $V_p$.
In particular, if $p$ is an accumulation point of 
the set of cuspidal edge singular points, then $H^L$ is also
unbounded at $p$.
\item  \label{item:A7}
If $p$ is a generic 
$($i.e. $\sigma_C(p)\ne 0)$ cuspidal edge
 singular point satisfying $i_p=2$ 
$($resp. $i_p=3)$,
then the Gaussian curvature  $K^L$ of $f$
is unbounded and takes different signs $($resp. the same sign$)$
on each side $($resp. both sides$)$ of $\Sigma_f$ in $V_p$. 
\item \label{item:A8}
If $p$ is a cuspidal edge singular point
then  the two principal curvatures of $f$ 
are both real-valued on $V_p$.
Moreover, if $p$ is generic $($i.e. $\sigma_C(p)\ne 0)$,
then one of them is unbounded on $V_p$.
In this setting, the other one is bounded on $V_p$ unless $i_p=3$.
\item \label{item:A9}
The umbilical or quasi-umbilical points 
cannot accumulate at any cuspidal edge singular point.
\end{enumerate}
}

\medskip
The corresponding assertions for cuspidal edges in $\E^3$ are 
known (cf. \cite{MSUY,F,T}), which are summarized in Fact \ref{cor:E}.
In particular, for an arbitrarily given cuspidal edge in $\E^3$,
its mean curvature function $H^E$ is always unbounded,
and its Gaussian curvature function $K^E$
is bounded if and only if its limiting normal 
curvature $\kappa^E_\nu$ vanishes identically on $\Sigma_f$. 
So, Theorem~A contains an analogue of these facts
in $\Lo^3$. The strategy of the proof of Theorem~A is an improvement of
the corresponding result in $\E^3$ given by the first author \cite{F}. 
In fact, we need more case separations and computations of
higher-order derivatives. As a consequence of Theorem~A, we have the following:

\medskip
\noindent
{\bf Corollary B.}
{\it Let $f:\mc U_f\to \Lo^3$ be a cuspidal edge along a
regular curve $\Gamma$ in $\Lo^3$.
If the mean curvature $H^L$ of $f$ is bounded on $\mc U_f\setminus \Sigma_f$, 
then $\Gamma$ is of type $L$.}

\medskip
By this assertion, as long as $\Sigma_f$ consists of
cuspidal edge singular points,
the mean curvature $H^L$ of $f$ can be  bounded
only when $\Gamma$ is of type $L$. 
So we consider the case that
$\Gamma$ is of type $L$: 

\begin{Definition}
A $C^\infty$-map $f:\mc U_f\to \Lo^3$ 
is called a {\it light-like generalized cuspidal edge of general type}
along $\Gamma$ (see Definition \ref{def:Lcg2887}) if 
it is a generalized cuspidal edge such that
\begin{itemize}
\item $\Gamma$ is of type L,
\item the order $i_p$ of $f$ at each point $p\in \Sigma_f$
is equal to $2$, and 
\item $\Gamma''(s)$ never vanishes for each $(s,0)\in \Sigma_f$.
\end{itemize}
\end{Definition}

If $\Gamma$ is of type $L$, the possibility of
the normal direction of $f$ is either space-like or light-like.
More precisely, the following assertion holds:

\medskip
\noindent
{\bf Proposition C.}
{\it Let $\Gamma:I\to \Lo^3$ be a regular curve such that
$\Gamma'(s_0)$ $(s_0\in I)$
is a light-like vector.
If $f$ is a generalized cuspidal edge along $\Gamma$
in $\Lo^3$. Then the following two assertions
are equivalent:
\begin{enumerate}
\item  \label{item:C1}
The normal direction of $f$ at $p:=(s_0,0)$ points in a
space-like direction $($resp. a light-like direction$)$,
\item \label{item:C2}
the order $i_p$ of $p:=(s_0,0)$ is
equal to $2$ $($resp. greater than or equal to $4)$.
\end{enumerate}
In particular, $i_p>2$ implies $i_p\ge 4$.
Moreover, if $\Gamma$ is of type $L$ satisfying $\Gamma''(s_0)\ne \mb 0$,
then $(1)$ and $(2)$ hold if and only if
$\sigma^C(p)\ne 0$ $($resp. $\sigma^C(p)=0)$.}

\medskip
\noindent
In particular, the normal vectors of
light-like generalized cuspidal edges of general type
always point in space-like directions. 
Regarding this, we show the following:

\medskip
\noindent
{\bf Theorem~D.} {\it Let $f$ be a 
light-like generalized
cuspidal edge of general type
along $\Gamma$ in $\Lo^3$. 
Then, for each singular point $p:=(s,0)\in \Sigma_f$,
there exists a sufficiently small neighborhood 
$V_p(\subset \mc U_f)$  of $p$ such that the following assertions hold:
\begin{enumerate}
\item \label{item:D2}
The sign $\sigma^C(p)$ does not vanish for each $p\in \Sigma_f$.
\item \label{item:D1}
$f$ is time-like on $V_p\setminus \Sigma_f$.
\item \label{item:D3}
The mean curvature $H^L$ 
is bounded on $V_p\setminus \Sigma_f$. 
\item \label{item:D4}
If $p$ is a cuspidal edge singular point, then
$K^L$ is unbounded and takes different signs on each side of $\Sigma_f$ 
on $V_p$.
\item \label{item:D5}
If $p$ is a cuspidal edge singular point, then 
the two principal curvatures of $f$ are both unbounded on $V_p$, amd 
they are real-valued on one side of $\Sigma_f$ 
and are not real-valued on the other side on $V_p$
$($that is, the singular set $\Sigma_f$ 
behaves rather like the locus of quasi-umbilical points$)$. 
\item \label{item:D6}
Umbilical points and quasi-umbilical points of $f$ never accumulate 
at any cuspidal edge singular points of $f$.
\end{enumerate}
}

\medskip
In the setting of Theorem~D, 
the authors do not know of 
any light-like generalized
cuspidal edge of general type
along $\Gamma$ in $\Lo^3$ 
for which $H^L$ vanishes identically
(if it happens, by Proposition~\ref{prop:O4}, $f$ cannot be a minface).

The case that $\Gamma'(s)$ points in a light-like direction only at 
a point $s=s_0$ is also discussed in Section~4.
We note that \ref{item:A9} of Theorem~A and \ref{item:D6} of Theorem~D
are special cases of the following statement:

\medskip
\noindent
{\bf Proposition E.}
{\it Let $f$ be a  cuspidal edge along a regular 
curve $\Gamma$ in $\Lo^3$. If a given cuspidal edge singular point
$p:=(s,0)\in \Sigma_f$ is of order less than $4$ $($i.e. $i_p\le 3)$, 
then umbilical points of $f$ never accumulate at $p$.
Moreover, quasi-umbilical points of $f$ also cannot
accumulate at $p$ unless $p$ is a light-like point
$($if $p$ is an isolated light-like point,
quasi-umbilical points can accumulate at $p$,
see $(3)$ of Proposition~\ref{prop:LG}$)$.
}

\medskip
For a light-like cuspidal edge singular point $p$, 
the authors know of no example in which 
umbilical points accumulate at $p$.

\begin{Definition}\label{def:CC613}
Let $f$ be a generalized cuspidal edge along a regular curve $\Gamma$
of type $L$.  If $i_p=4$ holds for each $p\in \Sigma_f$,
we call $f$ a {\it generalized cuspidal edge of order four}.
\end{Definition}

If $f$ be a generalized cuspidal edge along a  regular curve $\Gamma$ 
of type $L$ in $\Lo^3$, then the Lorentzian singular curvature function $\kappa^L_s$
along $\Gamma$ cannot defined. In this case,  we use
the Euclidean singular curvature function $\kappa^E_s$
along $\Gamma$, instead. We focus on the property that
the sign of $\kappa^E_s$ is an identifier of whether the cuspidal edge $f$
looks convex or concave in human's eyes (cf. \cite[Section~5]{SUY2}):

\begin{figure}[hbt]
\begin{center} 
\includegraphics[height=3.6cm]{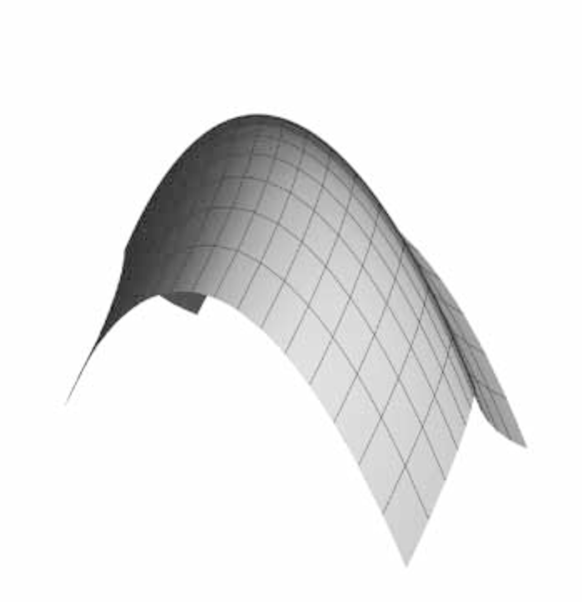}\qquad \qquad
\includegraphics[height=3.6cm]{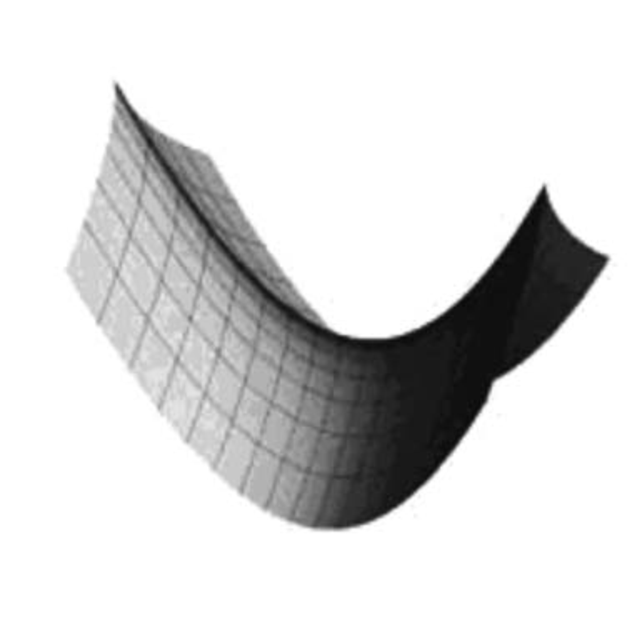}
\end{center}
\caption{
Cuspidal edges in $\E^3$ with positive (left) and 
negative (right) singular curvatures}\label{ex:C2}
\end{figure}

\begin{Definition}\label{def:CC613b}
Let $f$ be a generalized cuspidal edge of order four along a 
regular curve $\Gamma$ of type $L$. 
Then $f$
is said to be  of {\em convex type} $($resp. {\em concave type}$)$
at  a singular point $p=(s,0)\in \Sigma_f$
if $\kappa^E_s(s)$ is positive (resp. negative).
\end{Definition}

If $f$  is a  generalized cuspidal edge  of order four which is of 
convex $($resp.~concave$)$ type, then, for any time-like vector $\mb v$,
the image of the orthogonal projection of $f$ into the plane 
$\Pi^L_{\mb v}$ which is perpendicular to $\mb v$
is locally convex $($resp. locally concave$)$, see 
\cite[Proposition 2.19 and also Remark 2.20]{HFU} for details.

Let
\begin{equation}\label{eq:pi554}
\pi:\Lo^3\ni (x,y,z)\mapsto (x,y)\in \R^2
\end{equation} 
be the canonical projection,
and consider the curve $\gamma:=\pi\circ \Gamma$
in the space-like $xy$-plane. We denote by $P_s$ the plane in $\Lo^3$ 
passing through $\gamma(s)$, which is spanned 
by $\mb v_3:=(0,0,1)$ and $\mb n(s)$,
where $\mb n(s)$ is the normal vector of the curve $\gamma$ at $\gamma(s)$.
Then $P_s$ passes through $\Gamma(s)$ as well as $\gamma(s)$ 
(cf. Figure~\ref{fig:PS}). In this setting, we may
assume that $s$ is the arc-length parameter of $\gamma$. 
Let $\kappa(s)$ be the curvature  of $\gamma$
at $s\in I$, and $\mb c_s$ the section of the image of
$f$ by the plane $P_s$.
We denote by  $\mu(s,t)$ the function $($cf. Definition~\ref{def:mu}$)$ 
associated with the cuspidal curvature of $\mb c_s$,
which can be written as
\begin{equation}\label{eq:459}
\mu(s,t)=\mu_0(s)+\mu_1(s)t+\phi(s,t)t^2,
\end{equation}
where $\phi(s,t)$ is a certain $C^\infty$-function.
Then, we prove the following:

\medskip
\noindent
{\bf Proposition F.}
{\it Let $f$ be a cuspidal edge  
in $\Lo^3$ and $p$ its singular point satisfying $i_p\ge 4$.
Then $i_p>4$ holds if and only if $f$ is of concave type  at $p$
and 
\begin{equation}\label{eq:km476}
|\kappa|=\mu_0^2
\end{equation}
holds at the point $p$}.

\begin{figure}[h!]
\begin{center}
\includegraphics[height=6.5cm]{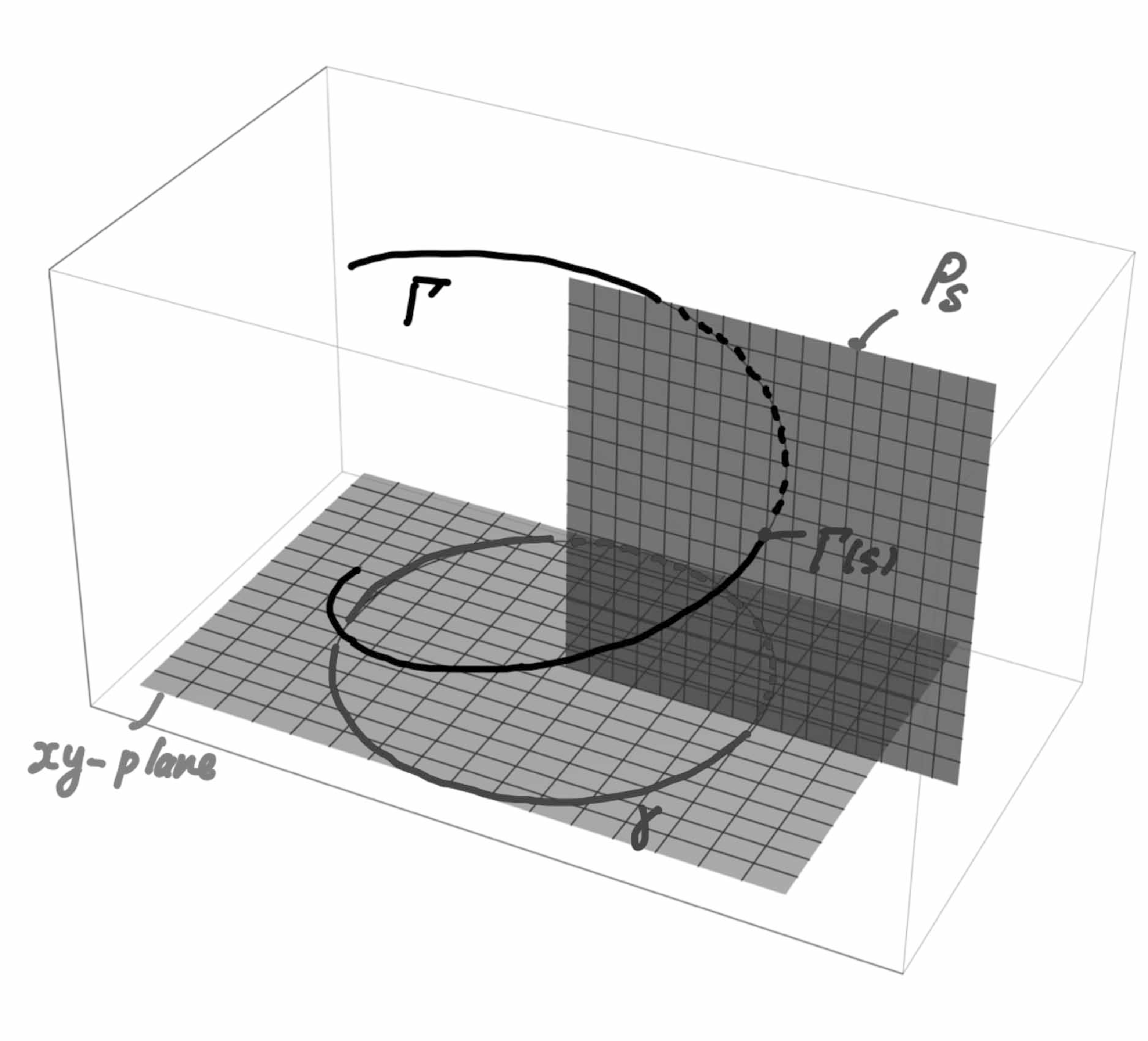}
\end{center}
\caption{The plane $P_s$ and the 
curves $\Gamma$ (a helix) and $\gamma$ (a circle)}
\label{fig:PS}
\end{figure}

\medskip
As nice classes of zero-mean curvature surfaces in $\Lo^3$,
``maxfaces" and ``minfaces" are known, and if they give 
generalized  cuspidal edges 
(in fact, cuspidal edges and cuspidal cross caps 
are typical examples of generalized 
cuspidal edges), 
then they are always of order four
 (cf. Proposition~\ref{prop:O4}).
We remark that  ``maxfaces" and ``minfaces" themselves may admit singular 
points which are not appeared on generalized cuspidal edges, 
like as swallowtails and cone-like singular points.
The following theorem summarizes the properties of the generalized 
cuspidal edges of order four, which is the deepest result in this paper:

\medskip
\noindent
{\bf Theorem~G.}
{\it Let $f$ be a generalized cuspidal edge of order four along 
a regular curve $\Gamma$ of type $L$. 
Suppose that  the regular curve
$\gamma(s):=\pi\circ \Gamma(s)$ $(s\in I)$ in the $xy$-plane
is parametrized by the arc-length 
and $\kappa(s)$ denotes the curvature function of $\gamma(s)$ as a plane curve.
Then, for each singular point $p:=(s_0,0)\in \Sigma_f$,
there exists a sufficiently small neighborhood $V_p(\subset \mc U_f)$  of $p$
such that the following assertions hold:
\begin{enumerate}
\item[(a)] %\label{item:Ga}
 The causal type of $f$  
is the same on both sides of the singular set $\Sigma_f$ on $V_p$.
\item[(b)] %\label{item:Gb}
The mean curvature function $H^L$ of $f$ is bounded on $V_p$
if and only if 
$($where the definition of the sign $\sigma$ is given in \eqref{3606}$)$
\begin{equation}\label{eq:518}
3\kappa\mu_1 =8 \mu_0 \mu'_0+3 \sigma  \kappa'
\end{equation}
holds on $(I\times \{0\})\cap V_p$, where $\mu_i$ $(i=1,2)$
are the coefficients of the function $\mu(s,t)$
as in \eqref{eq:459}.
\item[(c)] \label{item:Gd}
The generalized cuspdial edge singular point $p$
cannot be an accumulation point of umbilics of $f$.
\item[(d)]  \label{item:Gc}
The point $p$ is a cuspidal edge singular point
if and only if $\mu_0(s_0)\ne 0$.
\end{enumerate}
Moreover, if
$\Gamma''(s_0)$ does not vanish 
$($i.e.~$\kappa(s_0)\ne 0)$, 
then the following assertions hold:
\begin{enumerate}
\item \label{item:G1}
If $f$ is space-like on $V_p\setminus \Sigma_f$,
then $f$ is of concave type 
and $K^L$ is positive at each point on $V_p \setminus \Sigma_f$
$($cf. Figure~\ref{ex:L123}, right$)$.
\item \label{item:G2}
If $f$ is time-like on $V_p\setminus \Sigma_f$,
then the Gaussian curvature $K^L$ is positive 
$($resp. negative$)$
if and only if $f$ is of convex type $($resp. concave type$)$ at $p$, 
see Definition~\ref{def:CC613b}  and Figure~\ref{ex:L123}.
As a special case, if $p$ is not a cuspidal edge, then
the Gaussian curvature $K^L$ is positive and  $f$ is of convex type
$($cf. Corollary \ref{cor:3995}$)$.
\item \label{item:G3}
The Gaussian curvature $K^L$ diverges and 
takes the same sign on both sides of $\Sigma_f$, and
the two principal curvatures are both unbounded.
In this situation, if $K^L$ is positive $($resp. negative$)$,
then
the two principal curvatures both 
take values in $\R$ $($resp. $\C\setminus \R)$ 
on $V_p$.
\item \label{item:G4}
The quasi-umbilical points of $f$
cannot accumulate at $p$.
\end{enumerate}
}

\begin{figure}[htb]
\begin{center}
\includegraphics[height=4.0cm]{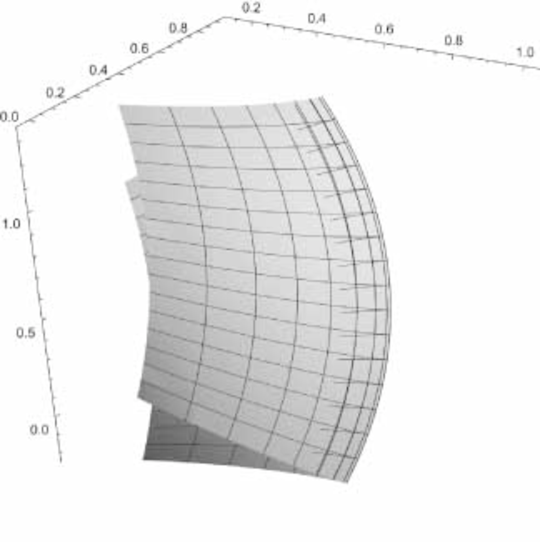}
\includegraphics[height=4.0cm]{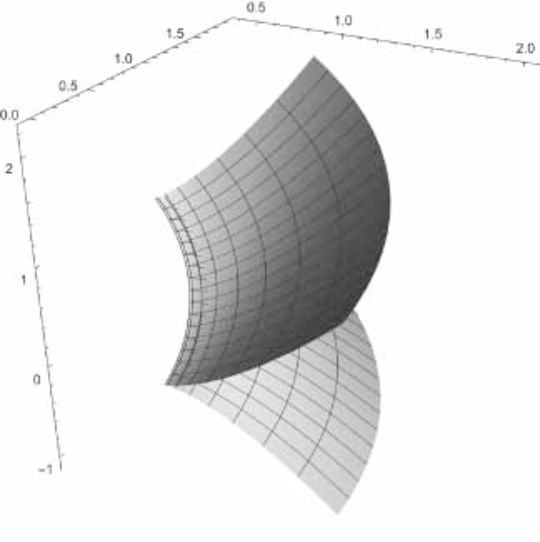}
\includegraphics[height=4.0cm]{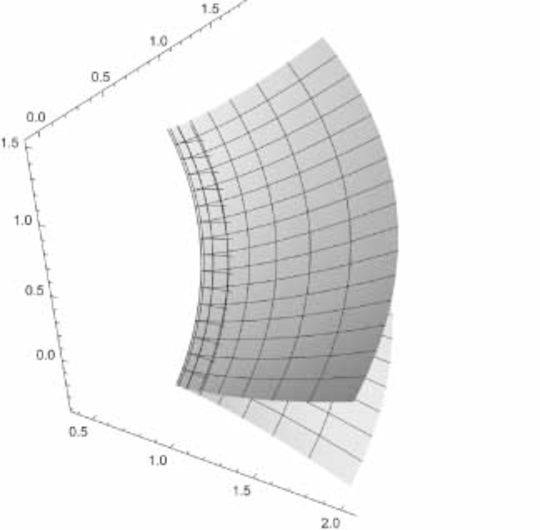}
\end{center}
\caption{Cuspidal edges of order four 
along a light-like helix, which
is time-like of convex type (left),
and time-like of concave type (center) and
space-like (right),
respectively. These are explained in
Example \ref{exa:4161}.}
\label{ex:L123}
\end{figure}

Theorem~G is not vacuous. 
In fact, there is a representation formula (cf. \eqref{eq:fT2499})
producing all generalized cuspidal edges of order four in $\Lo^3$, by which
we can construct examples satisfying the assumptions of Theorem~G.
At first glance, the conditions \eqref{eq:km476} and \eqref{eq:518}
seem to depend on the choice of an orthogonal coordinate system in $\Lo^3$.
However, these conditions are invariant under Lorentzian motions
belonging to the identity component of the isometry group of $\Lo^3$ 
(cf. Proposition~\ref{Cor:I1} and Corollary~\ref{Cor:I2}).
The assertion \ref{item:G4} of Theorem~G
cannot be expected when $\Gamma''(s)=\mb 0$. 
In fact, quasi-umbilical points may accumulate at the
singular point $(s,0)$ of $f$ satisfying $\Gamma''(s)=\mb 0$.
More precisely, such an example is 
constructed in  Akamine \cite[Figure 4]{A} (see Remark~\ref{rmk:DQ}). 
As stated in Theorems D and G, cuspidal edges of type L whose order are
 at most four cannot change their causal types. 
However, there exist cuspidal edges of type L with order five such that
they are space-like on one side of the singular set, and
time-like on the other side (see Example \ref{ex:34239}).
This means that cuspidal edges of type L also can change
their causal type as well as in the case \ref{item:A3} in Theorem~A.

In $\E^3$, minimal surfaces which are constructed using the 
Weierstrass formula do not admit cuspidal edges but only admit branch points
as isolated singular points.
As mentioned before Theorem G,
there are two canonical classes of
zero-mean curvature (``ZMC'' in short) surfaces in $\Lo^3$
without branch points, called ``maxfaces" and ``minfaces".
Roughly speaking, maxfaces (resp. minfaces)
are ZMC surfaces which can be represented as a holomorphic 
(resp. para-holomorphic) data using the Weierstrass-type representation formulas,
respectively. 

Akamine \cite{A} investigated the sign of the Gaussian curvature  $K^L$ 
near the singular points on a given time-like ZMC-surface (cf. \cite[Theorem~A]{A}.
In particular, at cuspidal edge points of a time-like surface $f$, 
he showed that the sign of $K^L$ depends on the convexity 
or concavity of $f$ (cf. \cite[Theorem~B]{A}), and also
investigated umbilical points and quasi-umbilical points
near the singular points. As a consequence, Akamine \cite{A} proved
(a), (d), (2), (3) and (4) of Theorem~G
for ``minfaces" ((1) for maxface was shown in \cite{UY}
except for the concavity),
and so, Theorem~G can be considered as generalizations of Akamine's results
(as we have already mentioned, 
generalized cuspidal edge singular points appeared on 
maxfaces and minfaces are all of order four).

Through Akamine's work, we can observe how the cuspidal edges are special 
amongst the non-degenerate singular points on zero-mean curvature surfaces.
Moreover, by Theorem~G, we are able to recognize that these 
interesting properties of generalized cuspidal edges having vanishing mean curvature
functions
attribute to the 
property that their order $i_p$ are equal to four 
along the singular curve. 
Also, we note that
the case where $\Gamma'(s)$ 
is a light-like vector only at a value $s=s_0$ is also 
investigated in Section 4.

We organize the paper as follows: In Section 1, we introduce the fundamental 
properties of generalized cuspidal edges in $\Lo^3$
and define \lq\lq singular curvature\rq\rq\ and \lq\lq
limit normal curvature\rq\rq\ along their singular sets. 
In Section 2 (resp. Section 3), we investigate the behaviors of
$H^L$ and $K^L$ for generalized cuspidal edges when $\Gamma$ is
time-like (resp. space-like), and prove Theorem~A and Corollary~B
in Sections 2 and 3. In Section 4, we consider the case that $\Gamma$ is light-like
and prove Theorem~D, Propositions C and E. 
In Section 5, we consider generalized cuspidal edges of order four,
and prove Proposition~F and Theorem~G. We have three appendices: In the first appendix,
a representation formula for cusps in Lorentz-Minkowski plane $\Lo^2$ is given.
In the second appendix, the existence of a certain parametrization 
of the generalized cuspidal edge is shown. In the third appendix, we discuss 
umbilical points on wave fronts in $\Lo^3$.

\section{Preliminaries}
In this section, we prepare fundamentals of the geometry of generalized
cuspidal edges in~$\Lo^3$ comparing to the Euclidean case.

\subsection{Regular surfaces in $\E^3$}

Let $U$ be a domain in the $uv$-plane $(\R^2;u,v)$.
A $C^\infty$-map $f:U\to \R^3$ is called
a {\it regular surface} if it is an immersion on $U$.
We denote by \lq\lq$\cdot$\rq\rq\
the canonical positive definite inner product on $\E^3$.
When the image of $f$ lies in $\E^3$,  the functions on $U$
$$
E:=f_u\cdot f_u,\quad
F:=f_u\cdot f_v, \quad 
G:=f_v\cdot f_v
$$
are called the {\it coefficients of the first fundamental form} $ds^2:=Edu^2+2Fdudv+Gdv^2$.
We set $\tilde \nu^E:=f_u \times f_v$,
which gives the normal vector field of $f$ on $U$,
where \lq\lq$\times $\rq\rq denotes the canonical vector product on $\E^3$.
By definition,  
\begin{equation}\label{eq:N}
(\Delta_E:=)\tilde \nu^E\cdot \tilde \nu^E =EG-F^2
\end{equation}
holds. We then set
\begin{align}
\label{tL}
\tilde L&:=f_{uu}\cdot \tilde \nu^E=\det(f_u,f_v,f_{uu}), \\
\label{tM}
\tilde M&:=f_{uv}\cdot \tilde \nu^E=\det(f_u,f_v,f_{uv}), \\
\label{tN}
\tilde N&:=f_{vv}\cdot \tilde \nu^E=\det(f_u,f_v,f_{vv}),
\end{align}
which  are not the usual coefficients 
of the second fundamental form,
because  $\tilde \nu^E$ may not be a unit vector field.
By setting
$$
L:=\frac{\tilde L}{\sqrt{\Delta_E}},\quad
M:=\frac{\tilde M}{\sqrt{\Delta_E}},\quad
N:=\frac{\tilde N}{\sqrt{'\Delta_E}},
$$
$Ldu^2+2Mdudv+Ndv^2$ gives the second fundamental form
of $f$ with respect to the unit normal vector field
$\nu^E:=\tilde \nu^E/|\tilde \nu^E|_E$,
where
$$
|\mb a|_E:=\sqrt{\mb a\cdot \mb a}\,\,(\mb a \in \E^3).
$$
We set
\begin{equation}\label{eq:WfE0}
\tilde W^E:=
\pmt{G & -F \\
          -F & E}\pmt{\tilde L & \tilde M \\
                              \tilde M & \tilde N}
\end{equation}
and
\begin{equation}\label{eq:WfE}
W^E:=\pmt{E & F \\ F & G}^{-1}\pmt{L & M \\ M & N}=
\frac{1}{\Delta_E^{3/2}}\tilde W^E.
\end{equation}
The matrix $W^E$ is called the {\it Weingarten matrix} or 
{\it the shape operator}.
Using this, 
\begin{align}\label{eq:KH584a}
K^E&:=\det(W^E)=\frac{\det(\tilde W^E)}{\Delta_E^3}
=\frac{\tilde L \tilde N-\tilde M^2}{\Delta_E^2}, \\ \label{eq:KH584aa}
H^E&:=\op{trace}(W^E)=
\frac{\op{trace}(\tilde W^E)}{2\Delta_E^{3/2}}
\end{align}
give the Gaussian curvature function and the mean curvature function of $f$, respectively.

\subsection{Regular surfaces in $\Lo^3$}
We next consider the case that $f$ is 
a regular surface (i.e.~an immersion)
 into  $\Lo^3$. For two vectors 
 $\mb a,\,\mb b\in \Lo^3$
 as column vectors,
the Lorentzian inner product and the vector product are defined by
\begin{equation}\label{eq:L-inn}
\inner{\mb a}{\mb b}:=\mb a^T E_3\mb a,\quad
\mb a \times_L \mb b:=E_3(\mb a\times \mb b),
\qquad E_3:=\pmt{1 & 0 & 0 \\
                 0 & 1 & 0 \\
                 0 & 0 &-1},
\end{equation}
where the symbol ${}^T$ denotes matrix transposition.
Then
\begin{equation}\label{eq:LMN}
\tilde \nu_L:=f_u\times_L f_v(=E_3 \tilde \nu_E)
\end{equation}
is a normal vector field of $f$ in $\Lo^3$.

\begin{Definition}\label{def:causality712}
Let $f:U\to \R^3$ be a regular surface.
A point $p\in U$ is said to be {\it space-like, time-like, light-like}
if the normal vector $\tilde \nu^L(p)$ is 
time-like, space-like, light-like in $\Lo^3$, respectively.
Moreover, $f$ is said to be {\it space-like} (resp. {\it time-like})
if all points of $U$ are space-like (resp. time-like).
\end{Definition}

Here, we consider the case that $f$ is a space-like or
time-like surface on $U$. We set
$$
E^L:=\inner{f_u}{f_u},\quad
F^L:=\inner{f_u}{f_v}, \quad 
G^L:=\inner{f_v}{f_v},
$$
and call the  function
\begin{equation}\label{eq:Delta}
\Delta_L:=E^LG^L-(F^L)^2
\end{equation}
the {\it identifier of the causality} of $f$
which depends on  the choice of coordinate system $(u,v)$ and
differs only by the multiplication of a
positive function under a coordinate change.
Since
$
\tilde \nu^L=E_3 \tilde \nu^E
$,
we have
\begin{align}\label{eq:inner387}
\inner{\tilde \nu^L}{\tilde \nu^L}
&=\inner{\tilde \nu^E}{\tilde \nu^E}
=(f_u\times f_v)^T E_3 (f_u\times f_v) \\
& \nonumber
=\det(E_3)(f_u\times f_v)\cdot (E_3 f_u\times E_3 f_v)\\  
\nonumber
&
=
-\Big((f_u\cdot E_3 f_u)(f_v\cdot E_3 f_v)-(f_u\cdot E_3 f_v)(f_v\cdot E_3 f_u) \Big) \\
\nonumber
&=
-\inner{f_u}{f_u}\inner{f_v}{f_v}+\inner{f_u}{f_v}^2
=-\Delta_L.
\end{align}
Moreover, since 
$
\inner{\mb v}{\tilde \nu^L}=\mb v\cdot \tilde \nu^E
$
holds for any vector $\mb v\in \R^3$, we have
(cf. \eqref{eq:LMN})
$$
\inner{f_{uu}}{\tilde \nu^L}=\tilde L,\qquad
\inner{f_{uv}}{\tilde \nu^L}=\tilde M,\qquad
\inner{f_{vv}}{\tilde \nu^L}=\tilde N.
$$
Regarding 
$$
\inner{\tilde \nu^L}{\tilde \nu^L}=\epsilon |\Delta_L|
\qquad 
\epsilon:=\op{sgn}\left(\inner{\tilde \nu^L}{\tilde \nu^L}\right),
$$
we set
$$
L^L:=\frac{\tilde L}{\sqrt{|\Delta_L}|},\quad
M^L:=\frac{\tilde M}{\sqrt{|\Delta_L}|},\quad
N^L:=\frac{\tilde N}{\sqrt{|\Delta_L}|}.
$$
Then the Gaussian curvature (i.e. the sectional curvature 
with respect to the first fundamental form)
$K^L$ and the mean curvature $H^L$ of $f$ are
defined by (cf. \cite[page 107 and page 101]{O})
\begin{equation}\label{eq:K1La}
K^L:=\epsilon\frac{L^L N^L-(M^L)^2}{\Delta_L}
\qquad
H^L:=\frac{E^L N^L-2F^LM^L+G^LL^L}{2\Delta_L}.
\end{equation}
Since 
$$
L^LN^L-(M^L)^2=\frac{\tilde L \tilde N-\tilde M^2}{|\Delta_L|}
=-\epsilon \frac{\tilde L \tilde N-\tilde M^2}{\Delta_L},
$$
we have that
\begin{equation}\label{eq:K1Lb}
K^L=
\frac{-(\tilde L \tilde N-\tilde M^2)}{\Delta_L^2},\qquad
H^L=\frac{E^L\tilde N-2F^L\tilde M+G^L\tilde L}{2\Delta_L\sqrt{|\Delta_L|}}.
\end{equation}
As an advantage of these expressions, we can use
$\tilde L$, $\tilde M$ and $\tilde N$ given in
\eqref{tL}, \eqref{tM}, \eqref{tN}
to compute $K^L$ and $H^L$.
About $H^L$, there might exist other definitions with different sign.
However, our results are related only  to
the absolute value of $H^L$,
and so this does not affect the latter discussions.
We set
\begin{align}\label{eq:Wf}
\tilde W^L&:=
\pmt{G^L & -F^L \\
          -F^L & E^L}\pmt{\tilde L & \tilde M \\
                              \tilde M & \tilde N}, \\
W^L&:=\pmt{E^L & F^L \\ F^L & G^L}^{-1}\pmt{L^L & M^L 
\\ M^{L} & N^L}=
\frac{1}{\Delta_L\sqrt{|\Delta_L}|}\tilde W^L.
\nonumber
\end{align}
Then $W_L$ is called the {\it Weingarten matrix} or {\it the shape operator}.
Using these two matrices, we have the following expressions:
\begin{equation}\label{eq:KH584b}
K^L=\epsilon \det(W^L)=-\frac{\det(\tilde W^L)}{\Delta_L^3},\qquad
H^L=\op{trace}(W^L)
=\frac{\op{trace}(\tilde W^L)}{2\Delta_L\sqrt{|\Delta_L|}}.
\end{equation}
In particular, the determinant of $W^L$ is $-K^L$ (resp. $K^L$) 
and the trace of $W^L$ is $2H^L$ when $f$ is space-like (resp. time-like).
Thus, the eigenvalues of $W^L$ are invariants of $f$, 
which are called the {\it  principal curvatures} of $f$.
Unlike the Euclidean case, the principal curvatures of $f$ may not
be real-valued if $\Delta_L<0$.

As mentioned in the introduction, the results in $\Lo^3$ 
sometimes exhibit different 
phenomena from those in $\E^3$, which is the motivation for this research.
In this paper, we will investigate the behavior of
the principal curvatures around 
singularities of surfaces in $\Lo^3$ as well as $K^L$ and $H^L$, 
keeping in mind the results obtained in 
the case of Euclidean geometry.

\begin{Definition}\label{def:umbilics}
The point $p$ on the domain of definition of a regular surface $f$ is called
an {\it umbilical point} (resp. a {\it quasi-umbilical point}) if
the two principal curvatures coincide 
and $W^L$ is a diagonal (resp.
not a diagonal) matrix at $p$.
\end{Definition}

\begin{Remark} \label{rmk:U}
If $p$ is a space-like regular point (i.e. $\Delta_L(p)>0$),
then there exists a regular matrix $P$ such that $P^{-1} \tilde W^L P$ is 
a symmetric matrix at $p$, and so quasi-umbilical points never appear like as 
the case of regular surfaces in $\E^3$. 
\end{Remark}

By \eqref{eq:KH584a}  and \eqref{eq:KH584b}, 
the sign of $K^L$ is opposite of that of $K^E$, that is,
we have (cf. \cite[Remark~4.6]{A})
\begin{equation}\label{eq:KEKL}
\op{sgn}(K^L)=-\op{sgn}(K^E).
\end{equation}

\subsection{Generalized cuspidal edges in $\R^3$ and $\E^3$}

In the introduction, we have defined generalized cuspidal edges. 
In this subsection, we show several important properties of them:

\begin{Definition}\label{def:AD880}
Let $f(s,t)$ be a smooth map satisfying 
$f_s(s,0)\ne \mb 0$ and $f_t(s,0)=\mb 0$ for each $s$.
For such $f$,  we consider a new local coordinate system 
$(u,v)$ satisfying
$$
t(u,0)=0,\quad  s_v(u,0)=0,\quad s_u(u,0)>0,\quad t_v(u,0)>0.
$$
In particular,  $f_v(u,0)=\mb 0$ holds for each $u$.
We call such a local coordinate 
$(u,v)$ an {\it admissible local coordinate system}
associated with $f$,
as well as the original coordinate system $(s,t)$.
\end{Definition}

If $f(s,t)$ is a generalized cuspidal edge,
then $f_{tt}(s,0)\ne \mb 0$ holds, 
which can be reproved from the next lemma (we omit the proof):

\begin{Lemma}\label{lem:860}
Let $(s,t)$ be an
admissible coordinate system of a
generalized cuspidal edge $f$.
Let $\phi(s,t)$ be a smooth 
$\R^l$ $(l\ge 1)$ valued function
satisfying 
$\phi_s(s,0)\ne \mb 0$ and $\phi_t(s,0)=\mb 0$ for each $s$.
Then, for a $C^\infty$-function $\phi(s,t)$,
the first non-vanishing order $k$ 
of the derivative $\partial^k \phi(s,0)/\partial t^k$
is independent of the choice of an admissible coordinate system of $f$.
Moreover, if $k$ is even, then
the sign of $\partial^k \phi(s,0)/\partial t^k$
is also independent of the choice of an admissible coordinate system.
\end{Lemma}

We introduce here typical examples of  generalized cuspidal edges:

\begin{Proposition}\label{prop:1182}
A cuspidal edge along a regular curve $\Gamma$ in $\R^3$
as in Definition~\ref{def:217} is a generalized cuspidal edge.
\end{Proposition}

\begin{proof}
Let $f(s,t)$ be a cuspidal edge along $\Gamma$
as in Definition \ref{def:217}.
Then, it has a unit normal vector field $\nu$ on $\mc U_f$, 
and if we set
$\lambda:=\det(f_s,f_t,\nu)$, then the
exterior derivative $d\lambda$ at $(s,0)$ does not vanish
for each $s\in I$ (\cite[Chapter 2]{SUY2}).
Since $\lambda_s(s,0)=0$ and $f_t(s,0)=\mb 0$, we have
$$
0\ne \lambda_t(s,0)=\det(f_s(s,t),f_t(s,t),\nu(s,t))_{t}^{}|_{t=0}
=\det(f_s(s,0),f_{tt}(s,0),\nu(s,0)),
$$ 
which implies that $f_s(s,0)$ and $f_{tt}(s,0)$ 
are linearly independent, proving the assertion.
\end{proof}

\begin{Example}\label{ex:1195}
The $C^\infty$-maps defined by
$$
f_1(s,t)=(s,t^2,0),\quad
f_2(s,t)=(s,t^2,st^3),\quad 
$$
give generalized cuspidal edges. 
The origin $(0,0)$ is a singular point of them, which is
called the {\it standard fold singular point}
and the {\it standard cuspidal cross cap singular point}, respectively.
\end{Example}

\begin{Definition}\label{def:736}
Fix a domain $U$ of $\R^2$. A $C^\infty$-map $f:U\to \E^3$ is called a {\it frontal} if
there exists a nowhere vanishing vector field  $\tilde \nu:U\to \E^3$ 
which is perpendicular to $f_u$ and $f_v$ at each point of $U$.
Let $P(\R^3)$ be the projective space associated with the
vector space $\R^3$ and $\R^3\ni \mb v\to [\mb v]\in P(\R^3)$ the canonical
projection. Then $f$ is called a {\it wave front} if
$U\ni p\mapsto (f(p),[\tilde \nu(p)])\in \R^3\times P(\R^3)$ is an immersion.
\end{Definition}

The maps $f_1$ and $f_2$ 
given in Example~\ref{ex:1195}
are frontals, but not are wave fronts.
Let $f:=f(s,t)$ be a generalized cuspidal edge along $\Gamma$,
and consider the vector field defined by
$
\mb v(s,t):={f_{t}(s,t)}/{t}\,\, (|t|<\delta),
$
which is $C^\infty$-differentiable at $t=0$ such that
$
f_{tt}(s,0)=\mb v(s,0)(\ne \mb 0).
$
By condition (c) of Definition \ref{def:GE},
\begin{equation}\label{tnuE}
\hat \nu^E:=f_s(s,t)\times \mb v(s,t)
\end{equation}
gives a normal vector field of $f$, 
which implies that $f$ is a frontal, that is,
$\hat \nu^E$ can be smoothly extended to the singular set.
Then, we have
\begin{equation}\label{tnuE0}
\hat \nu^E(s,0)=f_s(s,0)\times \mb v(s,0)=\Gamma'(s)\times f_{tt}(s,0).
\end{equation}
In this setting,
\begin{equation}\label{eq:nuE_unit}
\nu^E:=\frac{\hat \nu^E}{|\hat \nu^E|_E}
\end{equation}
gives a unit normal vector field of $f$ on $\mc U_f$
which can be smoothly extended  across the singular set $\Sigma_f$ of $f$.
As the converse of Proposition \ref{prop:1182},
the following is a well-known criterion 
for cuspidal edges (cf. \cite[Theorem 2.6.3]{SUY2}):

\begin{Fact}\label{frac:frontC}
A point $(s,0)$ of a generalized cuspidal edge is 
a cuspidal edge singular point if $f$ is a
wave front at $(s,0)$, that is, the map 
$
(s,t)\mapsto \left(f(s,t),\nu^E(s,t)\right)
$
is an immersion at $(s,0)$.
\end{Fact}

Using this, we prepare the following:

\begin{Proposition}\label{lem:ce708}
Let 
$f:\mc U_f\to \R^3$ be a generalized cuspidal edge
along $\Gamma$. Then $f$ is a frontal, that is, 
its normal vector field $($as $\R^3=\E^3$ or $\R^3=\Lo^3$ see
Remark~\ref{Rmk:1586}$)$ is defined at each singular point of $f$.
Moreover, let $P$ be a plane passing through $f(s_0,0)$ $(s_0\in I)$
which is transversal to the curve $\Gamma$ at $\Gamma(s_0)=f(s_0,0)$.
Then the section $\mb c$ of $f$ by $P$ is 
a generalized cusp in $P$ $($cf. Definition~\ref{def:GC1794}$)$,
and the point $p:=(s_0,0)$ is a cuspidal edge singular point 
if and only if $p$ is a cusp of $\mb c$ in the plane $P$.
\end{Proposition}

\begin{proof}
This statement is essentially the same statement as in
\cite[Lemma 3.2]{HNSUY2}. However, our proof is new and
is different from that in \cite{HNSUY2}:
Take a basis spanning the plane of $P$ at $f(s_0,0)$,
and extend it as a frame field $\{\mb a_1,\,\mb a_2\}$  along $\Gamma$.
By Proposition B.1 in the appendix, $f$ can be written in the form
$$
f(s,t)=\Gamma(s)+ x(s,t)\mb a_1(s)+y(s,t) \mb a_2(s),
$$
where $x(s,t)$ and $y(s,t)$ are smooth functions.
Since the map 
$$
(s,x,y)\mapsto \Gamma(s)+ x\mb a_1(s)+y\mb a_2(s)
$$
gives a tubular neighborhood of $\Gamma$ in $\R^3$,
$
f_0(s,t):=(x(s,t),y(s,t),s)
$
is a generalized cuspidal edge in $\R^3$
defined on a neighborhood of $(s_0,0)\in U$.
Without loss of generality, we can write
$$
x(s,t)=t^2 \hat x(s,t),\qquad y(s,t)=t^3\hat y(s,t),
$$
where $\hat x(s,t)$ and $\hat y(s,t)$ are smooth functions
satisfying $\hat x(s_0,0)\ne 0$. Since
the section $\mb c$ of $f$ by $P$ is a generalized cusp in $P$,
which can be identified with the curve
$
\mb c:t\mapsto (x(s_0,t),y(s_0,t))
$
in $\R^2$,  and the point $(s_0,0)$ corresponds to the 
singular point of $\mb c$. 
Moreover, $(s_0,0)$ is a cusp point of $\mb c$
if and only if $\hat y(s_0,0)\ne 0$.
Then, it can be  easily checked that
\begin{equation}\label{eq:1283}
N:=
\Big(-t \left(t \hat y_t+3 \hat y\right),t \hat x_t+
2 \hat x,
\big(\left(t \hat y_t+3 \hat y\right)\hat x_s -
\left(t \hat x_t+2 \hat x\right) \hat y_s\big)t^3 \Big)
\end{equation}
is a normal vector field of $f_0$ in the Euclidean 3-space $\E^3$.
So $f_0$ (and also $f$) is a frontal.  The unit normal vector field
$
\nu^E(s,t):={N(s,t)}/{|N(s,t)|_E}
$
of $f_0$ in $\E^3$ has the following asymptotic expansion
$$
\nu^E(s,t)=(0,1,0)+t \frac{3 \hat y(s,0)}{2\hat x(s,0)}\mb e_2+O(t^2)
\qquad (\mb e_2:=(0,1,0)),
$$
where $O(t^2)$ is the term such that $O(t^2)/|t|^2$ is a 
vector-valued bounded function of $(s,t)$ around  $(s_0,0)$.
So, the matrix
$$
\pmt{ 
(f_0)_s(s_0,0)& (f_0)_t(s_0,0) \\
\nu_s(s_0,0)&  \nu_t(s_0,0)
}
=
\pmt{ 
\mb e_3& 0 \\
* & \frac{3 \hat y(s_0,0)}{2\hat x(s_0,0)}\mb e_2
}
\qquad (\mb e_3:=(0,0,1))
$$
is of rank two if and only if $\hat y(s_0,0)\ne 0$. 
So $f_0$ (and also $f$) is a wave front 
on a sufficiently small neighborhood of $(s_0,0)$
if and only if $\hat y(s_0,0)\ne 0$ .
Since $(s_0,0)$ is a cusp point of $\mb c$
iff $\hat y(s_0,0)\ne 0$,
we obtain the last assertion by Fact \ref{frac:frontC}.
\end{proof}

\begin{Definition}
The vector field
\begin{equation}\label{eq:DEf}
\tilde{\mb D}^E_f(s):=f_{tt}(s,0)-\frac{f_{tt}(s,0)\cdot \Gamma'(s)}
{\Gamma'(s)\cdot \Gamma'(s)}\Gamma'(s)
\end{equation}
is called an {\it $\E^3$-cuspidal direction vector} of $f$ at $(s,0)$.
Then it is the projection of the vector $f_{tt}(s,0)$
to the normal plane of $\Gamma$ in $\E^3$.
By (c) of Definition \ref{def:GE},
we can set
\begin{equation}\label{eq:DE1101}
{\mb D}^E_f(s):=\frac{\tilde{\mb D}^E_f(s)}{|\tilde{\mb D}^E_f(s)|_E}
\end{equation}
and call it the {\it unit $\E^3$-cuspidal direction vector}
of $f$ at $s$.
\end{Definition}

By definition, $\tilde{\mb D}^E_f(s)$ is determined up to a
positive scalar multiplication if we take other admissible 
parametrizations of $f$ (cf. Definition \ref{def:AD880}).
Thus, ${\mb D}^E_f(s)$ is uniquely determined from $f$.

\begin{Proposition}\label{prop:1093}
$\Gamma'(s)\times {\mb D}^E_f(s)$ is a positive scalar 
multiplication of $\nu^E(s,0)$.
\end{Proposition}

\begin{proof}
In fact, it holds that (cf. \eqref{tnuE0})
$$
\Gamma'(s)\times {\mb D}^E_f(s)
=
\frac{\Gamma'(s)\times f_{tt}(s,0)}{|\tilde{\mb D}^E_f(s)|_E}=
\frac{\hat \nu^E(s)}{|\tilde{\mb D}^E_f(s)|_E},
$$
where $\nu^E$ is given in \eqref{eq:nuE_unit} and $\hat \nu^E=\nu^E(s,0)$.
\end{proof}

If we are thinking of $f:\mc U_f\to \R^3$ as a generalized cuspidal edge
in $\E^3$, the singular curvature function along $\Gamma$ is
defined by (cf.~\cite{F})
\begin{align}
\label{eq:ksE}
\kappa_s^E(s)&:=\frac{\Gamma''(s)\cdot \mb D^E_f(s)}
{\Gamma'(s)\cdot \Gamma'(s)}. 
\end{align}

\begin{Remark}
The original definition of the singular curvature $\kappa_s^E$ (cf. \cite{SUY,SUY2}) is
\begin{equation}\label{eq:1299}
\kappa_s^E:=\op{sgn}(\lambda_t) 
\frac{\det(\Gamma',\Gamma'',\hat \nu^E)}{|\Gamma'|_E^3}.
\end{equation}
We may assume $|\Gamma'|_E=1$. By setting   $\mb n:=\hat \nu^E\times \Gamma'$, 
we have
$$
\lambda_t=\det(\Gamma',f_{tt},\hat \nu^E)
=\mb n\cdot f_{tt}=\mb n\cdot D^E_f.
$$
Since $\mb n$ is a unit vector field,  we have $\mb n =\epsilon \mb D^E_f$,
where $\epsilon:=\op{sgn}(\lambda_t)$. Thus
$$
\kappa_s^E
=\epsilon \det(\Gamma',\Gamma'',\hat \nu^E)
=\epsilon (\hat \nu^E\times \Gamma')\cdot \Gamma''
=\epsilon \mb n \cdot \Gamma''=\mb D^E_f\cdot \Gamma''
$$
holds, and  \eqref{eq:ksE} is obtained.
\end{Remark}

This definition of $\kappa_s^E$ does not depend on the choice of
the parameter $(s,t)$ of $f$ and the sign ambiguity of
the normal vector field $\nu^E(s,t)$.
When $\kappa_s^E>0$ (resp. $\kappa_s^E<0$)
 at $p:=(s,0)$, then the orthogonal projection of the image of $f$
to the limiting tangent plane looks convex (resp. concave), 
see \cite[Fig. 5.2]{SUY2}. Regrading this, we give the following:

\begin{Definition}\label{def:convex}
A generalized cuspidal edge singular point $p:=(s,0)$ of $f$ 
is called of {\it convex type} (resp {\it concave type}) if the sign of $\kappa_s^E$ 
is positive (resp. negative), see Figure \ref{ex:C2}.
\end{Definition}

The limiting normal curvature function along the $s$-axis is defined by
\begin{equation}
\label{eq:knE}
\kappa^E_\nu(s):=\frac{\Gamma''(s)\cdot \nu^E(s,0)}{\Gamma'(s)\cdot\Gamma'(s)}.
\end{equation}
We denote by $\kappa^E(s)$ the curvature function of $\Gamma$
in $\E^3$. Then, it holds that
\begin{equation}\label{eq:pyE}
(\kappa^E)^2=(\kappa_s^E)^2+(\kappa_\nu^E)^2.
\end{equation}

\begin{Prop}\label{prop:GE}
For a singular point $p:=(s,0)$ of the generalized cuspidal edge $f$,
the sign of $\kappa_\nu^E(s)$ coincides with the sign $\sigma^C(s)$
$($cf. \eqref{eq:Ep258} and \eqref{eq:Dp262}$)$, that is,
\begin{equation}
\op{sgn}(\kappa_\nu^E(s))=\sigma^C(p)
\end{equation}
holds. Moreover,  $\kappa^E_\nu(s)=0$ $($i.e. $\sigma^C(p)=0)$ holds
if and only if $\mb D^E_f(s)$ lies in the osculating plane of $\Gamma$ at 
the point $\Gamma(s)$.
\end{Prop}

\begin{proof}
Since $\nu^E(s,0)$ is a positive scalar multiplication of
$f_s(s,0)\times f_{tt}(s,0)$ by \eqref{tnuE0},
 the numerator  $\Gamma'(s)\cdot \nu^E(s,0)$
of \eqref{eq:knE} is a non-zero scalar multiple of
$$
\Gamma''(s)\cdot(f_s(s,0)\times f_{tt}(s,0))=\det(f_{ss}(s,0),f_s(s,0),f_{tt}(s,0))
=d^C(p),
$$
proving the assertion. By the definition, $\mb D^E_f(s)$
lies in the osculating plane of $\Gamma$ if and only if
$\det(f_{ss}(s,0),f_s(s,0),f_{tt}(s,0))$ vanishes. So the second assertion is obtained.
\end{proof}

\subsection{Generalized cuspidal edges in $\Lo^3$}

Fix a generalized cuspidal edge $f(s,t)$ along $\Gamma(s)$
in the Lorentz-Minkowski 3-space $(\Lo^3;x,y,z)$ of signature $(++-)$.
We define the causal types of the regular curve $\Gamma$ in $\Lo^3$:
 
\begin{Definition}
A regular curve $\Gamma:I\to \Lo^3$ is said to be of {\it type $S$} (resp. {\it type $T$})
if $\Gamma'(s)$ is a space-like  (resp.  time-like) vector for each $s\in I$.
Moreover, $\Gamma:I\to \Lo^3$ is said to be of type $L$
if $\Gamma'(s)$ is a light-like vector for each $s\in I$.
\end{Definition}

We consider the vector field defined by
\begin{equation}\label{eq:ve1522}
\mb v(s,t):=\frac{f_{t}(s,t)}{t} \qquad (|t|<\delta),
\end{equation}
which is $C^\infty$-differentiable at $t=0$ such that
$
f_{tt}(s,0)=\mb v(s,0)(\ne \mb 0).
$
By (c) of Definition \ref{def:GE},
\begin{equation}\label{tnuL}
\hat \nu^L(s,t):=f_s(s,t)\times_L \mb v(s,t)
\end{equation}
gives a normal vector field of $f$ in $\Lo^3$ for sufficiently small 
choice of $\delta$, which implies that $f$ is a frontal and
\begin{equation}\label{tnuL2}
\hat \nu^L(s,0)=f_s(s,0)\times_L \mb v(s,0)=\Gamma'(s)\times_L f_{tt}(s,0)
\end{equation}
holds.  Unless $\hat \nu^L(s,0)$ is not a light-like vector,
we can set (for sufficiently small $|t|$)
$$
\nu^L(s,t):=\frac{\hat \nu^L(s,t)}{|\hat \nu^L(s,t)|_L},
$$
which gives a unit normal vector field of $f$,
where
$$
|\mb a|_L:=\sqrt{|\inner{\mb a}{\mb a}|}\,\,(\mb a \in \Lo^3).
$$
\begin{Remark}\label{Rmk:1586}
Since $f_s\times_L \mb v$ 
coincides with $E_3 (f_s\times_E \mb v)$
(see \eqref{eq:L-inn} for the definition of $E_3$),
a smooth map $f$ defined on a domain in $\R^2$
into $\Lo^3$ admits a Lorentzian normal vector field
if and only if $f$ admits a Euclidean normal vector field.
\end{Remark}

Regarding this,  a smooth map $f:\mc U_f\to \Lo^3$ is called a {\it frontal}
if and only if it is a frontal as a map into $\E^3$.
Then, by Proposition \ref{lem:ce708}, a generalized cuspidal edge is a frontal.

\begin{Definition}\label{def:causality712b}
Let $f:\mc U_f\to \Lo^3$ be a generalized cuspidal edge.
A point $p\in \mc U_f$ is said to be {\it space-like, time-like, light-like}
if the normal vector $\hat \nu^L(p)$ is 
time-like, space-like, light-like in $\Lo^3$, respectively.
Moreover, $f$ is said to be {\it space-like} (resp. {\it time-like})
if all points of $\mc U_f$ are space-like (resp. time-like).
\end{Definition}

By definition,
$p\in \mc U_f$ is {\it space-like, time-like, light-like}
if so is the limiting tangent plane.
Similar to \eqref{eq:DEf}, we give the following definition,
under the assumption that $\Gamma$ is of type $T$ or of type $S$.

\begin{Definition}\label{def:CD617}
The vector
\begin{equation}\label{eq:DLf}
\tilde{\mb D}^L_f(s):=f_{tt}(s,0)-
\frac{\inner{f_{tt}(s,0)}{\Gamma'(s)}}{\inner{\Gamma'(s)}{\Gamma'(s)}}\Gamma'(s)
\end{equation}
is called the  {\it $\Lo^3$-cuspidal direction vector} of $f$ at $(s,0)$
if $\Gamma$ is of type $S$ (resp. type $T$) on $\mc U_f$.
\end{Definition}

By (c) of Definition \ref{def:GE}, $\tilde{\mb D}^L_f(s)$ never vanishes, 
and can be considered as the projection of the vector $f_{tt}(s,0)$
to the normal plane of $\Gamma$ at $\Gamma(s)$ in $\Lo^3$.
Modifying the proof of Proposition \ref{prop:1093}, we obtain the following assertion:

\begin{Proposition}\label{prop:1335}
$\Gamma'(s)\times \tilde{\mb D}^L_f(s)$
is a positive scalar multiple of $\nu^L(s,0)$.
\end{Proposition}

In particular, the following assertion holds:

\begin{Corollary}
Let $f$ be a generalized cuspidal edge. Then, for each $s\in I$,
\begin{enumerate}
\item $\hat \nu^L(s,0)$ is a space-like vector 
if $\Gamma'(s)$ is time-like or $\tilde{\mb D}^L_f(s)$ is time-like,
\item $\hat \nu^L(s,0)$ is
a time-like vector if $\Gamma'(s)$ and $\tilde{\mb D}^L_f(s)$ are both 
space-like, and 
\item $\hat \nu^L(s,0)$ is light-like if  
$\Gamma'(s)$ is space-like and $\tilde{\mb D}^L_f(s)$ is light-like.
\end{enumerate}
\end{Corollary}

When $s$ is the arclength parameter of $\Gamma$,
the {\it curvature function} of $\Gamma$ in $\Lo^3$
is defined by
$
\kappa^L(s):=|\Gamma''(s)|_L.
$
The following assertion is obvious 
(cf.~Definition~\ref{def:causality712b}):

\begin{Proposition}
If the limiting normal vector $\hat \nu^L(s,0)$ is
space-like $($resp. time-like$)$,
then $f$ is time-like $($resp. space-like$)$ on
its regular set $\mc U_f\setminus \Sigma_f$ $($cf. \eqref{eq:Domf2}$)$.
\end{Proposition}

From now on, we consider the case
that 
$\nu^L(s,t):=\hat \nu^L/|\hat \nu^L|_L$
is well-defined (that is, $(s,0)$ is a space-like or a time-like point),
and assume that $\Gamma'(s)$ and  $\tilde{\mb D}^L_f(s)$ are not light-like vectors.
Then
\begin{equation}\label{eq:DL1386}
{\mb D}^L_f(s):=\frac{\tilde{\mb D}^L_f(s)}{|\tilde{\mb D}^L_f(s)|_L}
\end{equation}
is well-defined, which is called the {\it unit 
$\Lo^3$-cuspidal direction vector} of $f$ in $\Lo^3$.
We define the {\it singular curvature}  $\kappa^L_s(s)$
and the {\it limiting normal curvature} $\kappa^L_\nu(s)$ so that
\begin{equation}
\label{eq:knSL}
\Gamma''(s)=
\kappa^L_s(s) {\mb D}^L_f(s)
+
\kappa^L_\nu(s) \nu^L(s,0),
\end{equation}
where $s$ is the arc-length parameter of $\Gamma$.
Then the following assertion holds:

\begin{Prop}\label{prop:GL}
Suppose that $\Gamma''(s)$ does not vanish.
Then $\kappa_\nu^L(s)$ is equal to zero
if and only if so is $\kappa_\nu^E(s)$.
\end{Prop}

\begin{proof}
We may assume that $s$ is the arc-length parameter of $\Gamma$.
By the definition of $\nu^L(s)$, we have 
$\inner{f_{tt}(s,0)}{\nu^L(s)} =\inner{\Gamma'(s)}{\nu^L(s)}=0$.
Since $\nu^E(s,0)$ points in the same direction as
$E_3 \nu^L(s,0)$, we can write
$\nu^E(s,0)=a(s)E_3 \nu^L(s,0)$, where $a(s)$ is a non-zero smooth
function.
By \eqref{eq:knSL}, we have 
\begin{align*}
\op{sgn}(\kappa^E_\nu(s))
&=\op{sgn}(\Gamma''(s)\cdot \nu^E)
=\op{sgn}(a(s)) \op{sgn}\Big(\inner{\Gamma''(s)}{\nu^L(s,0)}\Big) \\
&
=\op{sgn}\Big(a(s) \op{\inner{\nu^L(s,0)}{\nu^L(s,0)}}\Big)
\op{sgn}(\kappa^L_\nu(s)).
\end{align*}
Since $a(s)\ne 0$, the function $\inner{\nu^L(s,0)}{\nu^L(s,0)}$
does not vanish, proving the assertion.
\end{proof}

By imitating the proof of  Proposition \ref{prop:GE}
the following assertion is obtained:

\begin{Proposition}
The section of the image of $f$ by 
the normal plane $P_s$ of $\Gamma$ at $\Gamma(s)$
in $\Lo^3$ is a generalized cusp $($cf. Definition \ref{def:GC1794}$)$
whose cuspidal direction vector in the plane $P_s$
is $\tilde{\mb D}^L_f(s)$.
Moreover, $\tilde{\mb D}^L_f(s)$ lies in the osculating plane of $\Gamma$ at $\Gamma(s)$
if and only if $\kappa^L_\nu(s)=0$.
\end{Proposition}

In the introduction, we mentioned the order $i_p$ at each singular point $p$.
We now define this concept as follows:
Let $f:\mc U_f\to \Lo^3$ be a generalized cuspidal edge.
Since $f_t(s,0)=\mb 0$, the identifier  of causality can be expressed as
$$
\Delta_L(s,t)
(:=E^L(s,t)G^L(s,t)-F^L(s,t)^2)
=t^r \phi_s(t),
$$
where $r(\ge 2)$ is a positive integer and $\phi_s(t)$ is
a certain smooth function with respect to $t$ depending on $s$.

\begin{Definition}\label{def:R}
In the above setting, if $\phi_s(0)\ne 0$ ($p:=(s,0)\in \Sigma_f$), we set 
\begin{equation}\label{eq:Ip}
i_p:=r
\end{equation}
and call it the {\it order} of the generalized cuspidal edge $f$ at the singular point $p$.
If $\Delta_L(s,t)$ is a flat function with respect to $t$
(that is, $\partial^k \Delta_L(s,0)/\partial t^k$ vanish for all $k$), 
then we set $i_p:=\infty$.
\end{Definition}

It can be easily checked that the {\it order} $i_p$
is an invariant of $f$ for each $p\in \Sigma_f$
(if the image of $f$ lies in $\E^3$, the same concept
can be defined but it is always equal to $2$). 
In $\Lo^3$, the order $i_p$ is generically $2$, and 
might be greater than $2$ in special cases.
If $i_p$ is an odd number for each $p$, 
then $f$ changes its causal type along the 
the curve $t \mapsto f(s,t)$.

\begin{Proposition}\label{prop:1477}
The order $i_p$ $(p\in \Sigma_f)$
is independent of the choice of
the admissible coordinate system $($cf. Definition \ref{def:AD880}$)$ of $f$. 
\end{Proposition}

\begin{proof}
Let $(u,v)$ be an admissible coordinate change of $f(s,t)$, and let $J(s,t)$
be its Jacobian. Then the identifier of the causality of $f$ with respect to
$(u,v)$ is $J^2 \Delta_L$. Since $J(s,0)\ne 0$, 
the assertion is obtained by applying Lemma \ref{lem:860} with $\phi:=\Delta_L$.
\end{proof}

\begin{figure}[thb]
\begin{center}
\includegraphics[height=3.8cm]{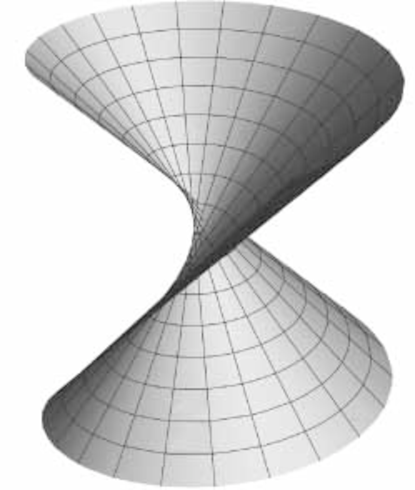}
\end{center}
\caption{The null surface in $\Lo^3$ associated with
a logarithmic spiral
}\label{fig:i-infty}
\end{figure}

We give here examples of cuspidal edges with $i_p=\infty$:

\begin{Example}[cf. \cite{AHUY}]\label{ex:IpInfty}
We let $\gamma(s)$ be a locally convex regular curve in the $xy$-plane
in $\Lo^3$ defined on an interval $I$ whose curvature 
function $\kappa$ and derivative $\kappa'=d\kappa/ds$ are both positive.
We assume that $s$ is the arc-length parameter of $\gamma$.
We then denote by $\mb n(s)$ the left-ward unit normal vector field
along $\gamma$, and consider a smooth map $f:I\times \R\to \Lo^3$ defined by
$$
f(s,t):=(\gamma(s),0)+t(\mb n(s), 1).
$$
As shown in \cite[Proposition 4.5]{AHUY}, $f$ gives a
null wave front in $\Lo^3$, and the singular set $\Sigma_f$ of $f$ is given by
$$
\Sigma_f=\Big\{(s,\frac{1}{\kappa(s)})\,;\, s\in I\Big\}(\subset I\times \R),
$$
which consists only of cuspidal edge singular points.
All points in the domain of definition of $f$ are light-like points
(that is, $\Delta_L(s,t)$ vanishes identically),
and so $i_p=\infty$ for each point $p\in \Sigma_f$.
Let $q$ be a regular point which is not space-like nor time-like
(in the domain of definition of $f$).
If the matrix $\tilde W^L$ vanishes at $q$, then such a point $q$
is called a  {\it light-like umbilical point} (cf. \cite{Tari})
which can be considered as ``fake umbilical points".  
By the above construction, all regular points of $f$
are fake umbilical points.
For example, if $\gamma$ is a logarithmic spiral,
then the image of $f$ is given as in Figure~\ref{fig:i-infty}.
\end{Example}

If a cuspidal edge singular point is not light-like,
then it cannot be an accumulation point of umbilical points
(cf. Appendix C). On the other hand, in the above example, we showed that
light-like umbilical points (i.e. fake umbilics) 
can accumulate at a cuspidal edge singular point.
Since our definition of umbilical points does not contain light-like points,
the question of 
whether umbilical points can accumulate
at a light-like cuspidal edge singular point naturally remains.
In fact, as mentioned in the introduction, the authors know of
any such examples.

\section{General calculations}

\subsection{General setting}
From now on, we will compute the Weingarten matrix of
a generalized cuspidal edge $f$ in $\E^3$ or $\Lo^3$.
Let $\Gamma:I\to \R^3$ be a regular curve
such that the origin $0\in \R$ belongs to the interval $I$.
Let $f:\mc U_f\to \R^3$ be a generalized cuspidal edge
defined on $\mc U_f:=I\times (-\delta,\delta)$,
where $\delta>0$. We will consider the following cases:
\begin{enumerate}
\item[(E)] $f$ lies in $\E^3$.
\item[(T)] $f$ lies in $\Lo^3$ and $\Gamma$ is time-like. 
\item[($\op{S}_s$)] $f$ lies in $\Lo^3$, $\Gamma$ is space-like 
and ${\mb D}_f^L$ is a space-like vector field along $\Gamma$. 
\item[($\op{S}_t$)] 
 $f$ lies in $\Lo^3$, $\Gamma$ is space-like, and ${\mb D}_f^L$ is a time-like vector
field along $\Gamma$. 
\item[($\op{S}_l$)] $f$ lies in $\Lo^3$, $\Gamma$ is space-like, 
and ${\mb D}_f^L(0)$ points in a light-like vector.
\item[($\op{L}_2$)] $f$ lies in $\Lo^3$, $\Gamma'(0)$ is light-like, 
$\Gamma''(0)\ne \mb 0$ and $i_{(0,0)}$ is equal to $2$.
\item[($\op{L}_4$)] $f$ lies in $\Lo^3$, $\Gamma$ is light-like, 
and the order $i_{(s,0)}$ is equal to $4$ for each $s\in I$.
\end{enumerate}

We first give a framework which will be useful to compute the matrix $\tilde W^L$
except for the cases ($\op{L}_2$) and ($\op{L}_4$).
We fix a symmetric bilinear form $(\,,\,)$ on $\R^3$
and let $\{\mb a_0(s),\mb a_1(s),\mb a_2(s)\}$ be a
frame field of $\R^3$ along $\Gamma(s)$ satisfying
\begin{equation}\label{eq:delta1424}
(\mb a_i,\mb a_j)=\epsilon_i \delta_{i,j} \qquad (i,j=0,1,2),
\end{equation}
where $\delta_{i,j}$ is Kronecker's delta, and 
$\epsilon_0,\epsilon_1,\epsilon_2\in \{-1,1\}$.
The cases  (E) and (T) are discussed in this section, and 
the cases  ($\op{S}_s$), ($\op{S}_t$) and ($\op{S}_l$) are discussed 
in the next section. Finally, the two remaining cases 
($\op{L}_2$) and ($\op{L}_4$) are discussed in Sections 4 and 5.

In the following calculation, we assume that $\Gamma(s)$
is parametrized by  arc-length with respect to
the bilinear form $(\,,\,)$, that is,
$$
(\Gamma'(s),\Gamma'(s))=\pm 1\qquad (s\in I)
$$
is assumed. Moreover, we set 
\begin{equation}\label{eq:A01960}
\mb a_0(s):=\Gamma'(s).
\end{equation}
There exists a $3\times 3$ matrix $\mc K$ satisfying
\begin{equation}\label{eq:a1434}
(\mb a'_0(s),\,\mb a'_1(s),\,\mb a'_2(s))=
(\mb a_0(s),\,\mb a_1(s),\,\mb a_2(s))\mc K(s),
\end{equation}
where $\Gamma'(s)=\mb a_0(s)$ for each $s\in I$
and $\mb a'_i:=d \mb a_i/d s$ ($i=0,1,2$).
By \eqref{eq:delta1424}, $\mc K$ can be written in the following form
\begin{equation}\label{eq:K1441}
\mc K:=\pmt{
0 & -\epsilon_0\epsilon_1 \kappa_1 & -\epsilon_0\epsilon_2\kappa_2 \\
\kappa_1 & 0 & -\epsilon_1\epsilon_2 \Omega \\
\kappa_2 & \Omega & 0},
\end{equation}
where $\kappa_1$, $\kappa_2$ and $\Omega$ are $C^\infty$-functions
on the interval $I$.
In fact, if we write $K=(k_{ij})_{i,j=0,1,2}$, then
$$
k_{ii}=(a'_i,a_i)=\frac{(a_i,a_i)'}{2}=\frac{\epsilon'_i}{2}=0
$$
holds for $i=1,2,3$, and we have
$$
\epsilon_j k_{ji}=(a'_i,a_j)=(a_i,a_j)'-(a_i,a'_j)=-(a_i,a'_j)=-\epsilon_i k_{ij}
$$
for $i\ne j$. In particular,  we have $-\epsilon_i\epsilon_j k_{ij}=k_{ji}$,
where $k_{10}:=\kappa_1$, $k_{20}:=\kappa_2$ and $k_{21}:=\Omega$.

Since $f$ is a generalized cuspidal edge along $\Gamma$,
we can write (cf. Proposition \ref{prop:fstE} in the appendix)
\begin{equation}\label{eq:f1452}
f(s,t)=\Gamma(s)+X(s,t)\mb a_1(s)+Y(s,t)\mb a_2(s).
\end{equation}
Then the plane curve defined by
$\mb c_s(t):=(X(s,t),Y(s,t))$
gives a cusp at $t=0$ for each $s\in I$. 
In particular,  $X(s,t)$ and $Y(s,t)$ can be written in the following form
\begin{align}\label{eq:X1464}
X&=\cos \theta A+\sin\theta B,\qquad Y=-\sin \theta A+\cos \theta B, \\
\label{eq:AB1468}
A(s,t)&=\sum_{i=0}^m \frac{\alpha_i(s)}{i!} t^i +O(t^{m+1}),\qquad
B(s,t)=\sum_{i=0}^m \frac{\beta_i(s)}{i!}t^i +O(t^{m+1}), 
\end{align}
where $O(t^{m})$ is a function written as $t^m\phi(s,t)$  
using a $C^\infty$-function $\phi$ defined on a neighborhood of $(s,0)$.
In this setting, 
\begin{align}
\label{eq:ab1474}
&\alpha_0(s)=\alpha_1(s)=0,\,\, \alpha_2(0)\ne 0,\\
\qquad 
&\beta_0(s)=\beta_1(s)=\beta_2(s)=0. \nonumber
\end{align}
We remark that 
the angular function
$\theta$ is a $C^\infty$-function of one variable 
which vanishes identically if $f$ is of type $\op{E}$, $\op{T}$, $\op{S}_s$ or $\op{S}_t$
like as Example \ref{exa:E1}.
We will use the following asymptotic expansion of
a given smooth function $\mu(s,t)$ with respect to $t$
\begin{equation}\label{eq:mu2029}
\mu(s,t)=\mu_0(s)+\mu_1(s)t+\frac{\mu_2(s)}{2!}t^2+\frac{\mu_3(s)}{3!}t^3+O(t^4) \qquad (s\in I),
\end{equation}
where $\mu_i(s)$ ($s=1,2,3$) are smooth functions on $I$ as follows:
If we think of $\mb c_s$ as lying in the Euclidean plane $\E^2$,
we can write (cf. Fact \ref{fact:SU})
\begin{equation}\label{eq:AB1498}
\pmt{A(s,t)\\
B(s,t)}
=
\int_0^t u\pmt{\cos \lambda(s,u)\\
\sin \lambda(s,u)} du,
\qquad
\lambda(s,t):=\int_0^t \mu(s,u)du.
\end{equation}
In this case,  we have
\begin{equation}\label{V1:1508}
\pmt{A \\
B}
=
\frac12\pmt{1\\ 0} t^2+
\frac13\pmt{0\\ \mu_0} t^3
+\frac18\pmt{-\mu_0^2\\ \mu_1} t^4
+\frac1{30}\pmt{-3 \mu_0\mu_1\\ -\mu_0^3 + \mu_2} t^5
+O(t^6)
\end{equation}
and
\begin{align}\label{eq:ab1519}
&\alpha_2=1,\quad \alpha_3=0,\quad \alpha_4=-3 \mu_0^2,\quad \alpha_5=-12 \mu_0\mu_1,  \\
&\beta_3=2\mu_0,\quad \beta_4=3\mu_1,\quad \beta_5=4(-\mu_0^3+\mu_2).  \nonumber
\end{align}
On the other hand, if we  think of $\mb c_s$ 
as lying in the Lorentz-Minkowski plane $\Lo^2$, we can write (cf. \eqref{prop:SU2})
\begin{equation}\label{eq:AB1525}
\pmt{A(s,t)\\
B(s,t)}
=
\int_0^t u\pmt{\cosh \lambda(s,u)\\
\sinh \lambda(s,u)} du,
\qquad
\lambda(s,t):=\int_0^t \mu(s,u)du.
\end{equation}
In this case,  we have
\begin{equation}\label{V1:1536}
\pmt{A \\
B}
=
\frac12\pmt{1\\ 0} t^2+
\frac13\pmt{0\\ \mu_0} t^3
+\frac18\pmt{\mu_0^2\\ \mu_1} t^4
+\frac1{30}\pmt{3 \mu_0\mu_1 \\ \mu_0^3 + \mu_2} t^5
+O(t^6),
\end{equation}
which implies that
\begin{align}\label{eq:ab1547}
&\alpha_2=1,\quad \alpha_3=0,\quad \alpha_4=3 \mu_0^2,\quad \alpha_5=12 \mu_0\mu_1,  \\
&\beta_3=2\mu_0,\quad \beta_4=3\mu_1,\quad \beta_5=4(\mu_0^3+\mu_2).  \nonumber
\end{align}
Thus, in both of the two cases (i.e. \eqref{eq:ab1519}
and  \eqref{eq:ab1547}),
we may assume
\begin{equation}\label{eq:AB1534}
\alpha_3(s)=0 \qquad (s\in I),
\end{equation}
which is useful for simplifying future calculations.
We have divided the expressions of $A$ and $B$ into two cases as above, 
where the difference between formulas 
\eqref{eq:ab1519} and \eqref{eq:ab1547}
is simply the sign of $\alpha_4$ and $\beta_5$.
However, the coefficients $\alpha_4$, $\alpha_5$ and $\beta_5$ are not important,
since they are not appear in the terms of asymptotic expansion of
the matrix $\tilde W^E$ or $\tilde W^L$ that we will use in future discussions.
Therefore, in the following discussion, we need to pay 
little attention to the above case separation of the definitions of $A$ and $B$.

\begin{Remark}\label{Rem:GI}
Let $P(s)$ be the plane passing through $f(s,0)$ which is
spanned by two vectors $\mb a_1(s),\mb a_2(s)\in T_{f(s,0)}\R^3$, where
$T_{f(s,0)}\R^3$ is the tangent space of $\R^3$ at $f(s,0)$.
Then the section $C_s(\subset \R^3)$ of the image of $f$ by the 
plane $P(s)$ is a generalized cusp in the plane $P(s)$,
by Proposition B.1,  and the
map $\mb c_s(t):t\mapsto (X(s,t),Y(s,t))$ can be identified with
a parametrization of $C_s$.
In our following representation formulas of generalized cuspidal edges
$f$ of type (E), (T), ($\op{S}_s$) and ($\op{S}_t$),
the frame field $\{\mb a_0(s),\mb a_1(s),\mb a_2(s)\}$
is uniquely constructed by the same method for all 
generalized cuspidal edge $f$ of the same type, and
the image of  $\hat{\mb c}_s:t\mapsto (A(s,t),B(s,t))$ is
congruent to $C_s$ in the plane $P(s)$.
Thus, $\mu_i(s)$ ($s=1,2,3$) can be considered as invariants of
$\mb c_s(t)$, and so they can be also considered as
invariants of generalized cuspidal edge $f$.
\end{Remark}

The following example illustrates our setting:

\begin{Example}\label{exa:E0}
Let $\Gamma(s)$ be a regular 
curve in $\E^3$
parametrized by arc-length.
If the curvature function $\kappa(s):=|\Gamma''(s)|_E$ does not vanish,
then the principal normal vector $\mb n(s)$
and the bi-normal vector $\mb b(s)$
are defined along $\Gamma$.
Any generalized cuspidal edges $f(s,t)$ 
along $\Gamma$ can be expressed as
(here $X,Y$ are given by \eqref{eq:X1464})
\begin{equation}\label{rep:T2}
f(s,t)=\Gamma(s)+(\mb n(s), \mb b(s))
\pmt{
\cos \theta(s) & \sin \theta(s) \\
-\sin \theta(s) & \cos \theta(s)
}
\pmt{A(s,t) \\
B(s,t)},
\end{equation}
where 
$$
\pmt{A(s,t)\\ B(s,t)}:=\int_0^t u\pmt{\cos \lambda(s,u)\\ \sin \lambda(s,u)}du,\quad
\lambda:=\int_0^t \mu(s,u)du,
$$
and $\theta(s), \mu(s,u)$ are 
arbitrarily given smooth functions
(cf. \eqref{eq:SU2}  in the appendix).
This is just the formula of the first author \cite{F},
which enables us to produce 
generalized cuspidal edges from two geometric 
data $\theta,\mu$ (the function $\theta(s)$ is called the
{\it cuspidal angle} which is the angle from $\mb n(s)$ to the cuspidal direction vector
$\mb D^E_f(s)$). Each coefficient of Taylor expansion of this formula can be
considered as geometric invariants of $f$, since
$s$ is the arc-length parameter of $\Gamma$
and $t$ is the normalized half-arc-length parameter of sectional cusps. 
We set 
\begin{align*}
\mb a_0:=\Gamma',\quad
\mb a_1:=\mb n,\quad \mb a_2:=\mb b,\quad
\epsilon_0=\epsilon_1=\epsilon_2=1,\quad
\kappa_1:=\kappa
,\,\, \kappa_2:=0,\,\,
\Omega:=\tau,
\end{align*}
where $\kappa$ is the curvature function and
$\tau$ is the torsion of the curve $\Gamma$.
Then, 
\eqref{eq:a1434} is reduced to the classical Frenet equation 
for regular space curves, and
the formula \eqref{rep:T2} coincides with
\eqref{eq:f1452}.
\end{Example}

The above formula \eqref{rep:T2} is quite useful to 
construct concrete examples without solving any ordinary differential equations. 
However, if the curvature function $\kappa(s)$ vanishes
for some $s$, it cannot make sense, since 
$\mb n(s)$ and $\mb b(s)$
are not defined along $\Gamma$.
In this case, we can give another formula
producing all of generalized cuspidal edges along $\Gamma$ as follows:

\begin{Example}\label{exa:E1}
Let $\Gamma(s)$ be a regular curve in $\E^3$
parametrized by the arc-length.
We let $f$ be a generalized cuspidal edge along $\Gamma$,
which is written as in \eqref{eq:f1452}.
Then we may set (cf. \eqref{eq:DE1101})
$
\mb a_1(s):=D^E_f(s,0),
$
which is a unit vector and is perpendicular to $\mb a_0:=\Gamma'$. 
By setting
$$
\mb a_2(s):=\mb a_0(s) \times \mb a_1(s) (=\nu^L(s,0)),
$$
$\mb a_0,\,\mb a_1,\mb a_2$ give an orthonormal frame field
along $\Gamma$. Moreover, since $f$ lies in $\E^3$
and $\mb a_1$ is the cuspidal direction vector, by setting
$\theta=0$, we can write
(in this setting $X=A$ and $Y=B$ hold by \eqref{eq:X1464})
\begin{equation}\label{eq:f1859}
f(s,t)=\Gamma(s)+A(s,t)\mb a_1(s)+B(s,t)\mb a_2(s)
\end{equation}
and
\begin{equation}\label{eq:XY1955}
\pmt{A(s,t)\\
B(s,t)}
=
\int_0^t \pmt{\cos \lambda(s,u)\\
\sin \lambda(s,u)} du,
\qquad
\lambda(s,t):=\int_0^t \mu(s,u)du,
\end{equation}
where the function $\mu(s,t)$ can be considered as a
geometric invariant of $f$. We set
\begin{align*}
\epsilon_0=\epsilon_1=\epsilon_2=1,\quad
\theta=0,\quad
\kappa_1:=\kappa_s^E=\Gamma''\cdot \mb a_1,
\quad \kappa_2:=\kappa_\nu^E=\Gamma''\cdot \mb a_2,\quad
\Omega:=-\omega^E.
\end{align*}
Then \eqref{eq:a1434}
can be written as
\begin{equation}\label{eq:2049}
\mathcal F'=\mathcal F\mathcal K, \qquad
\mathcal K=\pmt{
0 & -\kappa^E_s& -\kappa^E_\nu \\
\kappa^E_s & 0& \omega^E \\
\kappa^E_\nu & -\omega^E & 0
},
\qquad\mathcal F:=(\mb a_0,\,\mb a_1,\mb a_2).
\end{equation}
The function $\omega^E$
is called the {\it cusp-directional torsion} of $f$ in $\E^3$ (cf. \cite{MS}),
which is an invariant of a generalized cusp $f$
(in fact, $\omega^E=\tau-\theta'$ holds, where $\tau$
is the torsion of $\Gamma$ and $\theta$ is the cuspidal angle
defined in  Example \ref{exa:E0}).

Conversely, if one gives data consisting of four functions 
$$
\Big(\kappa_s^E(s),\kappa_\nu^E(s),\omega^E(s),\mu(s,t)\Big)
=
\Big(k_1(s),k_2(s),\omega(s),m(s,t)\Big)
$$
defined on an open interval containing $s=0$,
then, we solve the ordinary differential equation \eqref{eq:2049}
with $\mathcal F$ as an unknown matrix-valued function
under the initial condition that $\mathcal F(0)$ is the $3\times 3$ identity matrix.
By setting $\mathcal F(s)=(\mb a_0,\mb a_1,\mb a_2)$,
$\Gamma(s):=\int_{0}^s \mb a_0(u)du$ and  $A,B$ as in \eqref{eq:XY1955},
then the map $f$ given by \eqref{eq:f1859}
is a generalized cuspidal edge along the curve $\Gamma$  
whose singular curvature, limiting normal curvature,
cusp-directional torsion and the $\mu$-function are
$\kappa_s^E$, $\kappa_\nu^E$, $\omega^E$ and $\mu$,
respectively.
An advantage of this formula is that we can produce 
all of generalized cuspidal edges along $\Gamma$ even when the curvature 
function of $\Gamma$ admits zeros. 
\end{Example}

\subsection{Computation in the general setting}
For the sake of simplicity, we write
$$
S(s):=\sin \theta(s),\qquad C(s):=\cos \theta(s).
$$
By \eqref{eq:f1452}, we set
\begin{align}
f_s&=\mb a_0+\frac{
\epsilon_0(-\epsilon_1 \kappa_1 C+\epsilon_2 \kappa_2 S)\mb a_0
+(\epsilon_1\epsilon_2 \Omega-{\theta'})S\mb a_1
+(\Omega-{\theta'})C\mb a_2
}{2}t^2 \\ 
\nonumber
&\phantom{aaa}
-\frac{
\epsilon_0\beta_3
\left(\epsilon_1 \kappa_1 S+\epsilon_2 \kappa_2 C\right)\mb a_0}6 t^3 \\
&\phantom{aaaaaa}+ \nonumber
\frac{
\left(\beta_3(-\epsilon_1\epsilon_2\Omega+\theta')C+\beta'_3S\right)\mb a_1
+(\beta_3 (\Omega-\theta ')S+\beta'_3 C)\mb a_2
}6t^3  +O(t^4),\\
f_t&=(\mb a_1 C-S \mb a_2)t+\frac{\beta_3(S\mb a_1+C \mb a_2)}2 t^2 \\
&\phantom{aaaaaa}+ \nonumber
\frac{(\alpha_4C+\beta_4 S)\mb a_1+(-\alpha_4 S+\beta_4 C)\mb a_2}{6}t^3+O(t^4),
\end{align}
which imply that
\begin{align}\label{eq:E1577}
(f_s,f_s)&=\epsilon_0-(\epsilon_1 \kappa_1C-\epsilon_2 \kappa_2 S)t^2-
\frac{
\beta _3 
\left(\epsilon_1 \kappa_1 S+\epsilon_2 \kappa _2C\right)
}{3}t^3+O(t^4), \\
\label{eq:F1583}
(f_s,f_t)&=
\frac{1}{2} CS \big((-\epsilon_1+\epsilon_2)\theta'\big)t^3
+O\left(t^4\right), \\
\label{eq:G1588}
(f_t,f_t)&=
\left(\epsilon_1C^2+\epsilon_2 S^2 \right)t^2
+\left(\epsilon_0-\epsilon_2\right) \beta_3 C S t^3 +O\left(t^4\right).
\end{align}
If we set
\begin{equation}
\Delta:=(f_s,f_s)(f_t,f_t)-(f_s,f_t)^2,
\end{equation}
then, we have
\begin{equation}\label{eq:D1594}
\Delta=
\epsilon _0 \left(\epsilon _1 C^2+
\epsilon _2 S^2\right)t^2+
\epsilon _0 \left(\epsilon _1-\epsilon _2\right) \beta_3 C S t^3+O\left(t^4\right)
\end{equation}
and
\begin{equation}\label{eq:N1601}
\tilde \nu_E=f_s\times_E f_v=
(
S\mb a_1+C \mb a_2
)t
+
\frac{\beta_3}2(
-C \mb a_1+S \mb a_2
)t^2+O(t^3).
\end{equation}
We next compute the second derivatives of $f$. 
We  have
\begin{align}\label{eqA2149}
&f_{ss}
=\kappa _1\mb a_1 +\kappa _2 \mb a_2 
\\
\nonumber
&\phantom{aa}
-\frac{\epsilon_0}2 \biggl(
\Big(\epsilon_2 \kappa _2 \left(\Omega-2 \theta'\right)
+\epsilon_1 \kappa _1'\Big)C 
+
\Big(\kappa _1 (\epsilon_2 \Omega-2 \epsilon_1\theta')
-\epsilon_2 \kappa _2'\Big)S
\biggr)
\mb a_0t^2 \\ \nonumber
&\phantom{aa}
\small{
-\frac{
\Big(\epsilon_1\epsilon_2 \Omega \left(\Omega-2 \theta'\right)+\theta'^2
+\epsilon_0\epsilon_1 \kappa_1^2
\Big)C
+
\Big( 
\theta''-\epsilon_0\epsilon_2 \kappa_1 \kappa_2-\epsilon_1\epsilon_2 \Omega'
\Big)S
}2\mb a_1 t^2
} \\ \nonumber
&\phantom{aa}
\small{
+
\frac{
\Big(-\theta''-\epsilon_0\epsilon_1 \kappa_1 \kappa _2+\Omega'\Big)C
+\Big(
-2 \Omega \theta'+\theta'^2+\epsilon_0\epsilon_2 \kappa_2^2+\epsilon_1\epsilon_2 \Omega^2
\Big)S
}2
\mb a_2 t^2
}
\\ \nonumber
&\phantom{aaaaaaaaaaaaaaaaaaaaaaaaaaaaaaaaaaaaaaaaaaaaaaaaaaaaaa}
+O(t^3), 
\end{align}
and
\begin{align}
f_{st}&= \label{eqB2149}
\biggl(
\epsilon_0\Big(
-\epsilon_1 \kappa_1 C+\epsilon_2 \kappa_2 S
\Big)\mb a_0
+
\Big(
\epsilon_1\epsilon_2\Omega-\theta' 
\Big)S \mb a_1
+
\Big(
\Omega-\theta'
\Big)C \mb a_2
\biggr)t \\ \nonumber
&\phantom{aaa}+
\frac12\biggl(
-\epsilon_0\beta_3
\left(\epsilon_2 \kappa_2C+\epsilon_1 \kappa _1 S\right)\mb a_0
+
\Big(
\beta_3\left(\theta '-\epsilon_1\epsilon_2 \Omega\right)C
+\beta'_3S
\Big)\mb a_1 \\ \nonumber
&\phantom{aaasss}+
\Big (\beta'_3 C+\beta_3 (\Omega-\theta ')S\Big)\mb a_2
\biggr)t^2
+O(t^3), \\
f_{tt}&=\label{eqC2149}
C \mb a_1-S \mb a_2+
\beta_3(S \mb a_1+C \mb a_2)t \\
\nonumber
&\phantom{aaaa}+\frac12
\Big(
(\alpha_4C +\beta_4S)\mb a_1
+
(\beta_4 C-\alpha_4 S)\mb a_2
\Big)t^2+O(t^3).
\end{align}
By
\eqref{eq:Dp262}, we have
\begin{align}\label{eq:dp1752-0}
d^C((s,0))&=-\det(f_s,f_{ss},f_{tt})\Big |_{t=0} \\
\nonumber
&=-\det(
\mb a_0, \kappa_1 \mb a_1+\kappa_2 \mb a_2,
C\mb a_1-S \mb a_2
)=\kappa _1 S +\kappa _2 C.
\end{align}
Using 
\eqref{eq:N1601}
and \eqref{eqA2149},
we  have 
\begin{align}\label{eql:1760}
\tilde L&=(S \kappa_1+C \kappa_2)t
+\frac{\beta_3}2
\left(-\kappa_1 C+\kappa_2 S\right)t^2 \\ \nonumber
&\phantom{aaa}+
\frac16\biggl(
\left(-\beta_4\kappa_1+\alpha_4\kappa_2\right)C 
+\left(\alpha_4\kappa_1+\beta_4\kappa_2\right)S\biggr)t^3 \\ \nonumber
&\phantom{aaaaaaaa}
+\Big(
\left(\epsilon_1\epsilon_2-1\right) \Omega \theta '-\epsilon_0\epsilon_1 \kappa_1^2+
\epsilon_0\epsilon_2 \kappa_2^2
\Big)CS t^3 \\ \nonumber
&\,\phantom{aaassaaasaa}+\frac12\Big(
-\theta''-2\epsilon_0\epsilon_1 \kappa_1 \kappa_2+\Omega'
\Big)C^2 t^3 \\
&\phantom{aaaaaaaaaaaaaa}+ \nonumber
\frac12\Big(
-\theta''+2 \epsilon_0\epsilon_2 \kappa _1 \kappa _2+\epsilon_1\epsilon_2 \Omega '
\Big)S^2t^3. 
\end{align}
Similarly, using 
\eqref{eq:N1601}, \eqref{eqB2149}, \eqref{eqC2149}
with the relation $C^2+S^2=1$, we have
\begin{align}\label{eq:M1782}
\tilde M&=
\Big(
(\Omega(C^2+\epsilon_1 \epsilon_2S^2)-\theta'\Big) t^2 
+
\frac12\Big(
2(1-\epsilon_1\epsilon_2)\beta_3 \Omega CS+\beta'_3 
\Big)t^3, \\ \label{eq:N1789}
\tilde N&=
\frac{\beta_3}2 t^2+\frac{\beta_4}3 t^3+O(t^4).
\end{align}

Let $U$ be a domain of $\R^2$ and
$
f:U\to \R^3
$
a $C^\infty$-map.  A point $p\in U$ is called a {\it cuspidal cross cap singular point} of $f$
if there exist local  diffeomorphisms $\phi$ and $\Psi$
on $\R^2$ and $\R^3$ satisfying 
$$
\phi(p)=(0,0), \quad \Psi\circ f(p)=(0,0,0),
\quad\Psi\circ f\circ \phi^{-1}(s,t)=f_2,
$$
where $f_2$
is the standard map of cuspidal cross caps as in Example \ref{ex:1195}.
Cuspidal cross cap singular points are known as a typical singular points
appeared on generalized cuspidal edges as well as cuspidal edge singular points.
The following is an application of the previous calculations.

\begin{Lemma}\label{lem:1688}
The point $(s_0,0)$ $(s_0\in I)$ is a  cuspidal edge singular point 
$($resp. a cuspidal cross cap  singular point$)$ of
$f$ if and only if $\mu_0(s_0) \ne 0$ $($resp. $\mu_0(s_0)=0$ 
and $\mu'_0(s_0)\ne 0)$.
\end{Lemma}

\begin{proof}
Let $P_0$ be the plane passing through $\Gamma(s_0)$
spanned by $\mb a_1(s_0)$ and $\mb a_2(s_0)$.
Then $\mu_0(s_0) \ne 0$ if and only if
the section of the image of $f$ by the plane $P_0$
is a cusp. So the conclusion follows from
Proposition \ref{lem:ce708}.

On the other hand $(s_0,0)$ ($s_0\in I$) is a cuspidal cross cap singular point
if and only if  
$\phi(s_0)=0$ and $\phi'(s_0)\ne0$ (cf. \cite{FSUY} or \cite[Section 2]{SUY2}),
where
$$
\phi(s):=\det(\Gamma'(s),\check \nu_E(s,0),(\check \nu_E)_t(s,0)) \qquad (s\in I)
$$
and (cf. \eqref{eq:N1601})
\begin{equation}\label{eq:nuE}
\check \nu_E:=\frac{\hat \nu_E}{t}=
(
S\mb a_1+C \mb a_2
)
+
\frac{\beta_3}2(
-C \mb a_1+S \mb a_2
)t+O(t^2)
\end{equation}
gives a non-vanishing normal vector field along $f$. We set 
$$
\mb w_1:=S\mb a_1+C \mb a_2,\qquad \mb w_2=-C \mb a_1+S \mb a_2.
$$
Since $\Gamma'=\mb a_0$, three vectors
$\Gamma'(s), \mb w_1(s)$ and $\mb w_2(s)$ are linearly independent
for each $s\in I$. So we have
\begin{align*}
\phi(s)&=
\frac{\beta_3(s)}{2}\det(\Gamma'(s),\mb w_1(s),\mb w_2(s))=
\mu_0(s) \psi(s), \\
\psi(s)&=\det(\Gamma'(s),\mb w_1(s),\mb w_2(s)).
\end{align*}
Since $\psi(s)>0$ for $s\in  I$,
the condition $\phi(s_0)=0$ is equivalent to $\mu_0(s_0)=0$. 
Moreover, if $\mu_0(s_0)=0$, then we have
$$
\phi'(s_0)=
\mu'_0(s_0)\psi(s_0),
$$
which proves the assertion for cuspidal cross caps.
\end{proof}

\subsection{Generalized cuspidal edges of type E}
Let $\Gamma:I\to \E^3$ 
be a regular curve parametrized by the arc-length.
We use the setting given in
Example \ref{exa:E1}, that is,
$$
\epsilon_0=\epsilon_1=\epsilon_2=1,\quad
\theta=0,\quad
\kappa_1=\kappa_s^E,\quad \kappa_2=\kappa_\nu^E.
$$
Then, we have $C=1$, $S=0$ and $(X,Y)=(A,B)$. 
Since $(A,B)$ satisfies \eqref{eq:ab1519},
the equations \eqref{eq:E1577}, \eqref{eq:F1583}, \eqref{eq:G1588} 
and \eqref{eq:D1594} imply
\begin{equation}\label{eq:EFG1821}
E=1-\kappa_s^E t^2-\frac{\beta_3\kappa_\nu^E}{3}t^3+O(t^4),\quad 
F=O(t^4),\quad G=t^2+O(t^4)
\end{equation}
and
\begin{equation}\label{eq:D1836}
\Delta_E(=\Delta)=t^2+O(t^4).
\end{equation}
Moreover,
\eqref{eq:dp1752-0} implies that (we set $p:=(s,0)$)
\begin{equation}\label{eq:dp1752}
\op{sgn}(d^C(p))=\op{sgn}(\kappa_\nu^E(s)). 
\end{equation}
On the other hand,
\eqref{eql:1760}, \eqref{eq:M1782} and
\eqref{eq:N1789} are reduced to
\begin{align}\label{eq:LMN1833}
\tilde L&=\kappa_\nu^E t-\frac{\beta_3\kappa_s^E}2 t^2+
\frac16(-\beta_4\kappa_s^E+\alpha_4 \kappa_\nu^E-6\kappa_s^E\kappa_\nu^E
-3(\omega^E)')t^3
+O(t^4),\\ \nonumber
\tilde M&=
-\omega^E t^2 +\frac{\beta'_3}{2} t^3 +O\left(t^4\right),\quad 
\tilde N=
\frac{1}{2} t^2 \beta_3+\frac{1}{3} t^3 \beta_4+O\left(t^4\right).
\end{align}
In this situation, we may assume $\alpha$ and $\beta$ can be
written as \eqref{eq:ab1519}. 
As pointed out in \cite{HNSUY} (see also Appendix B), 
$\mu_0(s),\mu_1(s)$ and $\mu_2(s)$ are geometric invariants 
of the generalized cuspidal edges in $\E^3$. 
By \eqref{eq:EFG1821} and \eqref{eq:LMN1833}, we have 
\begin{align*}
(\tilde W^E&:=)\, \pmt{G & -F \\ -F& E}\pmt{\tilde L & \tilde M \\ \tilde M & \tilde L} \\
&=
\begin{pmatrix} 
\kappa_{\nu}^E t^3 -\frac{\beta_3\kappa_{s}^E}{2} t^4 +O\left(t^5\right) &
 -\omega^E t^4 +O\left(t^5\right) \\
 -\omega^E t^2+O\left(t^3\right) &
 \frac{1}{2} t^2 \beta_3+\frac{1}{3} t^3 \beta_4+O\left(t^4\right) 
\end{pmatrix} \\
&=
\begin{pmatrix} 
\kappa_{\nu}^E t^3 -\kappa_s^E\mu_0 t^4 +O\left(t^5\right) 
& -\omega^E t^4 +O\left(t^5\right) 
\\
-\omega^E t^2 +O\left(t^3\right) 
& \mu_0 t^2+\mu_1 t^3+O\left(t^4\right) 
\end{pmatrix},
\end{align*}
and then asymptotic behaviors of the Gaussian curvature 
and the mean curvature are given by
\begin{equation}\label{eq:KET957}
K^E=
\frac{\kappa_\nu^E \mu_0}{t}-\kappa_s^E \mu_0^2+\kappa_\nu^E \mu_1-(\omega^E)^2+
O(t), \quad
H^E =\frac{\mu_0}{2t}+\frac{\mu_1+\kappa_\nu^E}{2}+O(t).
\end{equation}
Moreover, since $\tilde W^E$ is a triangle matrix modulo $O(t^4)$-term,
by \eqref{eq:D1836}, 
 the diagonal components of $\tilde W^E$ divided by $t^3$
give the following asymptotic expansions of the principal curvatures of $f$:
\begin{equation}\label{eq:2b}
\lambda_1=
\frac{\mu_0}{t}+\mu_1+O(t),\qquad
\lambda_2=
\kappa_{\nu}^E +O(t).
\end{equation}
Summarizing these computations, the following facts are obtained, which have been
already known (cf. \cite{F}, \cite{MSUY}, \cite{SUY} and \cite{MU}).

\begin{Fact}\label{cor:E}
Let $f:\mc U_f\to \E^3$ be a generalized cuspidal edge along a
regular curve $\Gamma$. Then the following assertions hold:
\begin{enumerate}
\item The sign $\sigma^C(p)$ $(p:=(s,0))$ coincides 
with the sign of $\kappa_\nu^E(s)$.
\item The mean curvature $H^E$ is unbounded near $p\in \Sigma_f$
if  $p$ is a cuspidal edge singular point.
\item If $p:=(s,0)$ is a cuspidal edge of non-vanishing limiting
normal curvature, then 
the Gaussian curvature  $K^E$
is unbounded near $p$ and takes 
different signs on each side of $\Sigma_f$.
\item The two principal curvatures  of $f$ 
are real-valued on $\mc U_f\setminus \Sigma_f$, and
one of them is bounded.
Moreover, if $p$ is a cuspidal edge singular point, the
other principal curvature is unbounded.
\item If $\mu_0(s)\ne 0$ $($that is, $p$ is a cuspidal edge singular point$)$,
the umbilical points of $f$ never accumulate at $p$. 
\end{enumerate}
\end{Fact}

\begin{proof}
(1) follows from \eqref{eq:dp1752}.
Suppose that $(s,0)$ is a cuspidal edge singular point.
Since $\mu_0(s)\ne 0$ (cf. Lemma \ref{lem:1688}),
the second equation of \eqref{eq:KET957}
implies that  $H^E$ is unbounded, proving (2).
On the other hand, (3)  follows from 
the first equation of
\eqref{eq:KET957},
since 
$\mu_0\ne 0$ and $\kappa^L_\nu\ne0$.
Finally (4) and (5) follow from \eqref{eq:2b}. 
\end{proof}

\subsection{Generalized cuspidal edges of type T}
Consider a  time-like regular curve $\Gamma:I\to \Lo^3$.
Let $f$ be a generalized cuspidal edge along $\Gamma$.
Since $\Gamma$ is time-like, it can be parametrized by arc-length, and can set
(cf. \eqref{eq:DL1386})
\begin{equation}\label{eq:a0a1}
\mb a_0(s):=\Gamma'(s),\qquad
\mb a_1(s):={\mb D}^L_f(s,0),
\end{equation}
which are both unit vector fields. By setting
\begin{equation}\label{eq:a2}
\mb a_2(s):=\mb a_0(s) \times_L \mb a_1(s) (=\nu^L(s,0)),
\end{equation}
$(\mb a_0,\,\mb a_1,\mb a_2)$ is an orthonormal frame field of $\Lo^3$
along $\Gamma$.  By \eqref{eq:f1452} and \eqref{eq:X1464},
we can write
\begin{equation}\label{rep:T}
f(s,t)=\Gamma(s)+A(s,t)\mb a_1(s)+B(s,t)\mb a_2(s).
\end{equation}
Since $\mb a_1(s)$ points in the 
$\Lo^3$-cuspidal direction at $f(s,0)$, we may set
$\theta(s)=0$. Since the normal plane of $\Gamma$ is 
space-like, we can write
$$
\pmt{A(s,t)\\
B(s,t)}
=
\int_0^t u\pmt{\cos \lambda(s,u)\\
\sin \lambda(s,u)} du,
\qquad
\lambda(s,t):=\int_0^t \mu(s,u)du,
$$
and the function $\mu(s,t)$ can be considered as a
geometric invariant of $f$ (cf. Remark~\ref{Rem:GI}),
that is, 
\eqref{eq:ab1519}
 holds.
This corresponds to the case that 
$$
\epsilon_0=-1,\quad\epsilon_1=\epsilon_2=1,\quad 
C=1,\quad S=0.
$$
By \eqref{eq:knSL}, we have
$\kappa_1=\kappa_s^L$ and $\kappa_2=\kappa_\nu^L$.
The function
$
\omega^L:=-\Omega
$
should be called the {\it cusp-directional torsion} of $f$ in $\Lo^3$.
Then
\eqref{eq:a1434}
can be written as
\begin{equation}\label{eq:2449}
\mathcal F'=\mathcal F\mathcal K, \qquad
\mathcal K=\pmt{
0 & \kappa^L_s& \kappa^L_\nu \\
\kappa^L_s & 0& \omega^L \\
\kappa^L_\nu & -\omega^L & 0
},
\quad\mathcal F:=(\mb a_0,\,\mb a_1,\mb a_2).
\end{equation}

\begin{Remark}\label{rmk:rep1991}
If one gives data consisting of four functions 
$$
(\kappa^L_s,\kappa^L_\nu,\omega^L,\mu):=(k_1,k_2,\Omega,m)
$$
defined on an open interval containing $s=0$,
then, like as in Example \ref{exa:E1}, 
a generalized cuspidal edge of type $T$ is obtained
by solving the ordinary differential equation
\eqref{eq:2449}
with $\mathcal F$ as an unknown matrix-valued function.
\end{Remark}

Here,
\eqref{eq:E1577}, \eqref{eq:F1583} and \eqref{eq:G1588} are reduced to
\begin{equation}\label{eq:EFG1957}
E^L=-1-\kappa_s^L t^2-\frac{\beta_3\kappa_\nu^L}{3}+O(t^4),\quad 
F^L=O(t^4),\quad G^L=t^2+O(t^4)
\end{equation}
and \eqref{eq:D1594} and \eqref{eq:dp1752-0}
imply 
\begin{equation}\label{eq:dp1967}
\Delta_L=-t^2+O(t^4),\qquad 
d^C(p)=\kappa_\nu^L. 
\end{equation}
On the other hand,
\eqref{eql:1760},
\eqref{eq:M1782} and
\eqref{eq:N1789} are reduced to
\begin{align}\label{eq:LMN1974}
\tilde L&=\kappa_\nu^L t-\frac{\beta_3\kappa_s^L}2 t^2+
\frac16(-\beta_4\kappa_s^L+\alpha_4 \kappa_\nu^L
+6\kappa_s^L\kappa_\nu^L-3(\omega^L)')t^3
+O(t^4),\\ \nonumber
\tilde M&=
-\omega^L t^2 +\frac{\beta'_3}{2} t^3 +O\left(t^4\right),\quad 
\tilde N=
\frac{1}{2} t^2 \beta_3+\frac{1}{3} t^3 \beta_4+O\left(t^4\right).
\end{align}
In this situation, we may assume $\alpha$ and $\beta$ are
written as \eqref{eq:ab1547}. Here $\mu_0(s),\mu_1(s)$ 
and $\mu_2(s)$ are geometric invariants of 
generalized cuspidal edge $f$ (cf. Remark \ref{Rem:GI}). 
By \eqref{eq:EFG1957} and \eqref{eq:LMN1974},
we have 
\begin{align*}
\tilde W^L&=
\begin{pmatrix} 
\kappa_{\nu}^L t^3 -\frac{\beta_3\kappa_{s}^L}{2} t^4 +O\left(t^5\right) &
-\omega^L t^4 +O\left(t^5\right) \\
\omega^L t^2+O\left(t^3\right) &
-\frac{1}{2} t^2 \beta_3-\frac{1}{3} t^3 \beta_4+O\left(t^4\right) 
\end{pmatrix} \\
&=
\begin{pmatrix} 
\kappa_{\nu}^L t^3 -\kappa_s^L\mu_0 t^4 +O\left(t^5\right) 
& -\omega^L t^4 +O\left(t^5\right) 
\\
\omega^L t^2 +O\left(t^3\right) 
& -\mu_0 t^2-\mu_1 t^3+O\left(t^4\right) 
\end{pmatrix},
\end{align*}
and then the asymptotic behavior of
the Gaussian curvature and the
mean curvature are given by
\begin{align}\label{eq:KET2008b}
K^L&=
-\frac{\kappa_\nu^L \mu_0}{t}+\kappa_s^L \mu_0^2-\kappa_\nu^L \mu_1
-(\omega^L)^2+
O(t), \\
\label{eq:KET2008bb}
H^L &=\frac{\mu_0}{2t}+\frac{\mu_1-\kappa_\nu^L}{2}+O(t).
\end{align}
Moreover, since $\tilde W^L$ is a triangle matrix modulo $O(t^4)$-term,
the diagonal components of $\tilde W^L$ divided by $-t^3$
give the following asymptotic expansions of 
the two principal curvatures of $f$:
\begin{equation}\label{eq:LL2017}
\lambda_1=
\frac{\mu_0}{t}+\mu_1+O(t),\qquad
\lambda_2=
-\kappa_{\nu}^L +O(t).
\end{equation}
Summarizing these computations,  we obtain the following:

\begin{Thm}\label{thm:T}
Let $f:\mc U_f\to \Lo^3$ be a 
generalized cuspidal edge along a
regular curve $\Gamma$ of type $T$.
Then, for each $p:=(s,0)$ $(s\in I)$,
there exists a neighborhood $V_p$
satisfying the following properties:
\begin{enumerate}
\item $f$ is time-like at $p$.
\item The sign $\sigma^C(p)$ coincides with the sign of $\kappa_\nu^L(s)$,
and the order $i_p$ is $2$.
\item 
If $p$ is a cuspidal edge singular point with $\sigma^C(p)\ne 0$,
then the mean curvature  $H^L$ is unbounded on $V_p$.
Moreover, the Gaussian curvature $K^L$
is unbounded and takes 
different signs on each side of $\Sigma_f$ around $p$.
\item 
If $p$ is a cuspidal edge singular point,
then the two principal curvature functions of $f$ 
are real-valued on $V_p$, and
one of the principal curvature function is bounded.
\item
Moreover, if $p$ is a cuspidal edge singular point,
then the other principal curvature function is unbounded.
In particular, the umbilical points and quasi-umbilical points
never accumulate at $p$.
\end{enumerate}
\end{Thm} 

\begin{proof}
Since $\Gamma'(s_0)$ is a time-like vector in $\Lo^3$, the limiting tangent plane
is time-like obviously, and (1) holds.
(2) follows from the second equation of
\eqref{eq:dp1967}.
We now assume that
$p(=(s,0))$ is a cuspidal edge singular point.
Since $\mu_0(s)\ne 0$ holds,
the second equation of \eqref{eq:KET2008b} implies that
$H^L$ is unbounded on a neighborhood of $p$.
By
\eqref{eq:KET2008b} and \eqref{eq:KET2008bb},
we have
$$
(H^L)^2-K^L=\frac{1}{4t^2}\Big(\mu_0^2+O(t)\Big).
$$ 
This with the fact $\mu_0(s)\ne 0$
implies that the two principal curvatures 
are real-valued. Moreover, if  
$\sigma^C(p)\ne 0$, 
then $\kappa_\nu^L(s)$ does not vanish,
which implies 
the coefficient $\kappa_\nu^L \mu_0$ of
the first term of \eqref{eq:KET2008b}
does not vanish.
So $K^L$ is unbounded and takes 
different signs on each side of $\Sigma_f$ around $p$.
We have proved  (3) and (4).
Finally (5) follows from  \eqref{eq:LL2017}.
\end{proof}

\section{Generalized cuspidal edges of type S}

In this section, we fix a generalized cuspidal edge in $\Lo^3$
along a space-like regular curve
$\Gamma:I\to \Lo^3$.
Let $f$ be a generalized cuspidal edge along  $\Gamma$.
Since $\Gamma$ is space-like, 
$\Gamma$ can be parametrized by the arc-length.

\subsection{Generalized cuspidal edge of types $\op{S}_s$ and $\op{S}_t$}

We first consider the case that ${\mb D}^L_f(s)$ does not point
in a light-like direction for any $s\in I$.
We then set
$$
\mb a_0(s):=\Gamma'(s),\qquad 
\mb a_1(s):={\mb D}^L_f(s,0)
$$
and
$$
\mb a_2(s):=\mb a_0(s) \times_L \mb a_1(s) (=\nu^L(s,0)).
$$
Then $\mb a_0,\,\mb a_1,\mb a_2$ give an orthonormal frame field in $\Lo^3$
along $\Gamma$.

Then \eqref{eq:f1452} and \eqref{eq:X1464} are reduced to
$$
f(s,t)=\Gamma(s)+A(s,t)\mb a_1(s)+B(s,t)\mb a_2(s).
$$
Since the image of 
$f$ lies in $\Lo^3$ and $\mb a_1$ points in the cuspidal direction, 
we may set
$\theta$ is identically zero.
Since the plane perpendicular to $\Gamma'$ is time-like, we can write
$$
\pmt{A(s,t)\\
B(s,t)}
=
\int_0^t u\pmt{\cosh \lambda(s,u)\\
\sinh \lambda(s,u)} du,
\qquad
\lambda(s,t):=\int_0^t \mu(s,u)du.
$$
In particular, 
\eqref{eq:ab1547}
 holds.
The function $\mu(s,t)$ can be considered as an
geometric invariant of $f$ (cf. Remark \ref{Rem:GI}).
Since ${\mb D}^L_f(s,0)$ points in a space-like 
(resp. time-like) direction, we can set
\begin{align*}
& \epsilon_0=1,\,\,\epsilon_1=\epsilon,\,\,
\epsilon_2=-\epsilon,\,\, C=1,\,\, S=0, \quad
\kappa_1=\kappa_s^L,\quad \kappa_2=\kappa_\nu^L,
\end{align*}
where
$$
\epsilon:=
\begin{cases}
1 & (\text{${\mb D}^L_f$ is space-like}), \\
-1 & (\text{${\mb D}_f$ is time-like}).
\end{cases}
$$
In particular, $\epsilon=1$ corresponds to the case $(\op{S}_s)$
and
$\epsilon=-1$ corresponds to the case 
$(\op{S}_t)$.
Thus $f$ is a generalized cuspidal edge of types 
$\op{S}_s$ (resp. $\op{S}_t$)
if $\epsilon>0$ (resp. $\epsilon<0$).
We set
$$
\omega^L:=-\Omega
$$
and call it the {\it torsional curvature} of $f$ 
along $\Sigma_f$.
Then
\eqref{eq:a1434}
is reduced to
\begin{equation}\label{eq:2640}
\mathcal F'=\mathcal F\mathcal K, \qquad
\mathcal K=\pmt{
0 & -\epsilon \kappa^L_s&  \epsilon\kappa^L_\nu \\
\kappa^L_s & 0&  \omega^L \\
\kappa^E_\nu & -\omega^L & 0},
\quad\mathcal F:=(\mb a_0,\,\mb a_1,\mb a_2).
\end{equation}

\begin{Remark}\label{rmk:rep2218}
If one gives data consisting of four functions 
$$
(\kappa^E_s,\kappa^E_\nu,\omega^E,\mu):=(k_1,k_2,\Omega,m)
$$
defined on an open interval containing $s=0$,
then, like as Example \ref{exa:E1},
solving the ordinary differential equation
\eqref{eq:2640} with $\mathcal F$ as the unknown matrix-valued function,
we can construct generalized cuspidal edges 
of type $\op{S}_s$ or $\op{S}_t$.
\end{Remark}

Here, \eqref{eq:E1577}, \eqref{eq:F1583} and \eqref{eq:G1588} are reduced to
\begin{equation}\label{eq:EFG1957b}
E^L=1-\epsilon \kappa_s^L t^2+\frac{\epsilon\beta_3\kappa_\nu^L}{3}t^3+O(t^4),\quad 
F^L=O(t^4),\quad G^L=\epsilon t^2+O(t^4)
\end{equation}
and \eqref{eq:D1594} implies
\begin{equation}\label{eq:2708}
\Delta_L=\epsilon t^2+O(t^4).
\end{equation}
Also,
\eqref{eq:dp1752-0}
implies 
\begin{equation}\label{eq:dp1967b}
\op{sgn}(d^C((s,0)))=\op{sgn}(\kappa_\nu^L(s,0)).
\end{equation}
On the other hand,
\eqref{eql:1760},
\eqref{eq:M1782} and
\eqref{eq:N1789} yield
\begin{align}\label{eq:LMN1974b}
\tilde L&=\kappa_\nu^L t-\frac{\beta_3\kappa_s^L}2 t^2+
\frac16\Big(-\beta_4\kappa_s^L+\alpha_4 \kappa_\nu^L-6\epsilon \kappa_s^L\kappa_\nu^L
-3(\omega^L)'\Big)t^3
+O(t^4),\\ \nonumber
\tilde M&=
-\omega^L t^2 +\frac{\beta'_3}{2} t^3 +O\left(t^4\right),\quad 
\tilde N=
\frac{1}{2} t^2 \beta_3+\frac{1}{3} t^3 \beta_4+O\left(t^4\right).
\end{align}
In this situation, we may assume $\alpha$ and $\beta$ can be
written as \eqref{eq:ab1547}. Then $\mu_0(s),\mu_1(s)$ 
and $\mu_2(s)$ are invariants of the generalized cuspidal edges in $\Lo^3$
(cf. Remark \ref{Rem:GI}). 
By \eqref{eq:EFG1957b} and \eqref{eq:LMN1974b},
we have 
\begin{align*}
\tilde W^L&=
\begin{pmatrix} 
\epsilon\kappa_{\nu}^L t^3 -\epsilon \frac{\beta_3\kappa_{s}^L}{2} t^4 +O\left(t^5\right) &
-\epsilon \omega^L t^4 +O\left(t^5\right) \\
-\omega^L t^2+O\left(t^3\right) &
\frac{1}{2} t^2 \beta_3+\frac{1}{3} t^3 \beta_4+O\left(t^4\right) 
\end{pmatrix} \\
&=
\begin{pmatrix} 
\epsilon\kappa_{\nu}^L t^3 -\epsilon\kappa_s^L\mu_0 t^4 +O\left(t^5\right) 
& -\epsilon \omega^L t^4 +O\left(t^5\right) 
\\
-\omega^L t^2 +O\left(t^3\right) 
& \mu_0 t^2+\mu_1 t^3+O\left(t^4\right)
\end{pmatrix}.
\end{align*}
In particular, the Gaussian curvature and the
mean curvature can be computed as
\begin{align}\label{eq:KET2008c}
K^L&=
-\frac{\kappa_\nu^L \mu_0}{t}+\Big(
\kappa_s^L \mu_0^2-\kappa_\nu^L \mu_1+(\omega^L)^2\Big)+
O(t), \\
\label{eq:KET2008cc}
H^L &=\frac{\epsilon \mu_0}{2t}+\frac{1}2\left(\epsilon \mu_1+\kappa^L_\nu\right)+O(t).
\end{align}
Since $\tilde W^L$ is a triangle matrix modulo $O(t^4)$-term,
the diagonal components of $\tilde W^L$ give the following
asymptotic expansion of the two principal curvatures of $f$;
\begin{equation}\label{eq:LL2017b}
\lambda_1(s,t)=
\frac{\epsilon \mu_0(s)}{t}+\epsilon \mu_1(s)+O(t),\qquad
\lambda_2(s,t)=
\kappa_{\nu}^L(s) +O(t).
\end{equation}

By Lemma \ref{lem:1688},
the singular point $(s,0)$ is a 
cuspidal edge singular point of
$f$ if and only if $\mu_0(s) \ne 0$. 
Summarizing these computations, we obtain the following:

\begin{Thm}\label{thm:S1}
Let $f:\mc U_f\to \Lo^3$ be a generalized cuspidal edge along $\Gamma$
of type $S$ such that its $\Lo^3$-cuspidal direction $\mb D^L_f$ never points 
in light-like directions along $\Gamma$.
Then, for each $p:=(s,0)$ $(s\in I)$, 
there exists a neighborhood $V_p(\subset \mc U_f)$
satisfying the following properties:
\begin{enumerate}
\item The sign $\sigma^C(p)$ coincides with the sign of $\kappa_\nu^L(s)$,
and the order $i_p$ is $2$.
\item
$f$ is space-like $($resp. time-like$)$ at $p$
if and only if $\mb D^L_f(s)$ $(s\in I)$
points in a space-like $($resp. time-like$)$
direction.
\item 
If $p$ is a cuspidal edge singular point,
then the mean curvature  $H^L$ is unbounded on $V_p$.
Moreover, 
if $\sigma^C(p)\ne 0$, then
the Gaussian curvature $K^L$
is unbounded and takes 
different signs on each side of $\Sigma_f$ around $p$.
\item 
If $p$ is a cuspidal edge singular point,
then the two principal curvature functions of $f$ 
are real-valued on $V_p$, and
one of the principal curvature function is bounded.
\item
Moreover, if $p$ is a cuspidal edge singular point,
then the other principal curvature function is unbounded.
In particular, the umbilical points and quasi-umbilical points
never accumulate at $p$.
\end{enumerate}
\end{Thm}

\begin{proof}
(1) follows from \eqref{eq:dp1967b}.
By \eqref{eq:a2}, it is obvious that
$\mb a_2(s)$ points in
the normal direction of $f$ at $p$
and is time-like $($resp. space-like$)$
if and only if $\mb D^L_f(s)$
points in a space-like $($resp. time-like$)$
direction. So (2) is obtained.
The assertions (3), (4) and (5)
follow from by the same reason in
the proofs of (3), (4) and (5)
of Theorem \ref{thm:T}.
\end{proof}

\subsection{Generalized cuspidal edge of types $\op{S}_l$}

Here we consider
the case that $\tilde {\mb D}^L_f(s_0)$  points in a light-like 
direction for some $s=s_0\in I$.
Since $\Gamma(s)$ is parametrized by the arc-length parameter,
$$
\mb a_0(s):=\Gamma'(s)\qquad (s\in I)
$$
is a unit space-like vector field along $\Gamma$. 
We choose a space-like vector field 
$\mb a_1$ and a time-like vector field $\mb a_2$ along $\Gamma$ 
so that 
$\{\mb a_0,\mb a_1,\mb a_2\}$ forms an orthonormal frame field
such that
$$
1=(\mb a_1,\mb a_1)_L=-(\mb a_2,\mb a_2)_L, \qquad  \mb a_2:=\mb a_0\times_L \mb a_1.
$$
Without loss of generality, we assume that $0\in I$ and $s_0=0$.
Then the $\Lo^3$-cuspidal direction $\tilde {\mb D}^L_f(0)$ points in
a light-like direction, 
which is the case that we have not discussed yet.
Since
$\tilde {\mb D}^L_f(0)$ is a light-like vector,
it is parallel to one of the following two
vectors
$$
\mb a_1(0)+ \mb a_2(0),\qquad \mb a_1(0)- \mb a_2(0).
$$
If the latter case occurs, by replacing $\Gamma(s)$ by $\Gamma(-s)$,
the three vector fields turn to be 
$-\Gamma'(-s)$, $\mb a_1(-s)$, $-\mb a_2(-s)$.
So we may assume that
$\tilde {\mb D}^L_f(0)$
is parallel to
$
\mb a_1(0)+ \mb a_2(0).
$
Moreover, since
$\Gamma', \mb a_1,\mb a_2$
give the same orientation as
$\Gamma', -\mb a_1,-\mb a_2$, we may also assume that
$$
\tilde {\mb D}^L_f(0)
=\frac{\mb a_1(0)+ \mb a_2(0)}{\sqrt{2}}.
$$
We then apply Proposition \ref{prop:fstE} in the appendix  for $(s,t)=(0,0)$,
and can assume that $f$ has the following expression
\begin{equation}\label{eq:fT1215}
f(s,t)=\Gamma(s)+ X(s,t)\mb a_1(t)+Y(s,t) \mb a_2(t),
\end{equation}
where $X(s,t)$ and $Y(s,t)$ are smooth functions
written as \eqref{eq:X1464}
and \eqref{eq:AB1468}.
We can write
$$
\pmt{A(s,t)\\
B(s,t)}
=
\int_0^t u\pmt{\cos \lambda(s,u)\\
\sin \lambda(s,u)} du,
\qquad
\lambda(s,t):=\int_0^t \mu(s,u)du,
$$
that is, 
\eqref{eq:ab1519}
 holds.
Unlike the previous cases
(E), (T), ($\op{S}_s$)
and ($\op{S}_t$), the function $\mu(s,t)$
cannot be considered as a geometric invariant of $f$.
One reason is that the two vector fields $\mb a_0$ and $\mb a_1$ are 
not uniquely determined from $f$, and the other reason is
that we use the pair $(\cos \lambda,\sin \lambda)$ not 
$(\cosh \lambda,\sinh \lambda)$.
We set
$
\mb a'_0=\kappa_1\mb a_1+\kappa_2\mb a_2.
$
Since $(\mb a_1,\mb a_1)=1$, we can write
$
\mb a'_1=-\kappa_1\mb a_1+\Omega \mb a_2,
$
where $\Omega$ is a smooth function.
In the setting of the previous section, this is the case that
$$
\epsilon_0=\epsilon_1=1,\,\,
\epsilon_2=-1, 
\quad \,\, C=\cos\theta,\,\, S=\sin\theta, \\
$$
and 
\begin{equation}
\theta(0)=\frac{\pi}4.
\end{equation}
Then
\eqref{eq:a1434}
can be written as
\begin{equation}\label{eq:2861}
\mathcal F'=\mathcal F\mathcal K, \qquad
\mathcal K=\pmt{
0 & -\kappa_1& \kappa_2 \\
\kappa_1 & 0& \Omega \\
\kappa_2 & \Omega & 0
},
\qquad\mathcal F:=(\mb a_0,\,\mb a_1,\mb a_2),
\end{equation}
which can be considered as an ordinary differential equation
with $\mathcal F$ as an unknown matrix-valued function,
and we can produce all of cuspidal edges of type $(S_l)$
along $\Gamma$ like as the previous cases.
If we set
$$
\delta_1:=\frac{\kappa_1(0)+\kappa_2(0)}{\sqrt{2}},\qquad 
\delta_2:=\frac{-\kappa_1(0)+\kappa_2(0)}{\sqrt{2}},
$$
then
\begin{align*}
&E^L(0,t)=1-\delta_1 t^2+\frac{\beta_3(0) \delta_2}3 t^3+O(t^4),\quad
F^L(0,t)=-\frac{\theta'(0)}2t^3+O(t^4), \\
&G^L(0,t)=\beta_3(0) t^3+O(t^4)
\end{align*}
and so
\begin{equation}\label{eq:2425}
\Delta_L(0,t)=\beta_3(0) t^3+O(t^4).
\end{equation}
Moreover, we have (cf. \eqref{eq:Dp262})
\begin{equation}\label{eq:2429}
d^C((0,0))=\delta_1
\end{equation}
and
\begin{align*}
\tilde L(0,t)&=\delta_1 t+\frac{\beta_3(0) \delta_2}{2}t^2 \\
&\phantom{aaaaa}+
\left(\frac{\beta_4(0)\delta_2+\alpha_4(0) \delta_1}{6}
-\delta_1^2-\Omega(0)\theta'(0)-\frac{\theta''(0)}{2}\right)t^3 +O(t^4), \\
\tilde M(0,t)&= -\theta'(0) t^2+\left(\Omega(0)\beta_3(0)
+\frac{\beta'_3(0)}2\right)t^3+O(t^4), \\
\tilde N(0,t)&= \frac12\beta_3(0) t^2+\frac{\beta_4(0)}3t^3+O(t^4). 
\end{align*}
By \eqref{eq:ab1547}, we have
\begin{align*}
\tilde W^L(0,t)
=
\pmt{
 \beta_3(0) \delta_1t^4 +O(t^5)
& \frac{-3\beta_3(0) \theta'(0)}{4}t^5 +O(t^6)  \\
-\theta'(0)t^2  +O(t^3)
& \frac{\beta_3(0)}{2} t^2 +\frac{\beta_4(0)}{3} t^3+O(t^4)} \\
=
\pmt{
2\mu_0(0) \delta_1t^4 +O(t^5)
& \frac{-3\mu_0(0) \theta'(0)}{2}t^5 +O(t^6)  \\
-\theta'(0) t^2  +O(t^3)
& \mu_0(0) t^2 +\mu_1(0) t^3+O(t^4)}
\end{align*}
and
\begin{align}\label{eq:2946a}
\det(\tilde W^L(0,t))&=2\mu_0(0)^2 \delta_1 t^6
+\frac{(4\mu_1(0) \delta_1-3\theta'(0)^2)\mu_0}{2}t^7+O(t^8), \\
\op{trace}(\tilde W^L(0,t))&=\mu_0(0) t^2+\mu_1(0)t^3+O(t^4). \label{eq:2946b}
\end{align}
Since $\tilde W^L(t,0)$ is a triangle matrix 
modulo $O(t^5)$-term,
the diagonal components of $\tilde W^L(t,0)$ give the following
asymptotic expansions of the two eigenvalues of $\tilde W^L$;
\begin{equation}\label{eq:LL2017c}
\tilde \lambda_1(0,t):=\mu_0(0) t^2+\mu_1(0)t^3+O(t^4),\quad
\tilde \lambda_2(0,t):=2\mu_0(0) \delta_1 t^4+O(t^5).
\end{equation}
We prove the following:

\begin{Thm}\label{thm:S3}
Let $f:\mc U_f\to \Lo^3$ be a generalized cuspidal edge along 
$\Gamma$ of type $S$ in $\Lo^3$.
Suppose that 
$p:=(s_0,0)$ is a 
cuspidal edge singular point
and
the $\Lo^3$-cuspidal direction $\tilde{\mb D}^L_f(s_0,0)$
points in a light-like direction at $p$.
Then the following assertions hold:
\begin{enumerate}
\item $i_p=3$ holds, and $f$ is light-like at $p$.
\item The sign $\sigma^C(p)$ vanishes if and only if
$\Gamma''(s_0)$ points in a light-like direction which is different from
the direction of $\tilde{\mb D}^L_f(s_0,0)$.
\item $H^L$ is unbounded near a sufficiently small
neighborhood of $p$.
\item $f$ changes its causal type from space-like
to time-like near $p$.
Moreover, if $\sigma^C(p)\ne 0$ then $K^L$ diverges near $p$ and
takes the different signs on each side of $\Sigma_f$ in $V_p$. 
\item 
The two principal curvatures of $f$ are real-valued and 
one of them is unbounded near $p$. 
Moreover, other one is also unbounded 
if and only if $\sigma^C(p)\ne 0$.
\item The set of umbilical or quasi-umbilical points of $f$ 
cannot accumulate at $p$.
\end{enumerate}
\end{Thm}

\begin{proof}
The fact that $f$ is light-like at $p(=(s_0,0))$ follows from the fact that
$\tilde{\mb D}^L_f(s_0,0)$
points in a light-like direction at $p$. 
We now assume $p$ is a cuspidal edge singular point (i.e. $\mu_0(s_0)\ne 0$).
Then $i_p=3$ follows from \eqref{eq:2425}, proving (1).
We now prove (2):
By \eqref{eq:2861}, 
$$
\Gamma''(s_0)=\mb a'_0(s_0)=\kappa_1(s_0) \mb a_1(s_0)+\kappa_2(s_0) \mb a_2(s_0).
$$
By \eqref{eq:2429}, $\sigma^C(p)=0$ if and only if $\kappa_1(s_0)=-\kappa_2(s_0)$,
which is equivalent to the fact that
$\Gamma''(s_0)=\kappa_1(s_0)(\mb a_1(s_0)-\mb a_2(s_0))$.
Since $\mb a_1(s_0)-\mb a_2(s_0)$ is a light-like vector, we obtain the conclusion.
On the other hand, (3) and (4)
follow from
\eqref{eq:2946a} and
\eqref{eq:2946b}
(cf. \eqref{eq:KH584b}).
Since $\lambda_i$ ($i=1,2$) are the same order as
$\tilde \lambda_i/|\Delta_L|^{3/2}$ with respect to $t$,
\eqref{eq:LL2017c} implies that $\lambda_1$ is unbounded 
(cf. \eqref{eq:LL2017c}).
Moreover, $\lambda_2$ is unbounded if and only if $\delta_1\ne 0$,
which happens only when $\sigma^C(p)\ne 0$, proving (5).
By \eqref{eq:LL2017c}, the two eigenvalues $\tilde \lambda_1$
and $\tilde \lambda_2$ have different orders with respect to $t$,
and (6) is obtained.
\end{proof}

\begin{Rmk}\label{rmk:3188}
In Theorem \ref{thm:S3}, we assumed that
$p$ is a cuspidal edge singular point.
If we remove this assumption,
that is, if $p$ is a  generalized
cuspidal edge singular point such that $\tilde{\mb D}^L_f(s_0,0)$
points in a light-like direction at $p$,
then \eqref{eq:2425}
implies 
the order $i_p$ is greater than or equal to $3$.
By \eqref{eq:2425},
$i_p=3$ if and only if
$\beta_3(0)(=2\mu_0)$ does not vanish.
So $i_p=3$ holds only when
$p$ is a cuspidal edge singular point.
\end{Rmk}

\begin{figure}[htb]
\begin{center}
\includegraphics[height=3.6cm]{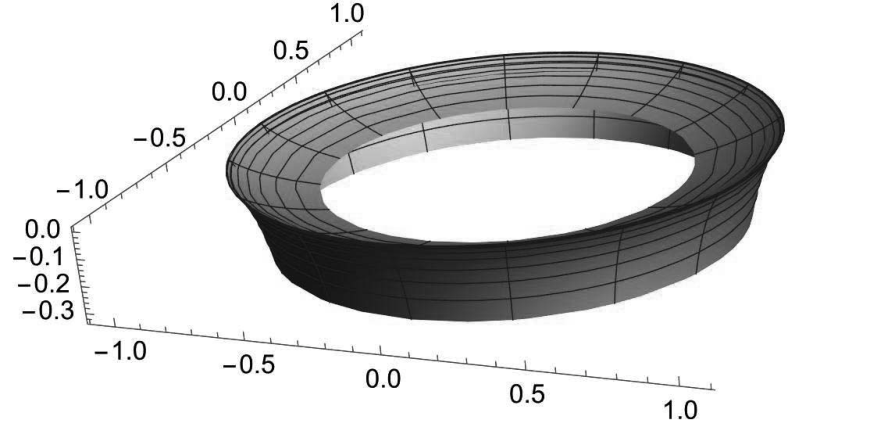}
\end{center}
\caption{A cuspidal edge of order three along a circle}\label{ex:L0}
\end{figure}

\begin{Example}
We set
$
\Gamma(s):=(\cos s,\sin s,0)
$
and $x(t):={t^2}/{2},\,\,y(t):={t^3}/{3}$.
Noticing that $\Gamma''=-\Gamma$ is
the normal vector of the plane curve $\Gamma$,
consider the cuspidal edge given by
$$
f(s,t):=(1-x(t)-y(t))\Gamma(s)+(-x(t)+y(t))\mb e_3
\qquad (\mb e_3:=(0,0,1)).
$$
Since
the coefficients of the first fundamental form are
computed as
\begin{equation}\label{eq:2999a}
E^L=\frac{1}{36} \left(2 t^3+3 t^2-6\right)^2,\quad F^L=0,\quad G^L=4t^3.
\end{equation}
Since the coefficient of $t^2$ for $E^L$ does not vanish,
we have $\sigma^C\ne 0$.
Moreover, we have 
\begin{equation}\label{eq:2999b}
\Delta_L=\frac{1}{9} t^3 \left(6-3 t^2-2t^3\right)^2,
\end{equation}
which is of order three at each point $(s,0)$ of $\Sigma_f$ (see Figure \ref{ex:L0}).
The functions $\tilde L, \tilde M,\tilde N$
are given by (cf. \eqref{tL}, \eqref{tM} and \eqref{tN})
\begin{equation}\label{eq:3006}
\tilde L=\frac{t(1-t)\left(6-3 t^2-2t^3\right)^2}{36},\quad \tilde M=0,\quad
\tilde N=\frac{t^2 }{3} \left(6-3 t^2-2t^3\right).
\end{equation}
By \eqref{eq:K1Lb}, \eqref{eq:2999a}, \eqref{eq:2999b} and \eqref{eq:3006},
the Gaussian curvature and the mean curvature do not depend on  $s$ 
and are computed as
$$
K^L=\frac{-3(1-t)}{4t^3(6-3t^2-2t^3)},\qquad
\pm H^L=\frac{6+9t^2-14t^3}{8 |t|^{5/2} (6-3t^2-2t^3)}.
$$
In particular, the principal curvatures are 
both unbounded.
\end{Example}

\begin{proof}[Proof of Theorem A and Corollary B]
Summarizing the assertions of Theorems \ref{thm:T}, 
\ref{thm:S1} and \ref{thm:S3}, we obtain Theorem~A.
We remark that
the second assertion of 
\ref{item:A5} follows from Remark \ref{rmk:3188}.

We next prove Corollary B:
If the mean curvature function of $f$ is bounded,
then \ref{item:A6} of Theorem~A implies that $\Gamma$
is a regular curve of type $L$, that is, $\Gamma'$ always
points in a light-like direction, proving Corollary ~B. 
\end{proof}

\section{The case that $\Gamma'$ points in a light-like direction}

In this section, we consider the case that 
$\Gamma'(s)$ points in a light-like direction for some $s$.
The case $(\op{L}_g)$ ($g=2,4$)
given in Subsection 2.1 is contained in this setting.

\subsection{Computations in the case that  $\Gamma'(0)$ is a light-like vector}
Fix an open interval $I$ containing $0\in \R$.
Let $\mathcal E^2$ be a space-like plane in $\Lo^3$.
Without loss of generality, we may assume that
$\mathcal E^2$ is the $xy$-plane.
The following fact can be proved easily.

\begin{Fact}\label{2995}
Let $\Gamma:I\to \Lo^3$ be a regular curve of 
type $L$. 
If we set $\Gamma(s)=(x(s),y(s),z(s))$
and assume that $s$ is the arc-length parameter of 
the curve $\gamma(s):=(x(s),y(s))$ in the $xy$-plane,
then $\Gamma$  has the expression
$
\Gamma(s)=(\gamma(s),\pm s)
$
for  $s\in I$.
\end{Fact}

Without loss of generality, we may assume that
\underline{$\Gamma(s)$ is future-pointing} at $s=0$.
Regarding this, 
we set
\begin{equation}\label{eq:G2754}
\Gamma(s):=(\gamma(s),\phi(s)) \qquad (s\in I),
\end{equation}
where $\gamma(s)$ is a regular curve in $\mathcal E^2$
parametrized by the arc-length and
$\phi(s)$ is a smooth function satisfying
\begin{equation}\label{eq:gc2116}
\phi'(0)=1,
\end{equation}
which implies that $\Gamma'(0)$ points in a light-like direction.
If $\Gamma$ is of type $L$, then 
$\phi$ is the identity map, that is, $\phi$ 
satisfies
\begin{equation}\label{eq:gc2121}
\phi(s)=s \qquad (s\in I).
\end{equation}
We 
set 
$
\mb e(s):=(\gamma'(s),0).
$
We denote by
$\mb n(s)\in \Lo^3$ 
the left-ward unit normal vector $\gamma'(s)$ in the $xy$-plane.
For each $s\in I$, three vectors
$$
\mb e(s),\quad \mb n(s),\quad 
\mb v:=(0,0,1)
$$
form an orthonormal basis of $\Lo^3$ for each $s\in I$, and
$\mb v$ is a time-like vector.
We fix our general setting as
$$
\mb a_0:=\mb e(s),\quad \mb a_1:=\mb n(s),\quad \mb a_2:=\mb v,\quad 
\epsilon_0=\epsilon_1=1,\quad \epsilon_2=-1.
$$
Then, 
$$
\Gamma'(s)=\mb e+\phi'(s) \mb v
$$
holds. Let $f$ be a generalized cuspidal edge along $\Gamma$.
By applying Proposition \ref{prop:fstE} in the appendix for $(s,t)=(0,0)$,
like as the general setting given in Section 2,
we may  assume that $f$ has the following expression
\begin{equation}\label{eq:fT1215b}
f(s,t)=\Gamma(s)+ X(s,t)\mb n(s)+Y(s,t)\mb v,
\end{equation}
where $X(s,t)$ and $Y(s,t)$ are smooth functions of the form
\begin{align}
\label{eq:A}
X(s,t)&=A(s,t)\cos\theta(s)+B(s,t)\sin \theta(s), \\
\nonumber
Y(s,t)&=B(s,t)\cos\theta(s)-A(s,t)\sin \theta(s)
\end{align}
and
\begin{align*}
&A(s,t)=\int_0^t u \cos \lambda(s,u)du,\qquad
B(s,t)=\int_0^t u \sin \lambda(s,u)du, \\
&\lambda(s,t):=\int_0^t \mu(s,u)du.
\end{align*}
The vector field $\cos \theta(s) \mb n(s)+\sin \theta(s) \mb v$ 
gives the direction of the cusp for the section of the image of $f$,
and the plane passing through $\Gamma(s)$ is
spanned by $\mb n(s)$ and $\mb v_3$.
Using the expansion of $\mu$ given in \eqref{eq:mu2029},
we have (cf. \eqref{eq:AB1468})
$$
\pmt{A \\
B}
=
\frac12\pmt{1\\ 0} t^2+
\frac13\pmt{0\\ \mu_0} t^3
+\frac18\pmt{-\mu_0^2\\ \mu_1} t^4
+\frac1{30}\pmt{-3 \mu_0\mu_1\\ -\mu_0^3 + \mu_2} t^5
+O(t^6)
$$
and
\begin{align}\label{eq:ab1519b}
&\alpha_2=1,\quad \alpha_3=0,\quad \alpha_4=-3 \mu_0^2,\quad \alpha_5=-12 \mu_0\mu_1,  \\
&\beta_3=2\mu_0,\quad \beta_4=3\mu_1,\quad \beta_5=4(-\mu_0^3+\mu_2).  \nonumber
\end{align}
Let $\kappa(s)$ be the curvature 
function of the curve $\gamma(s)$
in the $xy$-plane (as $\mc E^2$).
Then, we have (cf. \eqref{eq:K1441})
\begin{equation}\label{eq:FrameL1653}
\mathcal K=
\pmt{0 & -\kappa & 0 \\
     \kappa & 0 & 0 \\
     0 & 0 & 0}.
\end{equation}
For the sake of simplicity, we set
$$
C(s):=\cos \theta(s),\quad S(s):=\sin \theta(s),\quad
C_2(s):=\cos 2\theta(s),\quad S_2(s):=\sin 2\theta(s).
$$
Since $\Gamma'(s)$ is different from $\mb a_0(s)$ (cf. \eqref{eq:A01960}),  we cannot apply
the general computations given in Section 2, and 
must begin with the computations of
$f_s$ and $f_t$ as follows:
\begin{align}\label{eq:L2-2695}
f_s&=\mb e+\varphi'\mb v 
-\frac{\kappa C\mb e+\theta' S \mb n+\theta' C\mb v}{2}t^2 \\\nonumber
&\phantom{aaaa}+\frac{-\kappa \beta_3 S \mb e+(\theta'\beta_3C+\beta'_3S)\mb n
+(-\theta' \beta_3 S+\beta'_3C)\mb v}6 t^3
+O(t^4), \\
\nonumber
f_t&=(C \mb n-S \mb v)t+\frac{\beta_3S \mb n+\beta_3C \mb v}2 t^2 \\
&\phantom{aaaa}+\frac{(\alpha_4C+\beta_4S)\mb n+(-\alpha_4S+\beta_4C)\mb v}6 t^3 
\nonumber
+O(t^4).
\end{align}
Then, we have
\begin{align*}
E^L&=1-\varphi'^2+ \left(-\kappa+\theta'\varphi'\right)Ct^2+O(t^3),\\
F^L&=\varphi 'S t-\frac{\beta _3\varphi'C}{2} t^2 +O(t^3),\quad 
G^L=C_2t^2+ \beta_3S_2t^3 +O\left(t^4\right)
\end{align*}
and
\begin{equation}\label{eq:D2671}
\Delta_L=(C^2(1-\phi'^2)-S^2)t^2+\beta_3 CS(2-{\varphi'}^2)t^3+O(t^4).
\end{equation}
Moreover,
\begin{equation}\label{eq:3425}
\tilde \nu^E=(-\phi'C\mb e+S \mb n+C\mb v)t
-\frac{\beta_3}2(\phi'S\mb e+C\mb n-S\mb v) t^2+O(t^3)
\end{equation}
and
\begin{align}\label{eq:fss3090}
f_{ss}&=\kappa\mb n+\phi'' \mb v \\ \nonumber
&\phantom{a}+\frac{
(2\kappa \theta' S-\kappa'C)\mb e-((\kappa^2+\theta'^2)C+\theta''S)\mb n
+(\theta'^2S-\theta''C)\mb e
}2t^2 
\\ \nonumber
&\phantom{aaaaaaaaaaaaaaaaaaaaaaaaaaaaaaaaaaaaaaaaaaaaaa} +O(t^3),
\\ \nonumber
f_{st}&=-\Big(\kappa C \mb e+\theta'(S\mb n+C\mb v)\Big)t \\
&\phantom{aaa}\nonumber
+\frac{-\kappa \beta_3 S \mb e+(\theta'\beta_3C+\beta'_3S)\mb n
+(-\theta'\beta_3S+\beta'_3C)\mb v
}{2}t^2 +O(t^3), \\ \nonumber
f_{tt}&=C\mb n-S\mb v+(S\mb n+C\mb v)t \\
\nonumber
&\phantom{aaaaaaaaaa}
+\frac{(\alpha_4C+\beta_4S)\mb n
+(-\alpha_4S+\beta_4C)\mb v}2t^2+O(t^3).
\end{align}
In particular, we have
\begin{equation}\label{eq:ds0}
d^C((s,0)):=-\det(f_{s},f_{ss},f_{tt})|_{t=0}=S\kappa+C\phi''
\end{equation}
and
\begin{align*}
\tilde L&=
\left(\kappa S+\varphi ''C\right)t-\frac{
\beta_3(s)}{2}(\kappa C-\varphi '' S)t^2 +O\left(t^3\right), \\
\tilde M&=
\left(\kappa \varphi'C^2-\theta '\right)t^2+O\left(t^3\right),\quad
\tilde N=\frac{\beta _3}{2}t^2+O\left(t^3\right).
\end{align*}
If we set
\begin{equation}
(\tilde W^L
=)
\pmt{ w_{1,1} & w_{1,2} \\
w_{2,1} & w_{2,2}
}:=
\pmt{G^L & -F^L \\ -F^L & E^L}\pmt{\tilde L & \tilde M \\ \tilde M & \tilde N},
\end{equation}
then 
\begin{align}\label{eq:w12}
w_{1,1}&=\Big(S\theta'\phi'+\kappa S(C_2-C^2\phi'^2)+\phi''CC_2\Big)t^3
+O(t^4), \\ \nonumber
w_{1,2}&=
-\frac{1}{2} \varphi'\beta _3 St^3 +O\left(t^4\right), \\ \nonumber
w_{2,1}&=
\Big(
(\theta'-\kappa \phi'C^2)(-1+\phi'^2)-\phi'S(\kappa S+\phi''C)
\Big)t^2 
+O(t^3), \\ \nonumber
w_{2,2}&=
-\frac{\beta_3}{2} \left(\varphi'^2-1\right)t^2+
\frac{1}3
\Big(
3\phi' S(\theta'-\phi' \kappa C^2)+\beta_4(1-\phi'^2)
\Big)t^3+O(t^4)
\end{align}
hold. 
In particular,  the square of the difference of the two eigenvalues of
$\tilde W^L$ satisfies
\begin{equation}\label{eq:F2759}
(\op{trace}\,\tilde W^L(s,t))^2-4\det(\tilde W^L(s,t))=
\frac{\beta _3^2 \left(\varphi'^2-1\right)^2}{4}t^4 +O(t^5).
\end{equation}
Substituting $\phi'(0)=1$, we have
\begin{align}
\label{eq:K2759}
\det(\tilde W^L(0,t))
&=-\frac{\beta _3(0) S_{0}^2}{2}
\Big(\kappa (0)S_0+\varphi ''(0)C_0\Big)t^5+O\left(t^6\right),\\
\label{eq:W2764}
\op{trace}(\tilde W^L(0,t))
&=
\frac12 \Big (2S_0(-\kappa(0)+2\theta'(0))+
(C_0+C_{3,0})\phi''(0)\Big )t^3+O(t^4),
\end{align}
where
$$
C_0:=\cos \theta(0),\quad S_0:=\sin \theta(0),\quad
C_{3,0}=\cos 3\theta(0).
$$

Using the above computations,
we give the proof of Proposition C in the introduction:

\begin{proof}[Proof of Proposition C]
Without loss of generality, we may assume that $s_0=0$ and $p=(0,0)$.
By \eqref{eq:3425}, 
$$
\tilde \nu^L(p)=-C_0\mb e+S_0 \mb n-C_0\mb v
$$
is a normal vector of $f$ at $p$.
Since $\Gamma'(0)$ points in a light-like direction,
$\tilde \nu^L(p)$ is space-like or light-like.
Since
$$
1=C_0^2+S_0^2\ge C_0^2,
$$
the normal vector $\tilde \nu^L(p)$ is space-like (resp. light-like)
if and only if
$S_0\ne 0$ (resp. $S_0=0$).
By \eqref{eq:Delta2804b},
this shows the equivalency of \ref{item:C1} and \ref{item:C2}.

Since $\Gamma'(0)$ points in a light-like direction, we have $\phi'(0)=1$.
So \eqref{eq:D2671} implies
\begin{equation}\label{eq:Delta2804b}
\Delta_L(0,t)=-S_0^2 t^2+O(t^3).
\end{equation}
So, if $i_p$ is greater than $2$,
then $S_0=0$ holds, which implies $i_p\ge 4$.
In particular, either $i_p=2$ or $i_p\ge 4$ holds.

We now prove the last statement: 
So we assume that $\Gamma$ is of type $L$ 
satisfying $\Gamma''(0)\ne \mb 0$.
In this case $\phi''(0)=0$ and
\eqref{eq:ds0} yields that
\begin{equation}\label{eq:ds3}
d^C(p):=S_0\kappa(0).
\end{equation}
Since $\Gamma''(0)\ne \mb 0$ and $\phi''(0)=0$,
the first equation of
\eqref{eq:fss3090} implies $\kappa(0)\ne 0$.
Since the sign of $d^C(p)$ coincides with $\sigma^C(p)$,
\eqref{eq:ds3} implies the equivalency of \ref{item:C2} and 
the condition $\sigma^C(p)\ne 0$ (resp. $\sigma^C(p)=0$).
\end{proof}

The following is an analogue of Lemma~\ref{lem:1688}:

\begin{Lemma}\label{lem:ccrL}
The point $(s_0,0)$ $(s_0\in I)$
is a  cuspidal edge singular point 
$($resp. a cuspidal cross cap  singular point$)$ of
$f$ if and only if $\mu_0(s_0) \ne 0$ $($resp. $\mu_0(s_0)=0$ 
and $\mu'_0(s_0)\ne 0)$. 
\end{Lemma}

\begin{proof}
The statement for cuspidal edge singular points
can be proved by
imitating the proof
of Lemma~\ref{lem:1688}.
To prove the assertion for cuspidal cross caps,
we set
$$
\psi(s):=\det(\Gamma'(s),\check \nu_E(s,0),(\check \nu_E)_t(s,0))
$$
and (cf. \eqref{eq:3425})
$$
\check \nu_E:=\frac{\hat\nu_E}{t}=
(-\phi'C\mb e+S \mb n+C\mb v)
-\frac{\beta_3}2(\phi'S\mb e+C\mb n-S\mb v) t+O(t^2).
$$
By \eqref{eq:L2-2695}, we have
$\Gamma'=f_s=\mb e+\varphi'\mb v$.
Using the relation $\beta_3=2\mu_0$, we have
$$
\psi(s)=\mu_0(s)(1+\phi'(s)^2).
$$
So $\psi(s_0)=0$ if and only if $\mu_0(s_0)=0$.
If $\mu_0(s_0)=0$, then $\psi'(s_0)=0$ holds if and only if
$\mu'_0(s_0)=0$. 
So imitating the proof of  Lemma \ref{lem:1688},
we obtain the assertion for cuspidal cross caps.
\end{proof}

We next prepare the following:

\begin{Lemma}\label{lem:2775}
Let $\Gamma:I\to \Lo^3$ be a regular curve 
such that $\Gamma'(0)$ is a light-like vector, where
we assume $0\in I$.
Let $f(s,t)$ be a generalized cuspidal edge along $\Gamma$
and $p:=(0,0)$ a cuspidal edge singular point of $f$. 
If the order $i_p$ is equal to $2$,
then umbilical points of $f$ never accumulate at 
$p$.
\end{Lemma}

\begin{proof}
Suppose $\{(s_k,t_k)\}_{k=1}^\infty$ is a sequence of umbilical
 points of $f$
which  converges to $p:=(0,0)$. 
Then we have $w_{12}(s_k,t_k)=0$.
Since $\beta_3=2\mu_0\ne 0$
and $\phi'(0)=1$, 
the second equation of \eqref{eq:w12} implies
$$
0=w_{12}(s_k,t_k)=-\mu_0(s_k) \sin \theta(s_k) t_k^3+O(t_k^4).
$$
However, since $\sin\theta(0)\ne 0$ and $\mu_0(0)\ne 0$,
the right-hand side does not vanish when $k$ is 
sufficiently large, a contradiction.
\end{proof}

We fixed a space-like plane $\mathcal E^2$
in $\Lo^3$ and is regarding it as the $xy$-plane.
If we replace another space-like plane $\tilde{\mathcal E}^2$,
then we have the two orthogonal projections:
$$
\pi:\Lo^3\to {\mathcal E}^2, \qquad
\tilde \pi:\Lo^3\to \tilde{\mathcal E}^2.
$$

\begin{Proposition}\label{Cor:I1}
Even if we replace the
orthogonal projection $\pi$
by $\tilde \pi$,
the condition that $\sin \theta(s)=0$ 
does not change, that is,
this condition
is invariant under the choice of
the orthogonal projection to a space-like plane in $\Lo^3$.
\end{Proposition}

\begin{proof}
As shown in
Proposition \ref{prop:1477}, the order $i_p$ is an invariant of
generalized cuspidal edges, and $i_p\ge 3$ holds if and only if
$\sin \theta=0$ at the point $p$, 
as shown in the proof of Proposition C.
So we obtain the conclusion.
\end{proof}

Using this, we can prove Proposition E in the introduction:

\begin{proof}[Proof of Proposition E]
Let $f:\mc U_f\to \Lo^3$ be a generalized cuspidal edge along a
regular curve $\Gamma:I\to \Lo^3$.
Suppose that $p\in \mc U_f$ is a cuspidal edge singular point of $f$.
Without loss of generality, we may assume that $0\in I$ 
and $p:=(0,0)$.
If $\Gamma'(0)$ is not a light-like vector,
$p$ is not an accumulation point of the set of umbilics 
nor of the set of quasi-umbilics by \ref{item:A9} of Theorem A.
So we may assume that
$\Gamma'(0)$ is a light-like vector.
In this case, umbilical points cannot accumulate at $p$ by Lemma~\ref{lem:2775}.
\end{proof}

We next prove the following:

\begin{Proposition}\label{prop:LG}
Let $p:=(0,0)$ $(0\in I)$ be a 
cuspidal edge singular point of $f$ of order $2$
such that $\Gamma'(0)$ points in a 
light-like direction. 
Then there exists a neighborhood $V_p$ of $p$
such that the following assertions hold:  
\begin{enumerate}
\item The normal vector of  $f$ at $p$
is time-like.
\item If $\sigma^C(p)\ne 0$, 
then $K^L$ is unbounded on $V_p$. 
Moreover, on one side of $\Sigma_f\cap V_p$, the
principal curvatures of $f$ are real-valued,
and on the other side they are non-real-valued.
In this situation, the two principal curvatures are 
both unbounded on $V_p$.
\item If $\sigma^C(p)\ne 0$ and also
$\inner{\Gamma'(0)}{\Gamma''(0)}\ne 0$ $($that is, $\phi''(0)\ne 0)$, 
then the set of quasi-umbilical points of $f$
accumulates at $p$ and
lie in one side of $V_p\setminus \Sigma_f$.
\end{enumerate}
\end{Proposition}

\begin{proof}
(1) follows from Proposition C.
We assume $\sigma^C(p)\ne 0$. 
By \eqref{eq:ds0}, we have 
$$
0\ne d^C(p)=S_0\kappa(0)+C_0\phi''(0).
$$
Since $\beta_3(0)=2\mu_0(0)$,
substituting $\phi'(0)=1$
to \eqref{eq:K2759}, 
we have
\begin{equation}\label{eq:2888}
\det(\tilde W^L(0,t))=\mu_0(0)S^2_0 d^C(p) t^5+O(t^6),
\end{equation}
which implies 
\begin{equation}\label{eq:K2897}
K^L(0,t)=\frac1{S_0^4 t}\Big(\mu_0(0)d^C(p)+O(t)\Big).
\end{equation}
Since $p$ is a cuspidal edge singular point, we have $\mu_0(0) \ne 0$.
Since $\sigma^C(p)\ne 0$, the unboundedness of $K^L$ near the point $p$
follows from \eqref{eq:K2897}.
Finally, we prove (3):
We set
$$
F(s,t):=\frac{
(\op{trace}\tilde W^L(s,t))^2-4\det(\tilde W^L(s,t))}{t^4}.
$$
By \eqref{eq:F2759}, $F$ is a smooth function satisfying
\begin{equation}\label{eq:F2868}
F(s,t)=
\frac{\beta _3^2 \left(\varphi'^2-1\right)^2}{4} +O(t).
\end{equation}
Since $(\op{trace}(\tilde W^L(0,t))^2$ belongs to the class $O(t^6)$,
\eqref{eq:2888} implies that
\begin{equation}\label{eq:F3706}
F(0,t)=-4\frac{\det(\tilde W^L(0,t))}{t^4}+O(t)=4S^2_0 \mu_0(0)d^C((0,0)) t+O(t^2).
\end{equation}
In particular, the two eigenvalues of $\tilde W^L(0,t)$
are 
\begin{itemize}
\item real numbers if $\mu_0(0)d^C((0,0)) t\ge 0$, and
\item non-real numbers if $\mu_0(0)d^C((0,0)) t<0$.
\end{itemize}
By \eqref{eq:K2759}
and \eqref{eq:W2764},
the  
two principal curvatures $\lambda_1,\lambda_2$ have the same order
with respect to $t$.
Then, \eqref{eq:K2897} implies
$\lambda_1$ and $\lambda_2$ are both unbounded,
so (2) is proved.
By \eqref{eq:F3706},
we have
$$
F_t(0,0)=4\mu_0(0)S^2_0 d^C((0,0))(\ne 0).
$$
By the implicit function theorem, there exists a $C^\infty$-function of
$\psi(s)$ near $s=0$ such that 
$F(s,\psi(s))=0$ and $\psi(0)=0$.
Then we have
$$
F_s(0,0)+F_t(0,0)\psi'(0)=0.
$$
In particular, if $(s,\psi(s))$ is not a singular point,
it is a quasi-umbilical  point or an umbilical point.
By \eqref{eq:F2868} with $\phi'(0)=1$, 
we have $F_s(0,0)=0$, and so $\psi'(0)=0$ holds.
So, we have
\begin{equation}
\label{eq:FSS3728}
F_{ss}(0,0)+F_t(0,0)\psi''(0)=0.
\end{equation}
Again, by \eqref{eq:F2868} with $\phi'(0)=1$ and $\beta_3=2\mu_0\ne 0$,
we have 
$$
F_{ss}(0,0)=8\mu_0(0)^2 \varphi''(0)^2\ne 0.
$$
Since $F_t(0,0)\ne 0$, 
\eqref{eq:FSS3728} implies
$\psi''(0)\ne 0$.
Thus, for each $s(\ne 0)$, $(s,\psi(s))$ 
is not a singular point, and so, it
is a quasi-umbilical point or umbilical point.
Since we have shown that (cf. Lemma~\ref{lem:2775})
umbilical points of $f$ never accumulate at $(s,0)$,
the point $(s,\psi(s))$ must be a
quasi-umbilical point lying on
the side of the $s$-axis where the two principal curvatures
are not real.
\end{proof}

\subsection{Light-like cuspidal edges of general type}

\begin{Definition}\label{def:Lcg2887}
Let 
$
f:I\times (-\delta,\delta)\to \R^3
$
be a generalized cuspidal edge
along a regular curve $\Gamma$ of type $L$.
Then $f$ is 
called a {\it light-like cuspidal edge of general type}
if $i_p=2$ 
and $\Gamma''(s)\ne 0$ at each $p:=(s,0)$ ($s\in I$).
\end{Definition}

We prove the following assertion:

\begin{Proposition}\label{prop:2991}
If $f$ is a light-like cuspidal edge of general type, 
then $\sigma^C(p)\ne 0$ for each $p\in \Sigma_f$
and the curvature function $\kappa$ of $\gamma:=\pi\circ \Gamma$
 does not vanish everywhere, that is, $\gamma$ 
has no inflection points in the $xy$-plane.
\end{Proposition}

\begin{proof}
Since $\Gamma$ is of type $L$, we have $\phi''(s)=0$.
Then,
\eqref{eq:ds0} is reduced to
$d^C((s,0))=S\kappa$.
Since $i_p=2$, we have $S\ne 0$.
Since  $f_{ss}(s,0)=\Gamma''(s)\ne \mb 0$, 
the first equation of
\eqref{eq:fss3090} implies
$\kappa(s)\ne 0$.
\end{proof}

We now prove Theorem~D in the introduction:

\begin{proof}[Proof of Theorem D]
\ref{item:D2} follows from
Proposition \ref{prop:2991}, and
\ref{item:D1} follows from 
Propositions C.
Since $\Gamma$ is of type $L$,
$\phi''(0)=0$ holds.
Then, \eqref{eq:W2764} implies 
$$
|H^L(0,t)|
=
\frac1{2S_0^2}\Big|\kappa(0)-2\theta'(0)\Big|+O(|t|).
$$
Since $\Gamma'(s)$ ($s\in I$)
points in a light-like direction,
$0(\in I)$ can be replaced by any $s\in I$, we have
$$
|H^L(s,t)|
=
\frac1{2S^2}\Big|\kappa(s)-2\theta'(s)\Big|+O(|t|),
$$
where $S=\sin \theta(s)$.
So we obtain \ref{item:D3}.
Using the same method to obtain the expression of
$H^L(s,t)$ from the expression of $H^L(0,t)$, 
the expression $K^L(0,t)$
given in \eqref{eq:K2897} induces
$$
K^L(s,t)=\frac1t\Big(\mu_0S^{-4} d^C((s,0))+O(t)\Big),
$$
which proves \ref{item:D4}.
Since $H_L$ is bounded and $K^L$ is unbounded,
the two principal curvatures $\lambda_i$ ($i=1,2$)
are both unbounded.
Since $\Gamma'(s)$ ($s\in I$)
always points in a light-like direction,
we have (using $\beta_3=2\mu_0$)
$$
\det(\tilde W^L)
=
-{\mu_0\kappa S^3}t^5+O\left(t^6\right),\qquad
\op{trace}(\tilde W^L)
=
-(\kappa-2\theta')S t^3+O(t^4).
$$
Thus, the square of the difference of the two eigenvalues of
$\tilde W^L$ satisfies
$$
(\op{trace}\,\tilde W^L)^2-4\det(\tilde W^L)
=-4\det(\tilde W^L)t^5 +O(t^6)=
2\kappa \mu_0 S^3 t^5+O\left(t^6\right).
$$
Since $f$ is a
light-like cuspidal edge of general type,
$\kappa \mu_0 S^3$ never vanishes on $I$, 
and so \ref{item:D5} and \ref{item:D6} are obtained.
\end{proof}

\section{Proof of Proposition F and Theorem G}

Let $\Gamma:I\to \Lo^3$ 
be a regular curve of type $L$,
and set $\phi(s):=s$ for each $s\in I$.
Then \eqref{eq:G2754} is reduced to
$
\Gamma(s):=(\gamma(s),s)\,\, (s\in I),
$
where $\gamma(s)$ is a regular curve in the $xy$-plane
parametrized by the arc-length.
We let 
$
f:I\times (-\delta,\delta)\to \R^3
$
be a generalized cuspidal edge
satisfying $f(s,0)=\Gamma(s)$ ($s\in I$)
whose singular points are of order four,
that is, $\sin \theta$ is identically zero in the setting of
the previous section.
So, we set
\begin{equation}\label{3606}
\sigma:=\cos \theta \in \{1,-1\}
\end{equation}
and
\begin{equation}\label{eq:fT2499}
f(s,t)=\Gamma(s)+ \sigma\, x(s,t)\mb n(s)+\sigma\, y(s,t)\mb v,
\end{equation}
where
$$
x(s,t)=\int_0^t u \cos \lambda(s,u)du,\quad
y(s,t)=\int_0^t u \sin \lambda(s,u)du,\quad
\lambda(s,t):=\int_0^t \mu(s,u)du.
$$
The formula \eqref{eq:fT2499}
produces all generalized cuspidal edges of order 
greater than or equal to four.
By Lemma~\ref{lem:ccrL},
the singular point $(s,0)$ is a cuspidal edge if and only if $\mu_0(s)\ne  0$.
To obtain information on terms of higher order, we 
start with the computations of $f_s$ and $f_t$;
\begin{align*}
f_s&=\left(\mb e+\mb v\right)-\frac{\sigma \kappa \mb e}{2} t^2 
+\frac{\sigma \beta'_3}{6} \mb v t^3
+\frac{
-\sigma  \kappa \alpha_4\mb e+\sigma  \alpha'_4\mb a_1+
\sigma  \beta'_4\mb v}{24} t^4 +O\left(t^5\right),\\
f_t&=
\sigma\mb n t
+\frac{\sigma \beta _3}{2} \mb v t^2
+\frac{\sigma \alpha _4\mb n+\sigma  \beta_4\mb v}{6} t^3 
+
\frac{\sigma}{24} t^4 
\left(\alpha_5\mb n+\beta _5(s,0)\mb v\right)+O\left(t^5\right).
\end{align*}
The first three terms of the above expansions of $f_t$
and $f_s$ coincide with
those of \eqref{eq:L2-2695} by substituting $\theta'=0$, 
$C=1$ and $S=0$.
We  have
$$
E^L=-\sigma \kappa t^2 -\frac{\sigma  \beta'_3}{3} t^3 
+O\left(t^4\right), \quad
F^L=-\frac{\sigma  \beta_3}{2} t^2 
-\frac{\sigma  \beta_4}{6} t^3 +O\left(t^4\right),\quad
G^L=t^2+O\left(t^4\right)
$$
and
\begin{equation}\label{eq:Delta3152}
\Delta_L=-\frac{\beta_3^2+4\sigma\kappa}{4} t^4 +O\left(t^5\right)
=-(\mu_0^2+\sigma\kappa)t^4+O\left(t^5\right).
\end{equation}
Moreover, we have
$$
\tilde \nu^E=\sigma (-\mb e+\mb v)t-\frac{\sigma \beta_3}2t^2
+\frac{-\sigma\alpha_4\mb e+ \sigma \beta_4 \mb n+(-3\kappa+\sigma\alpha_4)\mb v
}6t^3+O(t^4)
$$
and
\begin{align*}
f_{ss}&=\kappa \mb n
-\frac{\sigma}{2} \left(\kappa '\mb e+\kappa^2\mb n\right)t^2
+\frac{\sigma\beta''_3}{6} t^3 \mb v +O\left(t^4\right), \\
f_{st}&=
-\sigma \kappa\mb e t+\frac{\sigma \beta'_3}{2} \mb v t^2+
\frac{\sigma}{6}   \left(-\kappa \alpha_4\mb e+\alpha'_4\mb n+\beta _4'\mb v\right)t^3
+O(t^4), \\
f_{tt}&=
\sigma  \mb n+\sigma \beta_3\mb v t+\frac{\sigma}2(\alpha_4 \mb n+\beta_4 \mb v)
+\frac{\sigma}6(\alpha_5 \mb n+\beta_5 \mb v)t^3+O(t^4).
\end{align*}
Then we have
\begin{align*}
\tilde L&=
-\frac{\sigma \kappa \beta_3}{2} t^2
+\frac{\sigma }{6}  \left(3 \sigma  \kappa'-\kappa\beta_4\right)t^3+O\left(t^4\right),\\
\tilde M&=
\kappa t^2+\frac{\beta'_3}{2} t^3 +O\left(t^4\right),\quad
\tilde N=
\frac{\beta _3}{2} t^2 +\frac{\beta _4}{3} t^3 +O\left(t^4\right).
\end{align*}
If we set
$
\tilde W^L=(w_{i,j})_{i,j=1,2}$,
then
\begin{align*}
w_{1,1}&=\frac{\sigma}{4} \left(\beta _3\beta _3'+2 \kappa '\right)t^5+O\left(t^6\right), \\
w_{1,2}&=\frac{\sigma}{4} \left(\beta _3^2+4 \sigma\kappa\right)t^4
+\frac{\sigma}{4}\left(2 \sigma  \beta'_3+\beta_3 \beta_4\right)t^5+O\left(t^6\right),\\
w_{2,1}&=
\frac{- \kappa}{4} \left(\beta_3^2+4\sigma \kappa\right)t^4
+
\frac{1}{12}\Big(3 \sigma \beta_3\kappa'-2 \kappa \left(5 \sigma \beta _3'+\beta _3 \beta _4\right)
\Big)t^5+O\left(t^6\right),\\
w_{2,2}&=\frac{\sigma}{12} \left(\beta _3\beta _3'-2\beta _4\kappa\right)t^5 +O\left(t^6\right)
\end{align*}
hold. Using the fact that $\beta_3=2\mu_0$ and $\beta_4=3\mu_1$ 
(cf.~\eqref{eq:ab1519b}), we have
\begin{align}\label{eq:3200}
\det(\tilde W^L)&= 
\sigma \kappa(\mu_0^2+\sigma  \kappa)^2 t^8+O(t^9), \\
\label{eq:3203}
\op{trace}(\tilde W^L)&=
\frac{\sigma}{6} \left(-3\kappa\mu_1 +8 \mu_0 \mu'_0+3 \sigma  \kappa'\right)t^5+O(t^6).
\end{align}
In particular, we have
\begin{align*}\label{eq:3301}
(\op{trace}\,\tilde W^L)^2-4\det(\tilde W^L)
&=-4\det(\tilde W^L)t^8+O(t^9) \\
\nonumber
&= -4\sigma \kappa\left(\mu_0^2+\sigma\kappa\right)^2 t^8+O(t^9).
\end{align*}

\begin{proof}[Proof of Proposition F]
Let $f:\mc U_f\to \Lo^3$ be a cuspidal edge along a curve 
$\Gamma$ of type $L$.
We fix a singular point $p=(s,0)$ ($s\in I$)
arbitrarily and assume $i_p>2$.
Then, by Lemma \ref{lem:2775},
we have $i_p\ge 4$.
By \eqref{eq:Delta3152}, $i_p>4$ if and only if
$\beta_3(s)^2+4\sigma\kappa(s)=0$.
Without loss of generality, we may assume  $\kappa(s)\ge 0$
(if necessary, we can reverse the orientation of the curve).
Since $p$ is a cuspidal edge singular point,
$\beta_3(s)=2\mu_0(s)\ne 0$ (cf. Lemma~\ref{lem:ccrL}), 
and so $\sigma \kappa(s)<0$.
Then we have $\kappa(s)>0$ and $\sigma<0$ follows. 
Since $\kappa_s=\sigma \kappa$,
the map $f$ is of concave type at $p$
and $|\kappa(s)|=\mu_0(s)^2$ holds.
\end{proof}

\begin{proof}[Proof of Theorem G]
Since $(s,0)$ is of order four, $\Delta(s,t)$ does not change sign at $t=0$.
So (a) holds. 
By \eqref{eq:3203}, we have
$$
|H^L|=\frac1{12|\mu_0^2+\sigma \kappa|^{3/2}|t|}
\left|-3\kappa\mu_1 +8 \mu_0 \mu'_0+3 \sigma  \kappa'\right|+O(1),
$$
which implies (b).
By Lemma \ref{lem:ccrL}, $f$ has cuspidal edge singular point
at $p:=(s_0,0)$ if and only if $\mu_0(s_0)\ne 0$, proving (d).
We next consider the behavior of umbilical points of $f$:
Since $f$ is of order four, we have
$
\mu_0^2+\sigma\kappa\ne 0
$
and the facts
$$
w_{1,2}=\sigma \left(\mu_0^2+\sigma\kappa\right)t^4+O(t^5)
=-\sigma \Delta_L +O(t^5)
$$
implies that matrix $\tilde W^L$ cannot be a diagonal matrix at $p$,
and so $p$ cannot be an accumulation point of
a sequence of umbilical points $\{(s_k,t_k)\}_{k=1}^\infty$. 
If not, $t_k\ne 0$ and $\tilde W^L(s_k,t_k)$ is diagonal for each $k$, 
however it is impossible,  because 
$\mu_0^2+\sigma\kappa\ne 0$ at $s=0$ by
\eqref{eq:Delta3152} and $w_{1,2}(0,0)\ne 0$, proving (c).

We now assume that $\Gamma''\ne \mb 0$ on $I$, which implies
the curvature function $\kappa$ of the plane curve $\gamma$
does not vanish. Without loss of generality, we may assume that $\kappa>0$.
By \eqref{eq:3200}, we have
\begin{equation}\label{eq:kl3298}
K^L=
\op{sgn}(\mu_0^2+\sigma  \kappa)
\frac{\sigma \kappa}{|\mu_0^2+\sigma  \kappa|t^4}+O(t^{-3})=
\frac{\sigma \kappa}{(\mu_0^2+\sigma  \kappa)t^4}+O(t^{-3}).
\end{equation}
If $f$ is space-like, then $\mu_0^2+\sigma  \kappa$ 
is negative, so $\sigma  \kappa<0$.
In particular, $K^L$ is positive,
and
$f$ is of concave type, proving \ref{item:G1}.
On the other hand, 
if $f$ is time-like, then $\mu_0^2+\sigma  \kappa$ 
is positive, so the sign of $K^L$ coincides with
that of $\sigma \kappa$.
Since
$\mu_0^2+\sigma\kappa$ and $\kappa$ are positive, 
it happens either
$\sigma>0$ or ``$\sigma<0$ and 
$\mu_0^2>\kappa$". So we obtain the first part of \ref{item:G2}
(the second part of \ref{item:G2} is the case that $(s,0)$ 
is not a cuspidal edge singular point, which is discussed 
in the proof of Corollary \ref{cor:3995} below).
We prove \ref{item:G3}: The fact $K^L$ diverges is a
consequence of \eqref{eq:kl3298} since $\mu_0^2+\sigma  \kappa\ne 0$.
Since $|H^L|$ is of order $1/|t|$, 
by \eqref{eq:3200} and
\eqref{eq:3203},
the absolute values  $|\lambda_i|$ ($i=1,2$) of two principal curvatures
are of order $1/|t|$, and are unbounded. 
Finally, we prove \ref{item:G4}:
Since $\beta_3=2\mu_0$, \eqref{eq:Delta3152} implies
\begin{equation}\label{eq:3928}
\Delta_L=-(\mu_0^2+\sigma\kappa) t^4 +O\left(t^5\right).
\end{equation}
Since $f$ is of order four, we have
$
\mu_0^2+\sigma\kappa\ne 0
$
and so $\op{trace}(\tilde W^L)^2-4\det(\tilde W^L)$ never vanishes near $t=0$,
which implies \ref{item:G4}.
\end{proof}

\begin{Corollary}\label{Cor:I2}
The sign of the function
$
\mu_0^2+\sigma\kappa
$
and
the zero set of
the function
$-3\kappa\mu_1 +8 \mu_0 \mu'_0+3 \sigma  \kappa'$
are both independent of
the choice of the projection of $\Gamma$ to a space-like plane $\mathcal E^2$.
\end{Corollary}

\begin{proof}
In fact, the top-terms of the asymptotic expansions of 
$K^L(s,t)$ and $|H^L(s,t)|$ along the curve $\Gamma$ 
with respect to the parameter $t$ do not depend on
the choice of an admissible coordinate change (cf. 
Definition \ref{def:GE} and
Lemma \ref{lem:860}),
and replacement of the space-like plane is
obtained only by
such an admissible coordinate change.
So, we obtain the conclusion.
\end{proof}

If $f(s,t)$ does not give cuspidal edge at $(s,t)=(s_0,0)$,
the following assertion holds:

\begin{Corollary}\label{cor:3995}
Let $f$ be a generalized cuspidal edge 
of order four along a regular curve $\Gamma$
of type $L$. 
If $p:=(s_0,0)$ $(s_0\in I)$ is not a cuspidal edge singular point
and $\Gamma''(s_0)$ does not vanish $($i.e.~$\kappa(s_0)\ne 0)$, 
then  
\begin{enumerate}
\item 
$f$ is space-like $($resp. time-like$)$ near $p$ if
and only if
$f$ is of concave type 
$($resp. convex type, see 
Definition~\ref{def:CC613b}$)$,
that is, 
$\sigma \kappa(s_0)<0$ $($resp. $\sigma \kappa(s_0)>0)$ holds,
\item
the Gaussian curvature $K^L$ diverges to $\infty$
when $(s,t)$ tends to $p$,  and
\item umbilical points and quasi-umbilical points never accumulate at $p$.
\end{enumerate}
\end{Corollary}

\begin{proof}
Since
$p=(s_0,0)$ ($s_0\in I$) is not a cuspidal edge singular point,
$\mu_0(s_0)$ vanishes (cf. Lemma~\ref{lem:ccrL}). 
So, \eqref{eq:3928} is reduced to
$$
\Delta_L=-\sigma\kappa t^4 +O\left(t^5\right).
$$
This with 
\ref{item:G1} and \ref{item:G2} of Theorem~G imply (1).
On the other hand,
since  $\Gamma''(s_0)$ does not vanish, 
$\kappa(s_0)\ne 0$ and as long as $s$ is close to $s_0$, 
\eqref{eq:kl3298} can be written as
$$
K^L(s_0,t)
=\op{sgn}(\sigma  \kappa(s_0))
\frac{\sigma \kappa(s_0)}{|\sigma  \kappa(s_0)|t^4}+O(t^{-3})
=\frac{1}{t^4}+O(t^{-3}),
$$
which implies (2). 
The last assertion (4) follows from
(c) and \ref{item:G4} of Theorem~G.
\end{proof}

\begin{Remark}
This corollary is a generalization of the corresponding assertion
in Akamine \cite[(3) of Theorem A]{A} for time-like zero mean curvature surfaces,
where the assumption $\Gamma''(s_0)\ne 0$ is dropped.
However, if $f$ is a maxface or a minface, the order of $f$ at a
generalized cuspidal edge singular point is equal to four.
Then $\kappa(s_0)\ne 0$, that is, $\Gamma''(s_0)\ne 0$ holds.
\end{Remark}

\begin{Remark}\label{rmk:DQ}
By Theorem~G,  umbilical points never
accumulate at any cuspidal singular points of $f$. 
However, the quasi-umbilical points of $f$ can accumulate, in general.
A concrete example of a cuspidal edge of order four with zero  
mean curvature function whose quasi-umbilical points 
accumulate at a cuspidal edge singular point
is given in Akamine \cite[Figure~4]{A}.
\end{Remark}

We give a family  of cuspidal edges with order four
whose mean curvatures are bounded:

\begin{Example}\label{exa:4161}
We fix a real number $a$, and set
\begin{align*}
&\Gamma(s)
:=\left(\cos s,\sin s,s\right),\\  
&\mb n(s):=-(\cos s,\sin s,0) \qquad (s\in \R), \\
&
x(t):=\frac{\sigma t^2}{2},\quad
y(t):=\frac{\sigma a t^3}{3}\qquad (\sigma\in \{1,-1\},\,\,
\mb v:=(0,0,1)\,),
\end{align*}
where $|t|$ is a sufficiently small number.
Then
$
f(s,t):=\Gamma(s)+x(t) \mb n(s)+y(t)\mb v
$
gives a light-like cuspidal edge along $\Gamma$,
which satisfies
\begin{equation}\label{eq:3951}
E^L=\frac{t^2}{4} \left(t^2-4 \sigma\right), \quad F^L=-a \sigma  t^2,\quad
G^L=t^2-a^2 t^4
\end{equation}
and
\begin{equation}\label{eq:3956}
\Delta_L=
t^4 \left(-a^2-\sigma \right)+t^6 \left(a^2 \sigma +\frac{1}{4}\right)-\frac{a^2 t^8}{4}.
\end{equation}
Moreover, we have (cf .\eqref{tL}, \eqref{tM} and \eqref{tN})
\begin{equation}\label{eq:3960}
\tilde L=-\frac{a \sigma  t^2 \left(\sigma  t^2-2\right)^2}{4},\qquad
\tilde M=t^2,\qquad
\tilde N=-\frac{a t^2 \left(\sigma  t^2-2\right)}{2}.
\end{equation}
In particular, $f$ is of order four if $a^2+\sigma\ne 0$.
By \eqref{eq:K1Lb}, 
\eqref{eq:3951}, \eqref{eq:3956} and \eqref{eq:3960},
the mean curvature function $H^L$ and
the Gaussian curvature are given by
\begin{align*}
K^L(s,t)&=
\frac{t^4}{8\Delta_L^2} \Big(
8\left(a^2 \sigma +1\right)-12 a^2 t^2+6 a^2 \sigma  t^4-a^2 t^6\Big), \\
H^L(s,t)&=\pm 
\frac{at^6}{8|\Delta_L|^{3/2}}\Big(
\left(8 a^2+14 \sigma \right)-t^2 \left(8 a^2 \sigma +3 \right)+2 a^2 t^4\Big).
\end{align*}
Consequently, the Gaussian curvatrue is unbounded and
the mean curvature function is bounded. Moreover,
if $\sigma=1$, then $f$ is of convex type and time-like at $(s,0)$ satisfying $K^L>0$
(cf. Figure~\ref{ex:L123}, left).
If $\sigma=-1$ and $a^2+\sigma$ is positive
(cf. Figure~\ref{ex:L123}, center),
then $f$ around $t=0$ is of concave type and time-like satisfying $K^L<0$.
If $a^2+\sigma$ is negative, $f$ is space-like,
and $\sigma=-1$, that is,
$f$ is of concave type (cf. Figure~\ref{ex:L123}, right).
In this case, the Gaussian curvature of $f$ is positive.
\end{Example}

By the above discussions, we can observe that
cuspidal edges along a light-like curve
cannot change their causal type if their order is less than 
or equal to four. However, if the order is greater than four,
such a cuspidal edge exists:

\begin{Example}\label{ex:34239}
We fix a  real number $\beta$, and set
$$
\Gamma(s):=(\cos s,\sin s,s),\quad
\mb n(s):=-(\cos s,\sin s,0)\qquad (s\in \R).
$$ 
Consider the cuspidal edge given by
$
f(s,t):=\Gamma(s)-x(t) \mb n(s)-y(t)\mb v,
$
where 
$$
x(t):=\frac{t^2}{2},\qquad
y(t):=\frac{\beta t^3}{6}-\frac{t^4}8,\qquad
\mb v=(0,0,1)
$$
for sufficiently small $|t|(>0)$.
The coefficients of the first fundamental form of $f$ are
computed as
\begin{align}\label{eq:4014}
&E^L=\frac{1}{4} t^2 \left(t^2+4\right),\quad
F^L=\frac{1}{2} t^2 (\beta-t), \\
\nonumber
&G^L=\frac{1}{4} t^2 \left(4-\beta^2t^2-2\beta t^3+t^4\right).
\end{align}
Since
$$
\Delta_L=(1-\frac{\beta^2}4)t^4+\frac{\beta}{2} t^5+O(t^6),
$$
by setting  $\beta=2$, $f$ gives a cuspidal edge of order 
$5$ at each point of $\Gamma$. This $f$ changes its causal type along $\Gamma$.
Moreover, we have
\begin{equation}\label{eq:4032}
\tilde L=\frac{-t^2}8 (t+2)(t^2+2)^2,\quad
\tilde M=t^2,\quad
\tilde N=\frac{-t^2}2 (t+1)(t^2+2).
\end{equation}
Using\eqref{eq:K1Lb},
\eqref{eq:4014} and \eqref{eq:4032},
we obtain
\begin{align*}
K^L&=
-\frac{16 (-24 + 32 t- 36 t^2 + 24 t^3 - 18 t^4 + 8 t^5 - 3 t^6 + 
   t^7)}{t^5 (-16 + 16 t - 16 t^2 + 8 t^3 - 4 t^4 + t^5)^2},\\
H^L&=
\pm \frac{-16 + 24 t + 8 t^2 - 44 t^3 + 44 t^4 - 32 t^5 + 16 t^6 - 6 t^7 + t^8}
{2|t|^{5/2}|-16 + 16 t - 16 t^2 + 8 t^3 - 4 t^4 + t^5|^{3/2}}.
\end{align*}
In particular, $K^L$ and $H^L$ are both unbounded near 
$(s,t)=(s,0)$. 
\end{Example}

In a joint work with Yamada, 
the fourth author \cite{UY}
introduced the concept of \lq\lq maxfaces" as 
a canonical class of space-like zero mean 
curvature surfaces with admissible 
singular points. Similarly, Takahashi 
\cite{Ta} and Akamine \cite{A} 
defined \lq\lq minfaces'' as a class of time-like 
zero mean curvature surfaces with 
admissible singular points.
In addition, Akamine \cite{A} showed
essentially the same assertion as in Theorem~G for
cuspidal edges on maxfaces and minfaces.
The following assertion implies 
Theorem~G is a generalization of it.

\begin{Proposition}\label{prop:O4}
Generalized cuspidal edges on maxfaces or on minfaces are 
light-like cuspidal edges of order four.
\end{Proposition}

\begin{proof}
Let $f$ be a maxface (resp. a minface).
As shown in the proof of \cite[Theorem 2.4]{FSUY}
(resp. the two 
line before the equation (8) in \cite{A}),
the (Euclidean) area density function $\lambda^E$
of a maxface (resp. a minface) 
is
a positive function multiple of 
\begin{equation}\label{eq:4181}
1-|g|^2 \qquad (\text{resp. } 1-g_1g_2),
\end{equation}
where $g$ (resp. $g_1$ and $g_2$)
is one of a pair of Weierstrass data
of $f$, which corresponds (resp. correspond)
to the Gauss map of $f$. 
We set $\Delta_L:=E^LG^L-(F^L)^2$.
If $f$ is a maxface, its first fundamental form is
given by (cf. \cite[(2.1)]{FSUY})
$ 
ds^2=(1-|g|^2)^2|\omega|^2.
$
Thus 
$\Delta_L$ 
is
a positive function multiple of 
$(1-|g|^2)^4$.
On the other hand, if $f$ is a minface,
it holds that (cf. \cite[Proof of Fact~4.2]{A} where
$f_u\times_E f_v$ is given)
$$
f_u\times_E f_v=\frac{\hat \omega_1\hat \omega_2}2
(1-g_1g_2)(g_1-g_2,-1-g_1g_2,-g_1-g_2),
$$
where we remark that
the signature of $\Lo^3$ in \cite{A} is $(-++)$, which is
different from the one in this paper.
So, we have
$$
f_u\times_L f_v=\frac{\hat \omega_1\hat \omega_2}2
(1-g_1g_2)(g_1-g_2,-1-g_1g_2,g_1+g_2),
$$
which implies 
\begin{align*}
\Delta_L&=\inner{f_u\times_L f_v}{f_u\times_L f_v} \\
&=\frac{\hat \omega_1^2\hat \omega_2^2}4(1-g_1g_2)^2\Big(
(g_1-g_2)^2+(1+g_1g_2)^2-(g_1+g_2)^2
\Big) \\
&=
\frac{\hat \omega_1^2\hat \omega_2^2}4(1-g_1g_2)^4.
\end{align*}
So
$\Delta_L$ 
is
a positive function multiple of 
$(1-g_1g_2)^4$.
Regarding \eqref{eq:4181}, in both of two cases,
$\Delta_L$ is  a positive function multiple of $(\lambda^E)^4$.
On the other hand, if $(s,t)$ is an admissible coordinate system of $f$, we have 
$(\lambda^E)_t(p)\ne 0$, since $f$ is a cuspidal edge singular point.
So we obtain the conclusion.
\end{proof}

\appendix
\section{A representation formula for cusps}

\begin{Definition}\label{def:GC1794}
Let $\mb c:I\to \R^2$ be a
$C^\infty$-map defined on an open interval $I$ in $\R$ 
containing the origin $0\in \R$.
We say that $\mb c(t)$ has a {\it generalized cusp} at $t=0$
if $\mb c$ satisfies $\mb c'(0)=\mb 0$ and $\mb c''(0)
\ne \mb 0$. We remark that
$\mb c''(0)$ is determined up to a positive constant if
we change the parameter of $\mb c$. Moreover,
$\mb c''(0)$ points in the {\it cuspidal direction},
that is, it points in the direction 
where the image of $\mb c$ exists.
We say that $\mb c(t)$ has a {\it cusp} at $t=0$
if $\mb c''(0)$ and $\mb c'''(0)$ are linearly independent.
\end{Definition}

We set
$$
\tau^E:=\frac{\det(\mb c''(0),\mb c'''(0))}{|\mb c''(0)|^{5/2}_E},
$$
which is called the {\it cuspidal curvature} of $\mb c$ at $t=0$,
where
$$
|\mb a|_E:=\sqrt{\mb a\cdot \mb a}\qquad (\mb a \in \E^2).
$$
The (Euclidean) curvature $\kappa^E(t)$  of $\mb c$  at $t\ne 0$ 
is given by
$$
\kappa^E(t)=\frac{\det(\mb c'(0),\mb c''(0))}{|\mb c'(0)|^3_E}.
$$
The function 
$$
s^E(t):=\int_0^t |\mb c'(0)|_E dt
$$
gives the signed-arc-length of $\mb c(t)$ from $t=0$.
As in \cite{SU},  it holds that
\begin{equation}
\tau^E=2 \sqrt{2} \lim_{t\to 0}\kappa^E(t) \sqrt{|s^E(t)|}.
\end{equation}
Moreover, the following formula is known:

\begin{Fact}[Shiba-Umehara \cite{SU}]
\label{fact:SU}
Let $\mu(t)$ be a $C^\infty$-function defined on an open interval $I$
containing $0\in \R$. If we set
\begin{equation}\label{eq:SU2}
\mb c(t):=\int_0^t u(\cos \lambda(u),\sin \lambda(u))du,
\qquad \lambda(t):=\int_0^t \mu(u)du,
\end{equation}
then $\mb c$ is a generalized cusp at $t=0$ in $\R^2$.
In particular, $t=0$ is a cusp if and only if
$\mu(0)\ne 0$. 
Conversely, any map germs of a generalized cusp in $\R^2$  
can be obtained by \eqref{eq:SU2}
up to motions in $\E^2$.
Moreover, the coefficients $\{\mu_i\}_{i=0}^\infty$ of the
Taylor expansion of $\mu$ satisfying
$$
\mu(t)\approx \mu_0+\mu_1 t+\frac{\mu_2}{2!}t^2
+\frac{\mu_3}{3!}t^3+\cdots
$$
gives the series of invariants for generalized cusps in $\E^2$. 
In particular, $2\mu_0$ coincides with $\tau_E$.
\end{Fact} 

This formula is slightly different from the original one in \cite{SU} 
but is essentially 
the same,
and the proof of the formula \eqref{eq:SU2} is 
given in the appendix of \cite{HNSUY} and in the book \cite{SUY2}.

\begin{Remark}
Let $\kappa^E(t)$ be a $C^\infty$-function defined on an open interval $I$
containing $0\in \R$. If we set
\begin{equation}\label{eq:SU3}
\mb c(t):=\int_0^t (\cos \lambda(u),\sin \lambda(u))du,
\qquad \theta(t):=\int_0^t \kappa^E(u)du,
\end{equation}
then $\mb c$ is a regular curve with arc-length parameter at $t=0$ in $\E^2$
whose curvature function is $\kappa^E$.
Conversely, a regular curve with arc-length parameter  
can be obtained by this formula up to motions in $\E^2$.
So 
Fact \ref{fact:SU}
can be considered as an analogue of this classical formula in $\E^2$.
The parameter $t$ of $\mb c$ in the formula
\eqref{eq:SU3}
is called a
{\it normalized half-arc-length parameter} of $\mb c$
(see \cite[Appendix B]{SUY2} for details).
\end{Remark}

In the plane $\Lo^2$, 
we can construct a similar representation formula as follows:
We denote by $\Inner{\,}{\,}$ the Lorentzian inner product of $\Lo^2$.

\begin{Definition}
A generalized cusp of $\mb c(t)$ at $t=0$ is said to be
{\it space-like} (resp. {\it time-like})
if 
$\mb c''(0)$ is space-like, that is,
$\inner{\mb c''(0)}{\mb c''(0)}>0$
(resp. time-like, that is,
$\inner{\mb c''(0)}{\mb c''(0)}<0$).
\end{Definition}

If $\mb c$ is a space-like (resp. time-like)  cusp,
then $\mb c'(t)$ ($t\ne0$) is a space-like (resp. time-like) vector
for each sufficiently small $|t|(>0)$.
By the transformation $(x,y)\mapsto (y,x)$, the
generalized space-like cusps become generalized time-like cusps,
so we only consider here generalized space-like cusps:

We set
$$
\tau^L:=\frac{\det(\mb c''(0),\mb c'''(0))}{|\mb c''(0)|^{5/2}_L},
$$
which is called the {\it cuspidal curvature} of $\mb c$ at $t=0$,
where
$$
|\mb a|_L:=\sqrt{\inner{\mb a}{\mb a}}\qquad (\mb a \in \Lo^2).
$$
The (Lorentzian) curvature $\kappa^L(t)$  of $\mb c$  at $t\ne 0$ 
is given by
$$
\kappa^L(t)=\frac{\det(\mb c'(0),\mb c''(0))}{|\mb c'(0)|^3_L}.
$$
The function 
$$
s^L(t):=\int_0^t |\mb c'(0)|_L dt
$$
gives the signed-arc-length of $\mb c(t)$ from $t=0$ in $\Lo^2$.
By imitating the argument in \cite{SU},  it holds that
\begin{equation}\label{eq:4567}
\tau^L=2 \sqrt{2} \lim_{t\to 0}\kappa^L(t) \sqrt{|s^L(t)|}.
\end{equation}

\begin{Proposition}\label{prop:SU2}
Let $\mu(t)$ be a $C^\infty$-function defined on an open interval $I$
containing $0\in \R$. If we set
\begin{equation}\label{eq:SU4}
\mb c(t):=\int_0^t u(\cosh \lambda(u),\sinh \lambda(u))du,
\qquad \lambda(t):=\int_0^t \mu(u)du,
\end{equation}
then $\mb c$ is a space-like generalized cusp at $t=0$ in $\Lo^2$.
In this situation, $t=0$ is a cusp if and only if
$\mu(0)\ne 0$.
Conversely, any map germs of space-like generalized cusp  
can be obtained by this formula up to motions in $\Lo^2$.
In particular, the coefficients $\{\mu_i\}_{i=0}^\infty$ of the 
Taylor expansion of $\mu$ satisfying
$$
\mu(t)\approx \mu_0+\mu_1 t
+\frac{\mu_2}{2!}t^2+\frac{\mu_3}{3!}t^3+\cdots
$$
gives
the series of geometric invariants of the generalized cusps in $\Lo^2$. 
In particular, $2\mu_0$ coincides with $\tau_L$.
\end{Proposition}

\begin{proof}
We denote by $\inner{}{}$, which is
the Lorentzian inner product on $\Lo^2$.
For a vector $\mb a\in \R^2$, we set
$|\mb a|_L:=\sqrt{|\inner{\mb a}{\mb a}|}$.
Let $\mb c(t)$ be a curve given in the formula of
Proposition \ref{prop:SU2}.
If we set
$$
\mb e(t):=\pmt{\cosh \theta(t)\\ \sinh \theta(t)},
$$
then we have
$
\mb c'(t)=t \mb e(t)
$.
Thus
$$
\mb n(t):=\pmt{\sinh \theta(t)\\ \cosh \theta(t)}
$$
gives a smooth unit normal vector field along $\mb c$.
So $\mb c$ gives a generalized cusp at $t=0$.
The (Lorentzian) arc-length parameter of $\mb c$ satisfies
$$
s(t):=\int_0^t |\mb c'(u)|_Ldu=\int_0^t |u| du=\op{sgn}(t)\frac{t^2}{2}.
$$
In particular, we have 
\begin{equation}\label{eq:4413}
t=\op{sgn}(s)\sqrt{2|s|}.
\end{equation}
Then,
$
\mb c''(t)=\mb e(t)+ t \mu(t) \mb n(t)
$
and the curvature function satisfies
\begin{equation}\label{eq:4719}
\kappa^L(t)=\frac{\det(\mb c'(t),\mb c''(t))}{|\mb c'(t)|^3_L}=
\frac{t^2 \mu(t) \det(\mb e(t),\mb n(t))}{|t|^3}=\frac{\mu(t)}{|t|}.
\end{equation}
By this with \eqref{eq:4413},
we have
\begin{equation}\label{eq:4430}
\mu(t)=\kappa^L(t)\sqrt{2|s(t)|} \qquad (\text{if $t\ne 0 $}).
\end{equation}
By \eqref{eq:4567},  we have
$
\mu^L(0):={\tau^L}/2
$.
Thus, $t=0$ is a cusp if and only if
$\mu^L(0)\ne 0$.

Conversely,
let $\mb c(t)$ ($t\in I$)
be a 
space-like generalized cusp at $t=0$ in $\Lo^2$
defined on an open interval $I$ containing $0\in \R$.
Since $t=0$ is a cusp,
the arc-length 
$
s(t):=\int_0^t |\mb c'(t)|_Ldt
$
of the curve $\mb c$ is not smooth at $t=0$.
However (cf \cite[Proposition B.2.1]{SUY2})
\begin{equation}\label{eq:tau}
v(t):=\op{sgn}(t)\sqrt{|2s(t)|}
\end{equation}
is a $C^\infty$-function on $I$ satisfying $v'(t)>0$.
In fact, 
$
\mb w(t):=(1/t)\mb c'(t)
$
is a vector-valued $C^\infty$-function of $t$ even at $t=0$.
Since $\mb w(0)\ne \mb 0$,
the absolute value 
$\psi(t):=|\mb w(t)|_L$
is also $C^\infty$-differentiable. 
By \cite[(B.8) and Proposition A.4]{SUY2},
we have 
$
v(t)=t \sqrt{\Psi(t)},
$
where $\Psi$ is a $C^\infty$-function of $t$ given by
$
\Psi(t):=\int_0^1 u \psi(t u)du.
$
Since $\Psi(0)\ne 0$ (cf. \cite[(A.2)]{SUY2}), 
$v(t)$ is $C^\infty$-differentiable at $t=0$.
So we can use $v$ as a new parameter of $\mb c$ around the cusp, 
which is called the {\it normalized half-arc-length parameter} of $\mb c$.
By  L'Hopital's law, we have
\begin{align*}
v'(0)&=\lim_{t\to 0}\op{sgn}(t)\frac{\sqrt{2|s(t)|}}{t}
=\lim_{t\to 0}\left|\frac{s(t)}{t}\right| \\
&=
\left|\lim_{t\to 0}\frac{s(t)}{t^2}\right|=
\left|\lim_{t\to 0}\frac{|\mb c'(t)|_L}{2t}\right|
=
\frac{|\mb c''(0)|_L}{2}>0,
\end{align*}
and so $v$ can be taken as a new parameter of $\mb c$.
Since $\mb c(v)$ has a cusp at $v=0$,
there exists a smooth unit normal vector field
$\mb n(t)$ along $\mb c$.
We the take a smooth normal vector field 
$\mb e(v)$ along $\mb c$
such that
$(\mb e(v),\mb n(v))$ gives an orthonormal frame field
satisfying $\det(\mb e(v),\mb n(v))=1$ for each $v\in I$.

Since $\mb c(t)$ has a cusp at $t=0$,
the direction of the unit normal vector $\mb n(t)$ changes from the 
right side of the curve $\mb c(t)$ to the left side, 
or from the left side of $\mb c(t)$ to the right side, just at t=0.
There exists a smooth unit tangent vector field
$\mb e(t)$ along $\mb c$, which change direction with respect to $\mb c'$
at $t=0$.
So, by changing $\mb n(v)$ by $-\mb n(v)$,
we may assume that $\mb e$ satisfies
\begin{equation}\label{eq:es}
\mb e(s)=
\begin{cases}
\mb c_s(s) &\text{if $s>0$},\\
-\mb c_s(s) &\text{if $s<0$},
\end{cases} 
\end{equation}
where $\mb c_s:={d\mb c}/{ds}$.
Since $v^2/2=\op{sgn}(s)s$, we have
$$
\mb c_v(v)=
\mb c_s(v) s_v(v)= (\op{sgn}(v) \mb e(v))
(\op{sgn}(v)(v^2/2)_v)=v\mb e(v).
$$
By \eqref{eq:es}, $\mb n(v)$ is the left-ward unit normal vector field of $\mb c$
if $v>0$. By the Frenet equation, we have
$$
\mb e_v(v)
=\mb e_s(v) s_v(v)=
 \Big(\op{sgn}(v)\kappa(v) \mb n(v)\Big)\Big(\op{sgn}(v)v^2/2\Big)_v 
=\kappa(v) v \mb n(v).
$$
We set $\mu(v):=\kappa(v) v$,
which can be extended as a smooth function of $v$.
By definition, 
$\mb e_v=\mu \mb n$ holds, and
$\mu(v)$ coincides with the function defined by
\eqref{eq:4719} (by setting $t:=v$).
Since $\mb n_v$ is perpendicular to $\mb n$,
we can write $\mb n_v(v)=a(v) \mb e(v)$, and
$$
a=\mb n_v\cdot \mb e=(\mb n\cdot \mb e)_v
-\mb n\cdot \mb e_v=-\mb n\cdot \mb e_v=\mu.
$$
Consequently, we obtain the formula
 $$
\pmt{\mb c(v),\mb e(v),\mb n(v)}_v
=\pmt{\mb c(v),\mb e(v),\mb n(v)}
\pmt{
0 & 0 & 0 \\
v & 0 & \mu(v) \\
0 & \mu(v) & 0
},
$$
which can be considered as a linear ordinary differential equation
when we think $\mu(v)$ is a known function and
$\mb c(v),\mb e(v),\mb n(v)$ are unknown vector valued functions.
By replacing $v$ by $t$,
the curve given in the formula of
Proposition~\ref{prop:SU2}
gives a solution of this equation.
Then, the uniqueness of the solution with an initial value condition
implies that any generalized cusp can be represented as the formula
in Proposition~\ref{prop:SU2} up to a 
orientation preserving motion in $\Lo^2$. 
\end{proof}

\begin{Definition}\label{def:mu}
In this paper, we call
the function $\mu(t)$ appeared in the
formulas \eqref{eq:SU2} and \eqref{eq:SU4}
the {\it $\mu$-function}
associated with the germ of cusp $\mb c$.
\end{Definition}

\section{Special coordinate systems 
for generalized cuspidal edges}

In this section, we show the following:

\begin{Proposition}\label{prop:fstE}
Let $f(s,t)$ $(s\in I,\, |t|<\epsilon)$
be a generalized cuspidal edge 
along a regular curve $\Gamma:I\to \R^3$, where $I$ is an open
interval. Suppose that there exist vector fields $\mb a_1(s)$ and
$\mb a_2(s)$ along $\Gamma$ such that
$
\Gamma'(s),\,\, \mb a_1(s),\,\, \mb a_2(s)
$
are linearly independent in $\R^3$ for each $s\in I$.
Then, for each $s_0\in I$,
there exist  a local diffeomorphism germ $(u,v)\mapsto (s(u,v),t(u,v))$ 
defined on a sufficiently small neighborhood $(s_0,0)$
such that 
\begin{enumerate}
\item $s(u,0)=u$ and $t(u,0)=0$ for each $u$,
\item $f$ is written in the form
$$
f\circ \phi(u,v)=\Gamma(u)+ x(u,v)\mb a_1(u)+y(u,v)\mb a_2(u),
$$
where $x(u,v)$ and $y(u,v)$ are $C^\infty$-functions
defined on a neighborhood of $(s_0,0)$, and
\item the map $v\mapsto (x(u,v),y(u,v))$ 
is a generalized cusp as defined in Appendix~A.
\end{enumerate}
\end{Proposition}

If $\{\mb a_1(s),\,\, \mb a_2(s)\}$ spans the plane 
orthogonal to $\Gamma'(s)$ in $\E^3$, 
almost the same result was proved in \cite[Lemma 3.2]{HNSUY}.  
The above statement can be seen as a generalization of this, 
but the proof is different.

\begin{proof}
Without loss of generality, we may assume that $I$ contains 
$0$ and set $s_0:=0$.
Consider the map defined by
$$
\Phi:\R^2 \times I\ni (x,y,z)\mapsto \Gamma(z)+x\mb a_1(z)+y\mb a_2(z)\in \R^3.
$$
Since 
$
\Gamma'(0),\, \mb a_1(0),\, \mb a_2(0)
$
are linearly independent, 
$\Phi$ gives a diffeomorphism defined on a neighborhood
of the origin in $\R^3$.
Since $\Gamma(s)=f(s,0)$ holds,
$g(s,t):=\Phi^{-1}\circ f(s,t)$ satisfies $g(s,0)=(0,0,s)$ for sufficiently small $|s|$.
Then we can apply the argument given in \cite[Page 105]{SUY2}
(at which only the fact that $f$ is a generalized cuspidal edge is
needed until the final argument given in \cite[Page 106]{SUY2}).
We can take an admissible local coordinate system
(see Definition \ref{def:AD880}) 
$(u,v)\mapsto (s(u,v),t(u,v))$ such that
$s(0,0)=t(0,0)=0$ and 
$$
g(u,v)=(x(u,v),y(u,v), u), \qquad
x(u,v)=v^2\alpha(u,v),\quad y(u,v)=v^2 \beta(u,v),
$$
where $\alpha$ and $\beta$ are $C^\infty$-functions defined on
a neighborhood of the origin in $\R^2$.
So we have
$$
f(u,v)=\Phi(v^2 \alpha(u,v),v^2 \beta(u,v), u)
=\Gamma(u)+v^2 \alpha(u,v) \mb a_1(u)
+v^2 \beta(u,v)\mb a_2(u).
$$
Since $f$ is a generalized cuspidal edge,
we have
$$
\mb 0\ne f_{vv}(u,0)=2\alpha(u,0)\mb a_1(u)+2\beta(u,0)\mb a_2(u),
$$
which implies that $(x_{vv}(u,0),y_{vv}(u,0))\ne \mb 0$,
and $v\mapsto (x(u,v),y(u,v))$
gives a generalized cusp for sufficiently small $|v|$.
\end{proof}

\section{Umbilic points on a wave front in $\Lo^3$}

Consider a wave front $f:U\to \R^3$, where $U$ is a domain of $(\R^2;u,v)$.
Since cuspidal edge singular points appear on wave fronts,
this fits the setting of this paper.
If we think $\R^3=\E^3$,
umbilical points never accumulate at
cuspidal edge singular points 
of $f$ (cf. \cite[Corollary 5.3.3]{SUY2}).
As an analogue of this fact, for $\R^3=\Lo^3$,
we show the following:

\begin{Proposition}\label{thm:4087}
If $p$ is a  space-like or time-like cuspidal edge singular point,
then umbilical points of $f$ never accumulate at $p$.
\end{Proposition}

As we have mentioned at the end of Section 1,
we do not know whether umbilical points can accumulate
at a light-like cuspidal edge singular point
or not.  To prove the proposition, we first consider an immersion $f:U\to \Lo^3$
and assume that $U$ consists only of space-like points
or only of time-like points.
Then we can take a unit normal vector field
$
\nu^L:U\to \Lo^3
$
of $f$.
Define two $3\times 3$ matrices by
$P_1:=(f_u,f_v,\nu^L)$ and $P_2:=(\nu_u,\nu_v,\nu^L)$.
By setting,
$$
I:=\pmt{
E^L & F^L \\
F^L & G^L
},
\qquad
I\! I:=\pmt{
L^L & M^L \\
M^L & N^L
},
$$
$
W:=I^{-1}I\! I
$
coincides with the matrix given 
in \eqref{eq:Wf}.
Then, by \eqref{eq:L-inn},
we have that
\begin{align*}
-(P_1^{-1})^T P_2
&=-(P_1^T E_3 P_1)^{-1}(P_1^T E_3 P_2) \\
&=
\pmt{I & 0 \\
       0 & \pm 1}^{-1}
\pmt{I\!I & 0 \\
       0 & \pm 1} 
=
\pmt{I^{-1}I\!I & 0 \\
        0      & \pm 1}
=
\pmt{W & 0 \\
        0      & \pm 1}.
\end{align*}
So we obtain the following Lorentzian version of the
Weingarten formula
\begin{equation}\label{eq:WL4156}
(\nu^L_u,\nu^L_v)=-(f_u,f_v)W.
\end{equation}

\begin{Lemma}\label{lem:4158}
Let $f:U\to \Lo^3$ be a wave front, and $p\in U$ its singular point
at which $f$ is space-like or time-like.
Consider the family of parallel surfaces of $f$ given by
$f^t:=f+t\nu^L$ $(t\in \R)$,
where $\nu^L$ is the unit normal vector field of $f$.
Then for sufficiently small $t\ne 0$, $p$ is a regular point of $f^t$.
\end{Lemma}

\begin{proof}
The proof is parallel to the case of $\R^3=\E^3$, see
\cite[page 65]{SUY2}.
\end{proof}

\begin{Proposition}
Let $f:U\to \Lo^3$ be a wave front and $p\in U$ a regular point
which is space-like or time-like.
Then the parallel surface $f^t$ $(t\ne 0)$
has a singular point at $p$
if and only if $1/t$ coincides with the principal curvature of $f$ at $p$.
\end{Proposition}

\begin{proof}
Since $\nu^L$ is a unit vector field,
$(\nu^L)_u$ and $(\nu^L)_v$ are perpendicular to
$\nu^L$. So $\nu^L$ is a common 
unit normal vector field of $f^t$ ($t\in \R$).
By \eqref{eq:WL4156}, we have
$$
(f^t)_u=f_u+ t \nu^L_u=f_u (I+t W),\qquad
(f^t)_v=f_v+ t \nu^L_v=f_v (I+t W)
$$
and 
\begin{equation}\label{eq:ftuftv}
\Big((f^t)_u,(f^t)_v\Big)=(f_u,f_v)(I+t W).
\end{equation}
Thus, $p$ is a singular point of $f^t$ if and only if
$\det(I+t W)=0$, that is, $1/t$ is an eigenvalue of the matrix $W$,
proving the assertion.
\end{proof}

\begin{Corollary}
If $p$ is an umbilical point of $f$, then
it is also an umbilical point of $f^t$ unless $1/t$ coincides 
with the principal curvature of $f$.
\end{Corollary}

\begin{proof}[Proof of Proposition \ref{thm:4087}]
Suppose that there exists a sequence of umbilics 
$\{p_k\}_{k=1}^\infty$ of $f$ on $U$ converging to 
$p$. By Lemma \ref{lem:4158},
we can choose $t\ne 0$  so that $p$ is a regular point of
$f^t$. Then there exists a positive integer $N$
so that each $p_k$ ($k\ge N$)
is an umbilical point of $f^t$.
Since $p$ is the limit of $\{p_k\}_{k=1}^\infty$, 
$p$ itself is an umbilical point of $f^t$.
Since $f=(f^t)^{-t}$,
the value $-1/t$ must coincides with one of the two principal curvature of $f^t$ at $p$.
Moreover, by \eqref{eq:ftuftv}, 
$(f_u,f_v)=((f^{t})^{-t}_u,(f^{t})^{-t}_v)$ 
vanishes at $p$, which contradicts the fact that 
the Jacobi matrix is of rank one at $p$.
\end{proof}

If $f$ is a generalized cuspidal edge, then
umbilical points may accumulate at a singular point: 

\begin{Exa}
We set $\Gamma(u):=(u,u^2,u^4)$, and consider a map
$
f(u,v)=\Gamma(u)+v\Gamma'(u)\,\, (u,v\in \R),
$
which is the map given in
\cite[Example 1.13]{FSUY}.
As pointed out in \cite{FSUY}, $o:=(0,0)$ is a cuspidal cross cap 
singular point of $f$.
In $\Lo^3$, the origin $o$ is a space-like singular point,
and the second fundamental form vanishes along
the $v$-axis. Since $(0,v)$ ($v\ne 0$) are
regular points of $f$, they are umbilics of $f$
which accumulate at the singular point $o$.
\end{Exa}

\begin{acknowledgements}
The authors thank the reviewer,
Shintaro Akamine and
Atsufumi Honda for fruitful discussions and
Wayne Rossman for valuable comments.
\end{acknowledgements}

\end{document}